\documentclass[a4paper,11pt,reqno]{article}
\usepackage[body=17cm,top=3cm,bottom=4cm]{geometry}
\usepackage[english]{babel}
 \usepackage{amsrefs}
\usepackage{amsmath,amsfonts,amsthm,amscd,amssymb,mathrsfs,graphicx,color,dsfont,xcolor}
\usepackage[colorlinks=true]{hyperref}
\usepackage{lipsum}
\usepackage{tikz}
\usepackage{geometry}
\usetikzlibrary{arrows,chains,matrix,positioning,scopes}
\usepackage{graphics}

\usepackage{cases}
\usepackage{bbm}
\usepackage{graphicx}
\usepackage{subfig}
\usepackage{float}
\usepackage{mathrsfs}
\usepackage{enumerate}
\usepackage{appendix}
\usepackage{times}
\usepackage{microtype}
\usepackage{mathrsfs}
\usepackage{tikz}
\usepackage{geometry}

\newcommand{\eps}{\varepsilon}

\makeatletter
\newcommand{\eqcolon}{\mathrel{\mathord{=}\raise.2\p@\hbox{:}}}
\newcommand{\coloneq}{\mathrel{\raise.2\p@\hbox{:}\mathord{=}}}
\makeatother
\newcommand{\der}{\delta}
\newcommand{\dd}{\mathrm{d}}

\newcommand{\C}{\mathbb C}
\newcommand{\T}{\mathbb T}
\newcommand{\E}{\mathbb E}

\newcommand{\Z}{\mathbb Z}
\newcommand{\R}{\mathbb R}
\newcommand{\N}{\mathbb N}
\newcommand{\PP}{\mathbb{P}}

\newcommand{\cF}{\mathscr{F}}

\newtheorem{theorem}{Theorem}[section]

\newtheorem{corollary}[theorem]{Corollary}

\newtheorem{definition}[theorem]{Definition}

\newtheorem{lemma}[theorem]{Lemma}

\newtheorem{Proposition}[theorem]{Proposition}

\newtheorem{remark}[theorem]{Remark}

\begin{document}
\title{The continuous Anderson hamiltonian in dimension two}
\author{Romain Allez, Khalil Chouk\\WIAS Berlin, H.U Berlin\\romain.allez@gmail.com, khalil.chouk@gmail.com}
\maketitle

\begin{abstract}
We define the Anderson hamiltonian on the two dimensional torus $\R^2/\Z^2$. This operator is formally defined as 
$\mathscr H:= -\Delta + \xi$ where $\Delta$ is the Laplacian operator and  where $\xi$ belongs to a general class of 
singular potential which includes the Gaussian white noise distribution. 
We use the notion of paracontrolled distribution as introduced by Gubinelli, Imkeller and Perkowski in \cite{gip}. 
We are able to define the Schr\"odinger operator $\mathscr H$ as an unbounded self-adjoint operator on $L^2(\T^2)$ and we prove that its real spectrum is discrete 
with no accumulation points for a general class of singular potential $\xi$. We also establish that the spectrum is a continuous function 
of a sort of enhancement $\Xi(\xi)$ of the potential $\xi$. 
As an application, we prove that a correctly renormalized smooth approximations $\mathscr H_\eps:= -\Delta + \xi_\eps+c_\eps$ 
(where $\xi_\eps$ is a smooth mollification of the Gaussian white noise $\xi$ and $c_\eps$ an explicit diverging renormalization constant)
converge in the sense of the resolvent towards the singular operator $\mathscr H$. 
In the case of a Gaussian white noise $\xi$, we obtain exponential tail bounds for the minimal eigenvalue 
(sometimes called ground state) of the operator $\mathscr H$ as well as its order of magnitude $\log L$ 
when the operator is considered on a large box $\T_L:= \R^2/(L\Z^2)$ with $L\to \infty$. 
\end{abstract}

\tableofcontents
\section{Introduction and main results}

The aim of this paper is to define and study the spectral statistics of the so called Anderson hamiltonian
which is a random linear operator on the torus $\mathbb T^d :=    \mathbb R^d/ \mathbb Z^d$,
formally defined as 
\begin{align}\label{def-H-xi}
\mathscr{H}=-\Delta+\xi
\end{align}
where $\Delta$ is the Laplacian operator with periodic boundary conditions and $\xi$ a real  
white noise distribution on $\mathbb T^d$, i.e. a centered Gaussian random field with covariance function given by 
\begin{align*}
\mathbb E[\xi(x)\xi(y)]=\der(x-y)
\end{align*}
where $\der$ is the Dirac delta distribution. We are interested in 
the two dimensional case $d=2$ for which the operator $\mathscr H$ 
is ill-defined, as we shall see. The Anderson hamiltonian has been defined for $d=1$ in ~\cite{fukushima} and we briefly review
the main spectral properties in this case (see section \ref{dimension-one}) to be compared with the properties we establish for 
the two dimensional case (see below).  

The Parabolic Anderson model (PAM) is at the heart of an active research area both in mathematics and theoretical physics. 
This model refers to the (linear) Cauchy problem  
\begin{align}\label{pam}
\partial_t u(t,x)-\Delta u(t,x)=u(t,x)\xi(x),\quad u(0,x)=u_{0}(x)
\end{align}
for $x\in \T^2$ and where the function $u_0\in L^2 (\mathbb T^2)$. 
The PAM has connections with questions on random motions in random potential, directed polymers, trapping of random paths, 
branching processes in random medium, Anderson localization, etc. Of course, the solution $u$ 
to the Parabolic Anderson equation may be written for $x\in \T^2$ as a function of the operator $\mathscr H$ as 
\begin{align}\label{spectral_decomposition}
u(t,x) = \exp(-t\mathscr H) u_0(x) := \sum_{n=0}^{+\infty} \exp(-t \Lambda_n) \langle e_n, u_0\rangle_{L^2(\T^2)} e_n(x)\,, 
\end{align}
provided one is able to define the operator $\mathscr H$ and prove that it 
has a discrete real spectrum $(\Lambda_n)\in \R^\N$ and associated orthonormal eigenvectors $(e_n)_{n\in \N}$
in $L^2(\T^2)$. This fact is not trivial at all in dimension two 
where the operator $\mathscr H$ is ill-defined due to the irregularity of the white noise distribution $\xi$.

The equation \eqref{pam} and the associated operator $\mathscr{H}$ have been investigated in 
many papers before for general dimension $d$ in a discrete setting. 
In this case, the Laplacian is discrete on a grid (for example $\Z^2$) with a 
fixed mesh size and the white noise $\xi$ is a sequence
of independent and identically distributed (i.i.d.) random variables indexed by $\Z^2$.  
The operator $\mathscr{H}$ is first defined on a finite box of $\Z^2$  
with some boundary conditions (either periodic, Dirichlet or Neumann). The main challenge is then to consider 
the case of a large volume and establish the limiting properties of the model when the box size tends to infinity. 
In our continuous setting, a similar situation holds. We work in a finite volume where the Anderson 
hamiltonian is restricted to the two dimensional torus (our results may easily be extended to other boundary conditions such as 
Dirichlet or Neumman). Our main results provide the construction of the Anderson hamiltonian and we also establish that 
this operator displays a discrete real spectrum with an orthonormal family of eigenvectors in $L^2(\T^2)$. We are also able 
to give partial results on the limiting statistics of its ground state (i.e. the minimal eigenvalue) of the Anderson hamiltonian 
in the limit of a large volume, considering the growing family of torus $\T_L^2:= \R^2/(L^{-1} \Z^2)$ for $L\to +\infty$.   

One should not confuse our continuous setting with a finite volume with the infinite volume case of the discrete setting. 
In particular, an important conjecture is that the limiting spectrum (when the volume tends to infinity) 
 in the discrete setting contains only isolated (pure) points in $\R$ 
(counting multiplicity) associated to localized eigenvectors. 
Equivalently, the solution of the parabolic Anderson equation \eqref{pam} on $\mathbb R^2$ is expected to be 
intermittent i.e. with a support that is localized on a few isolated islands that are far apart from each other,
for large time (see~\cite{kg,kg1,HTN,Hu}
or the forthcoming book of W. K\"onig~\cite{wk} for a state of the art review on the discrete setting).  
%
%
In our continuous setting (the mesh size is $0$), we do prove that 
the spectrum is discrete when the phase space has a finite volume but this situation does not correspond to the previous one
(discrete setting with infinite volume for the phase space). 
Regarding the limit of infinite volume in our continuous setting, we conjecture that  
the spectrum becomes continuous (with a limiting density) when the volume tends to infinity, 
in contrast with the conjecture made 
in the discrete setting for the infinite volume with a finite mesh size.

As mentioned before, the main difficulty lies in handling the singularity of the white noise distribution $\xi$. 
If $\xi_\eps$ is a smooth function,  
the definition of the operator $\mathscr{H}_\eps:=-\Delta+ \xi_\eps$ is elementary and it is 
classical to prove using the spectral theory for operators with compact resolvent that it is a self-adjoint 
operator with a discrete spectrum $(\Lambda_n^\eps)_{n\in \N}$ 
and orthonormal eigenfunctions $(e_n^\eps)_{n \in \N}$.    

With a rough potential $\xi$ such as the two-dimensional white noise, the situation is much more delicate
as one has to make sense of $\mathscr{H}f:=-\Delta f + \xi f$ 
for a sufficiently large class of functions $f$ 
to contain the eigenvectors of the operator $\mathscr{H}$. 
The eigenfunctions of $\mathscr{H}$ ought to have H\"older regularity $1^-$ (they barely fail to be differentiable)
and for such functions, the product $\xi f$ is not well defined. 
This is the classical problem which motivated It\^o's theory of stochastic integrals. 
Powerful tools to make sense of such products have recently been provided by 
paracontrolled distributions introduced in \cite{gip} by Gubinelli, Imkeller and Perkowsky. Our approach relies on 
their results.  
The theory of regularity structures developed by M. Hairer in \cite{hairer} in 2013 could also have 
been used for our purpose.

Our construction of the operator involves two steps: in the first one, which is purely analytic,
we construct the Schr\"odinger operator $\mathscr H$ for a general class of rough potentials $\xi$ living in a space of  H\"older distributions. 
The important point in this deterministic construction is that the knowledge of the distribution $\xi$ is actually not sufficient 
to define the operator $\mathscr H$ (as an unbounded operator of $L^2(\T^2)$). We will in fact need another piece of information $\Xi_2$ 
which is (roughly speaking) the ill-defined part of the product $\xi(1-\Delta)^{-1}\xi$ (this is explained in more details below). 
The operator $\mathscr{H}$ is then defined on an explicit domain $\mathscr{D}_\Xi \subset L^2(\mathbb T^2, \R)$ of 
functions $f:\T^2\to \R$ which depends on an enhancement $\Xi:=(\xi,\Xi_2)$ of the rough distribution $\xi$ containing the additional information 
$\Xi_2$, necessary to make sense of the ill-defined product $\xi f$ between two distributions. 
In the second part of our work (see section \ref{sec:Stoch}), 
we show that the Gaussian white noise fits 
in the analytic framework developed in section \ref{construct-order-1}. More precisely, we prove that one can construct $\Xi_2$ in a robust way via 
smooth approximation techniques. Note that this part is somehow purely stochastic and relies on classical stochastic analysis techniques. 
 
We first give the results obtained in the deterministic setting where $\xi$ is a general rough distribution living in a Sobolev space with index $\alpha <-1$ defined as
\begin{align*}
H^\alpha(\mathbb T_L^2,\R) := \{f \in \mathscr {S}' ( \mathbb T_L^2,\mathbb R): \sum_{k\in \Z_L^2} (1+|k|^2)^\alpha \, |\hat{f}(k)|^2 < +\infty\}
\end{align*}   
where $\mathbb T_L^2:=\R/ (L^{-1}\mathbb Z^2)$, $\mathscr{S}'(\T_L^2,\R)$ is the Schwartz space of tempered distributions and
where for $k \in \mathbb Z_L^2$, $\hat{f}(k)$ denotes the $k$-th Fourier coefficient of the distribution $f$,
\begin{align*}
\hat{f}(k):= \langle f,L^{-1} \exp(-i 2\pi  \langle k,\cdot\rangle\rangle =\frac{1}{L}
\int\limits_{\mathbb T_L^2}\exp(-i 2\pi  \langle k,x\rangle)   \, f(x)  \, dx \,.
\end{align*}
The Fourier transform will sometimes be denoted $\mathscr {F}$ so that $ \mathscr{F} f (k) = \hat{f}(k)$ for $k\in \Z_L^2$ and 
$f\in \mathscr{S}'$. Let also denote by $\mathscr C^{\alpha}$ the H\"older-Besov space (see below in section \ref{sec:bony} for a reminder of the definitions of those spaces).

The following Theorem describes the results obtained in the first analytical part of our work. Because we are interested 
in the limiting spectral properties of the operator $\mathscr H$ when considered on a Torus or domain with large volume, we will 
enunciate this Theorem for the Torus $\T_L^2:= \R^2/(L^{-1} \Z^2)$ of size $L$. We note that the bounds we obtain 
are uniform in $L$. This property shall be useful later on, for the asymptotic study of the spectrum of $\mathscr H$ in 
the limit of large volume $L\to \infty$.   

\begin{theorem}
\label{main}
Let $\alpha\in (-\frac{4}{3},-1)$. Then, there exists a Banach space 
$\mathscr X^\alpha(\mathbb T^2_L)\subset \mathscr C^{\alpha}(\T_L^2)\times\mathscr C^{2\alpha+2}(\T_L^2)$ 
such that for all $\Xi=(\xi,\Xi_2)\in\mathscr X^\alpha$, there exists a Hilbert space $\mathscr D_{\Xi}\subset L^2(\T_L^2)$  
(which is dense in $L^2(\T_L^2)$) and a unique self-adjoint operator $\mathscr H(\Xi):\mathscr D_{\Xi}\to L^2(\T_L^2)$ with the following 
properties:
\begin{enumerate}
\item If $\xi$ is a smooth function, then we can choose $\Xi_2$  such that: 
$$
\mathscr D_{(\xi,\Xi_2+c)}=H^2(\mathbb T_L^2),\quad\mathscr H(\Xi)f=-\Delta f+f(\xi+c)
$$
for all $f\in H^2(\mathbb T_L^2)$ and $c\in\mathbb R$.
\item The spectrum $(\Lambda_n(\Xi))_{n\in\mathbb N^{*}}$  of $\mathscr H(\Xi)$ is real, 
discrete without any accumulation point and satisfy $\Lambda_n(\Xi)\to+\infty$ when $n\to\infty$, 
$$
\Lambda_1(\Xi)\quad \leq \quad  \Lambda_2(\Xi) \quad  \leq \quad  \cdots\quad  \leq\quad \Lambda_n(\Xi)
$$
and $dim(\Lambda_n(\Xi)-\mathscr H(\Xi))<+\infty$. Moreover, $L^2(\T_L^2)=\bigoplus_{n}\text{ker}(\Lambda_n(\Xi)-\mathscr H(\Xi))$.
\item The eigenvalues $(\Lambda_n)_{n\in \N}$ are solution of 
a min-max principle (see Lemma~\ref{lemma:min-max} for a more precise statement).
\item For each $n\in\mathbb N$, the map $\Xi\to \Lambda_n(\Xi)$ is locally-Lipschitz. More precisely, there exists two positive constants 
$C$ and $M$ which do not depend on $L$ such that,
for all $ \alpha \in (-4/3,-1)$, $\gamma <\alpha+2$, $n\in \N $, $\Xi, \tilde \Xi \in \mathscr X^\alpha$, 
$$
|\Lambda_n(\Xi)-\Lambda_n(\tilde\Xi)|\leq Cn\left(1+n^{\frac{2\gamma-\alpha}{\alpha+2}}+(1+\Lambda_n(0))^{2\gamma}\right)^2\|\Xi-\tilde\Xi\|_{\mathscr X^\alpha}(1+\|\tilde\Xi\|_{\mathscr X^\alpha}+\|\Xi\|_{\mathscr X^\alpha})^M
$$ 
where $\Lambda_n(0)$ is the $n$-lowest eigenvalue of the 
Laplacian operator $-\Delta$. 
\item For all $a\in \mathbb R\setminus\{\Lambda_n(\Xi),n\geq 1\}$, the resolvent map $\Xi\to \mathcal G_a(\Xi)=(a+\mathscr H(\Xi))^{-1}$ is locally Lipschitz.
\end{enumerate}
\end{theorem}

Let us emphasize that the Gaussian white noise on the Torus satisfies the assumption of Theorem \ref{main} since for any $\alpha <-1$, 
we have 
$\xi\in \mathscr C^\alpha$ almost surely.  

The conclusions of Theorem \ref{main} follow from the spectral Theorem applied to the resolvent operator $\mathcal{G}_a:= (a+\mathscr{H})^{-1}$ 
which is shown to exist for $a$ sufficiently large (with a fixed point argument) and to be a compact self-adjoint operator.

We see at least two interesting applications of Theorem \ref{main}. 
The first one concerns the Parabolic Anderson model \eqref{pam} in two dimension considered on the Torus.    
The operator $\mathscr{H}$ is simply the hamiltonian associated to the linear stochastic partial differential equation \eqref{pam}
and Theorem \ref{main} leads to the spectral decomposition \eqref{spectral_decomposition} 
of the solution $u(x,t)$ of \eqref{pam} which was first constructed in \cite{gip}. 

The second application we see  is about the Schr\"odinger equation 
 \begin{align}\label{schrodinger}
\partial_t u = i \left(\Delta u - \xi u\right) \,, \quad u(x,0) = u_0(x) 
\end{align}
considered with periodic boundary conditions and where  $u_0\in L^2(\T^2,\R )$. 
The solution of this singular stochastic partial differential equation (SPDE) has not been constructed so far 
(to the best of our knowledge). 
This is due to the fact that the imaginary factor $i$ kills the Schauder's estimate 
which is usually available in the presence of a Laplacian term in a singular SPDE. 
Theorem \ref{main} provides again a spectral decomposition of the solution $u(x,t)$ to the Schr\"odinger equation \eqref{schrodinger}
\begin{align*}
u(x,t) = \sum_{n=1}^{+\infty} \exp(-i \Lambda_n t)\,  \langle e_n, u_0\rangle_{L^2} \,  e_n(x)\,. 
\end{align*}

Our construction is straightforward to extend on a more general domain 
with Dirichlet or Neumann boundary conditions (see also ~\cite{ism} where the authors study singular PDEs on general domains). 

\begin{remark}
Let us notice that the well posedness of the parabolic Anderson equation \eqref{pam} 
which was proven in \cite{gip} 
implies the second point of Theorem~\ref{main}. Indeed, as pointed out in~\cite{HL}, the global well posedness result 
of the Parabolic Anderson equation ensures that the heat kernel $\mathscr K_t u_0:=u(t,\cdot)$ is a compact self-adjoint operator of 
$L^2(\T_L^2)$ which satisfies $\mathscr K_t\mathscr K_s=\mathscr K_{t+s}$. Therefore, we can define the operator $\mathscr H$ as
$$
\mathscr H:=\frac{1}{t}\log\mathscr K_t
$$  
where the logarithm is understood in the sense of functional analysis. From a classical spectral analysis result (see~\cite{davie}), it is well known that the compactness of  $\mathscr K_t$ implies the compactness of the resolvent of  $\mathscr H$ so that the 
 the spectral theorem applies. Our work in this paper can be seen as the converse of this approach in the sense that we define the operator $\mathscr H$ on an explicit domain and we recover the well posedness result by taking $u(t,\cdot)=e^{-t\mathscr H}f$. 
 Our construction for the operator has the advantage of being more explicit. 
\end{remark}
\begin{remark}
Let us emphasize that Theorem~\ref{main} can be generalized to the $d$ dimensional Torus. 
Moreover, the reader who is familiar with the theory of regularity structure \cite{hairer}, 
may guess that the condition $\alpha>-\frac{4}{3}$ 
could be improved to $\alpha>-2$. This extension would allow one to handle even rougher potentials $\xi$ such as 
the Gaussian white noise on the three dimensional torus.      
\end{remark}

As mentioned previously, the two-dimensional Gaussian white noise $\xi$ satisfies the assumptions of 
Theorem~\ref{main}. In particular, 
the operator $\mathscr H$ has in this case a discrete real spectrum which is continuous 
with respect to the enhanced Gaussian
white noise $(\xi,\Xi_2)$. 
As explained below, Theorem~\ref{main} shall also permit one to obtain a smooth approximation result for
the operator $\mathscr H$ associated 
to the Gaussian white noise $\xi$. 
We now explain this approximation result in more details. 
If $\hat\theta_\eps=\eps^{-2}\hat\theta(\frac{\cdot}{\eps})$ is an approximation of the identity 
and $\xi_\eps=\xi\star\hat\theta_\eps$ is a mollification of the Gaussian white noise $\xi$, then the operator 
$$
\mathscr H_\eps:=-\Delta+\xi_\eps
$$ 
is an unbounded operator of $L^2(\T_L^2)$ whose domain is the Sobolev space $H^2(\T^2)$.  
Its resolvent $\mathcal{G}_a^\eps:= (a+\mathscr{H}_\eps)^{-1}: L^2\to H^2$, 
which is well defined for $a$ large enough, is compact and we can apply the spectral theorem. Note that 
we know from Theorem \ref{main}
that there is a choice of $\Xi_2^\eps$ such that 
$\mathscr H_\eps=\mathscr H(\xi_\eps,\Xi^\eps_2)$. 
Our approximation result can be enunciated as follows. 
\begin{theorem}\label{th:conv}
Let $\alpha<-1$, $\xi$ be a Gaussian white noise, $\xi_\eps:=\xi\star\hat\theta_\eps$ be a smooth mollification 
of $\xi \in \mathscr C^\alpha$ 
and $\Xi_2^{\eps}$ as given in Theorem~\ref{main} such that 
$\mathscr H(\xi_\eps,\Xi^\eps_2)= \mathscr H_\eps$. Then, 
there exists $\Xi^{wn}=(\xi,\Xi^{wn}_2)\in\mathscr X^\alpha(\mathbb T_L^2)$  
and a constant $c_\eps:=c_\eps(\theta)
\to+\infty$ as $\eps \to 0$ such that the following convergence holds 
$$
(\xi_\eps,\Xi_2^\eps+c_\eps)\quad \underset{\eps\to0}{\longrightarrow} \quad (\xi,\Xi^{wn}_2)
$$
in $L^p(\Omega,\mathscr C^{\alpha}\times\mathscr C^{2\alpha+2})$ for all $p>0$ and 
almost surely in $\mathscr C^{\alpha}\times\mathscr C^{2\alpha+2}$. Moreover, 
the limiting distribution $\Xi$ does not depend on the mollification function $\theta$ 
and the normalizing constant $c_\eps$ has the following asymptotic expansion 
\begin{align}\label{expansion_c_eps}
c_\eps=\frac{1}{2\pi}\log(\frac{1}{\eps}) + O(1),
\end{align}
where $O(1)$ refers to any fixed constant, independent of $\eps$. 
\end{theorem}

\begin{remark}
Note that the asymptotic expansion \eqref{expansion_c_eps} is universal at the leading order when $\eps\to 0$, 
in the sense that the largest term does not depend 
on the mollification $\theta_\eps$ used for the regularization.  The second term $O(1)$ is however not universal and one can actually choose
any constant in $\R$ so that the Theorem \ref{th:conv} remains valid. As a consequence, the limiting distribution $\Xi_2$ is unique 
up to an additive constant. Theorem \ref{th:conv} was first proved in~\cite{gip,hairer} where the authors obtain 
the well posedness of the parabolic Anderson equation on the two dimensional torus. The present version is a slight modification 
of their result. 
\end{remark}

Endowed with Thoerems \ref{main} and \ref{th:conv}, we are now able to define the Schr\"odinger operator $\mathscr H$ associated 
to the Gaussian white noise potential $\xi$ simply by setting 
\begin{align}\label{def_H_gaussien}
\mathscr H:= \mathscr H(\Xi^{wn})\,. 
\end{align}
We now establish the convergence in the sense of the resolvent (convergence of the spectrum) 
of the smooth approximations $\mathscr H_\eps +c_\eps$ as defined above in Theorem \ref{th:conv} 
towards the operator $\mathscr H$, so that the definition \ref{def_H_gaussien}
makes sense. The following Theorem is the second main result of our paper.  

%

\begin{theorem}
With the same notations as in Theorem \ref{th:conv}, we denote by
\begin{align*}
\Lambda_1^\eps \quad \le \quad \Lambda_2^\eps \quad \le \quad \Lambda_3^\eps \quad \le \quad \cdots
\end{align*}
the eigenvalues of the operator $\mathscr{H}_\eps$. 
Then, for any $n\in \N$, almost surely,
\begin{align*}
\Lambda_n^\eps + c_\eps \quad \underset{\eps\to 0}{\longrightarrow}  \quad \Lambda_n(\Xi^{wn})\,,
\end{align*}
where $(\Lambda_n(\Xi^{wn}))_{n\in \N}$ denotes the discrete set of the eigenvalues of $\mathscr{H}(\Xi^{wn})$.  
\end{theorem}

%

So far we have constructed the operator $\mathscr H(\Xi^{wn})$ associated to the two-dimensional Gaussian
white noise $\xi$ and establish the convergence of smooth approximations. 
We are now interested in the limiting 
spectral statistics of the operator  $\mathscr H(\Xi^{wn})$ 
when the volume of the torus, denoted as $L$ above, 
tends to $+\infty$. In the limit of a very large torus, it is expected that the eigenfunctions 
associated to the lowest eigenvalues will be localized (this property 
is known as the Anderson localization) as can be observed in the
picture of the eigenfunction associated to the bottom eigenvalue 
in Fig. \ref{fig-anderson-localization} for $L=10$ (this picture is obtained from a numerical 
simulation and diagonalization of the discretized operator on a grid with small mesh size).
We provide a picture of the first eigenfunction in Fig.  \ref{fig-eigenfunction-1} with $L=1$ 
for comparison and to illustrate the effect of the Gaussian white noise.  

Instead of the localization, a weaker result is 
to prove the convergence of the bottom eigenvalues in their scaling region as $L\to \infty$
towards a Poisson point process. Even in the one dimensional case 
(described below in section \ref{dimension-one}), this conjecture remains to be proved 
(see \cite{laure} for a discussion on this conjecture and \cite{laure-sine} for a related model 
where this convergence is proved).  

In the two-dimensional case, we obtain in this paper only partial results in this direction. 
Mainly, we are able to give an upper-bound on the asymptotic order of the ground state in the limit of large volume $L\to \infty$. In the one-dimensional case, McKean \cite{mckean} established the convergence in law 
of the ground state towards a Gumbel distribution, the asymptotic order of the minimal eigenvalue being (up to a multiplicative factor) $- \log(L)^{2/3}$.  

\begin{figure}\center
\includegraphics[scale=0.6]{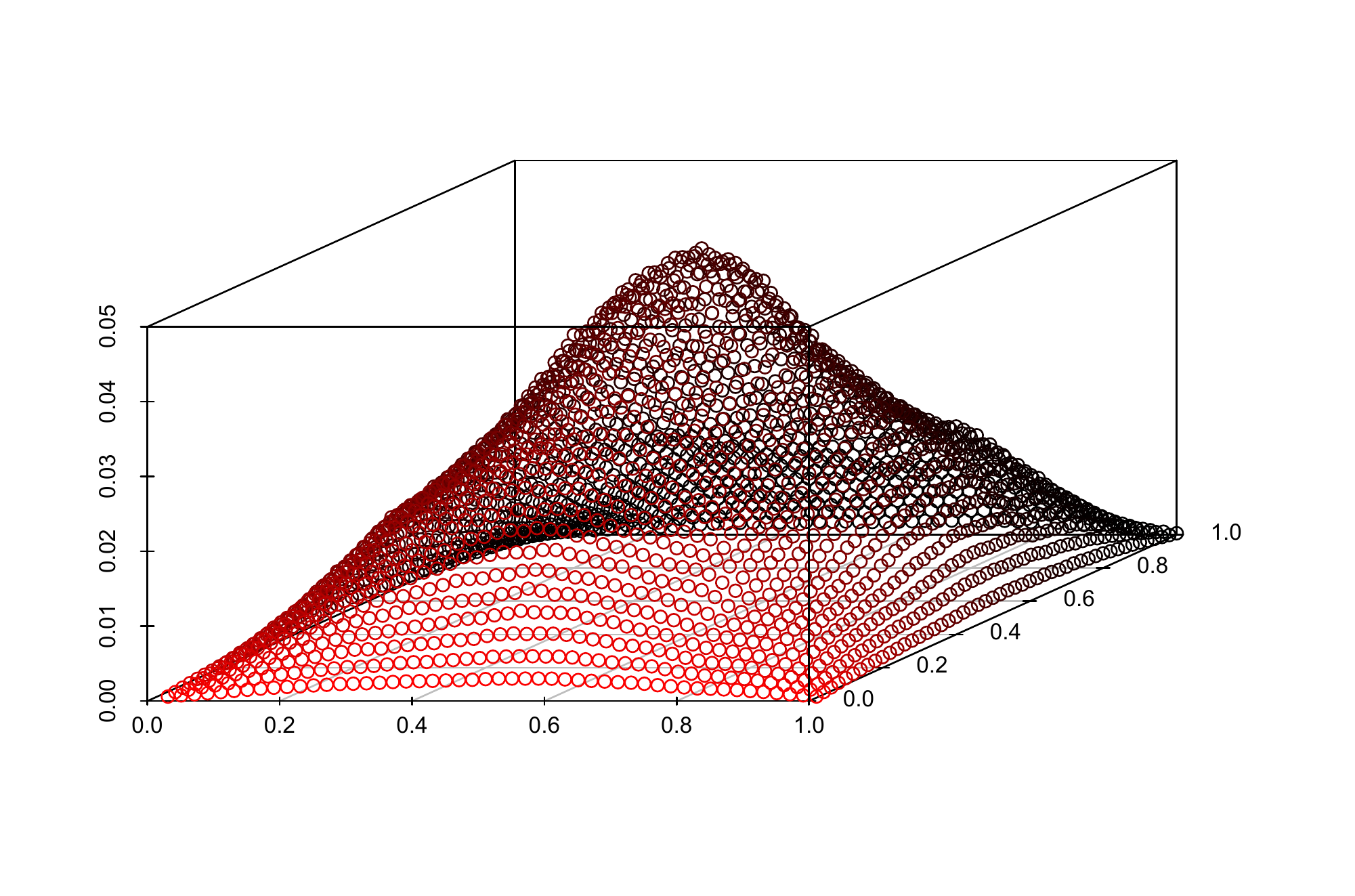}
\includegraphics[scale=0.6]{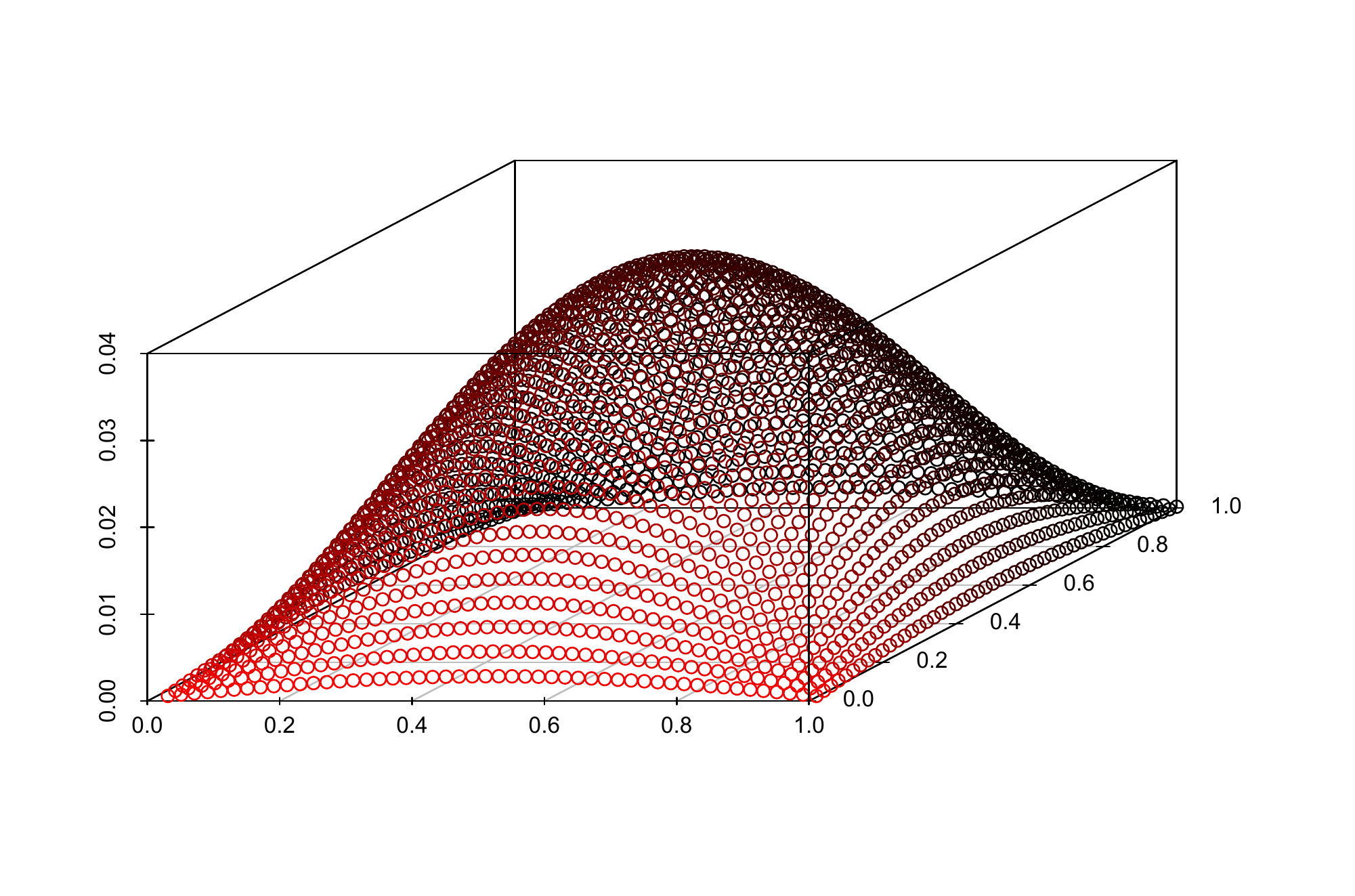}
\caption{(Color online). Sample graph of the first eigenfunction associated to the minimal eigenvalue 
of the operator $\mathscr{H}$ with Dirichlet boundary conditions and on a square of size $L=1$. 
We also display the first eigenfunction of the Laplacian operator $-\Delta$ to illustrate the effect of the noise.}
\label{fig-eigenfunction-1}
\end{figure}

\begin{figure}\center
\includegraphics[scale=0.6]{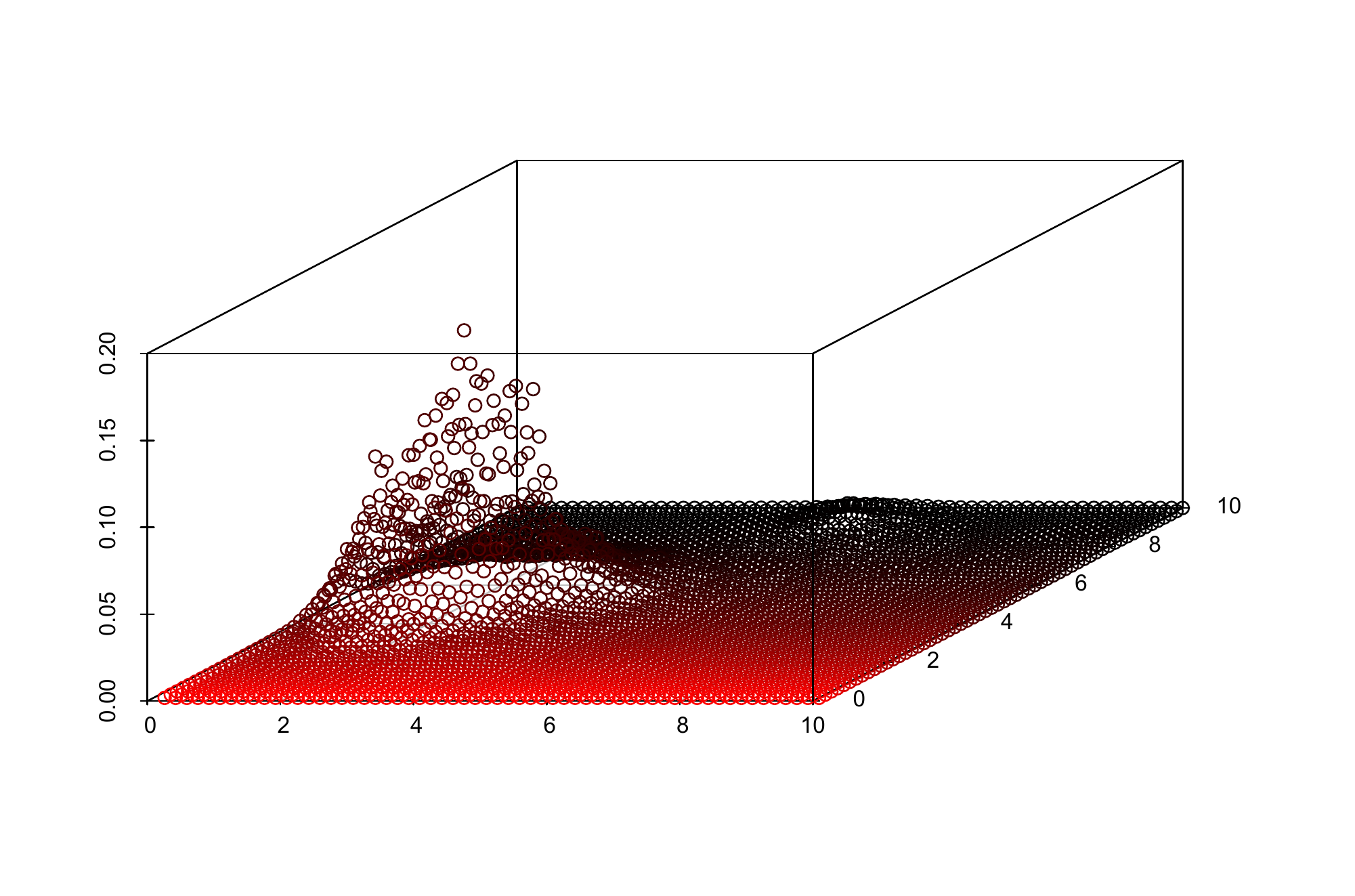}
\caption{(Color online). Anderson localization of the first eigenfunction of $\mathscr{H}$ on a larger box of size 
$L=10$. }\label{fig-anderson-localization}
\end{figure}

\begin{theorem}\label{th:spectre}
For any $n\in \N$ and $p\geq 1$, 
\begin{align}\label{eq:logestimate}
\sup_{L>0}\mathbb E\left[ \left|\frac{\Lambda_n(\Xi^{wn})}{\log L}\right|^p \right] < +\infty
\end{align}

Besides, there exist two positive constants $C_1$ and $C_2$ such that for any $x<0$, we have
\begin{align}
e^{C_2x}\leq\mathbb P(\Lambda_1(\Xi^{wn})\leq x)\leq e^{C_1x}\,. 
\end{align}
\end{theorem}

\begin{remark}
The estimate \eqref{eq:logestimate} proves that the asymptotic order of the ground state is larger or equal to 
$-\log L$ when $L\to\infty$ (the ground state can not be much below a multiple of $-\log L$ as $L\to +\infty$). 
The tail estimates for the ground state $\Lambda_1(\Xi^{wn})$ hold 
for any $L>0$ but unfortunately, we do not have a good control on the constants $C_1$ and $C_2$ in $L$, 
so that we are not able to extract further information on the asymptotic order of the ground state. 
We think that this question deserves to be investigated in more details. 
\end{remark}

Let us close this section with a brief remark about the three dimensional case. 
\begin{remark}
As pointed out before, our work can be generalized to the case of the three dimensional Gaussian 
white noise which lives in the space $\mathscr C^{-\frac{3}{2}-}(\mathbb T^3)$. 
According to the well-posedness result of the Parabolic Anderson equation stated in~\cite{HL}, 
the renormalization constant takes the form 
\begin{align*}
c_\eps=\frac{a}{\eps}+c\log(\frac{1}{\eps})+O(1)
\end{align*}
when $\eps\to0$ and where $a$ and $b$ are two constants. Moreover, as we shall see in 
Section~\ref{section:growth}, the growth and tail estimates for the ground state are proved with 
a scaling argument which should work also in general dimension $d\leq 3$. We expect 
the following estimates to hold in dimension $d$,
$$
\exp(-C_2(-x)^{2-\frac{d}{2}})\leq\mathbb P(\Lambda_1(\Xi^{wn})\leq x)\leq \exp(-C_1(-x)^{2-\frac{d}{2}})
$$        
for any $x<0$ and  
$$
\sup_{L}\mathbb E\left[\left|\frac{\Lambda_n(\Xi^{wn})}{(\log L)^\frac{2}{4-d}}\right|^{p}\right]<\infty
$$
for all $p\geq 1$.
\end{remark}

\noindent {\bf Acknowledgements:} 
We are very grateful to Professor M. Hairer for pointing out that the space of rough distribution is trivial in our setting 
and for his numerous comments which have permitted us to improve this manuscript. We also would like to thank Professors 
P. K. Friz and N. Perkowski 
for the numerous discussions and fruitful advices. KC is funded by the RTG 1845 and a large part of this work was carried out while K.C was employed by the T.U Berlin university. Both authors were supported by the  European Research Council 
under the European Union's Seventh Framework Programme (FP7/2007-2013) / ERC grant of Professor P.K.Friz.
agreement nr. 258237.

\section{The Anderson hamiltonian on a finite interval}\label{dimension-one}
A rigorous definition of the random operator $\mathscr{H}$ in dimension $d=1$ 
was first given in 1977 by Fukushima and Nakao in ~\cite{fukushima}. 
Although no precise formulation of the eigenvalues problem was available before this paper, 
the random spectrum of $\mathscr{H}$ was first 
studied in the physic literature by Frisch and Lloyd back in 1960 ~\cite{lloyd}
and also by Halperin in 1965 ~\cite{halperin}. 
A rigorous approach due to McKean to compute the limiting 
distribution of the ground state when the operator $\mathscr{H}$ is 
considered on a long box can be found in ~\cite{mckean}.

Let us briefly recall the construction ~\cite{fukushima}  of the operator $\mathscr{H}$ 
and the main results on the ground state 
distribution due to McKean ~\cite{mckean}.  

The authors of ~\cite{fukushima} work with Dirichlet boundary conditions and 
define the stochastic linear operator 
\begin{align*}
\mathscr{H}:= - \frac{d^2}{dx^2} + B^\prime(x)
\end{align*}
on the space of functions 
\begin{align*}
\mathscr{D} := \{f\in H^{1}([0,L],\R): f(0)= f(L)=0\} \,. 
\end{align*}
Their construction can be extended to other boundary conditions along the same lines.
If $f\in  H^{1}([0,L],\R)$, the product $B^\prime(x) f(x)$ is defined as the Schwarz derivative 
of a continuous function using integration by part:
for any $x\in[0,L]$,  
\begin{align*}
B^\prime(x) f(x):= \frac{d}{dx} \left[f(x) B(x) - \int_0 ^x f^\prime(y) B(y) dy \right]\,. 
\end{align*}
At this point, the operator $\mathscr{H}$ is defined on the domain $\mathscr{D}$ 
and takes values in the space of distribution. The eigenvalue 
problem can now be defined: we say that $(\lambda,f_\lambda) \in \R \times \mathscr{D}$
is an eigenvalue/eigenfunction pair if 
\begin{align}\label{eigenvalue-pb-1}
-f_\lambda^{\prime\prime}(x) + B^\prime(x) f_\lambda(x) = \lambda f_\lambda(x)
\end{align}

This equality can be integrated with respect to $x$ and rewritten in its equivalent integrated form 
\begin{align}\label{integrated-form}
f_\lambda^\prime(x) - f_\lambda^\prime(0) = f_\lambda(x) B(x) - \int_0^x f_\lambda(y) \, B(y)\, dy - \lambda \int_0^x f_\lambda(y) dy 
\end{align}
for any $x\in [0,L]$. 
From \eqref{integrated-form}, we see that $f_\lambda^\prime$ has the same regularity as the Brownian motion 
$B(x)$ and therefore has H\"older regularity $1/2-\varepsilon$ for any $\varepsilon >0$. 
We can conclude that the eigenfunction $f_\lambda$ itself has H\"older regularity $3/2-$
and in particular belongs to the space $H^1$, which proves that 
it is indeed sufficient 
to define $\mathscr{H}f$ for $f\in H^1$ (the space $H^1$ contains the eigenvectors of $\mathscr{H}$). 

With this definition of the stochastic linear operator $\mathscr{H}$, the authors of ~\cite{fukushima}
prove the following Theorem. 
\begin{theorem}[Fukushima, Nakao (1977)]
Let $L>0$.
The random spectrum of $\mathscr{H}$, when considered with Dirichlet boundary conditions on the finite box 
$[0,L]$, has a well defined $k$-th lowest element $\Lambda_k$. 
Furthermore, almost surely,
\begin{align*}
\Lambda_1\quad < \quad \Lambda_2 \quad < \quad \Lambda_3 \quad < \quad \cdots 
\end{align*}
The eigenvectors $(f^\star_n)_n$ are H\"older $3/2-$
and form an orthonormal basis of $L^2$. 
\end{theorem}
The proof follows the classical lines of the spectral Theorem
with a minimization of the associated quadratic form. 
Let us outline this proof here for completeness. 

The quadratic form associated to the operator $\mathscr{H}$ is well defined for $f\in H^1$ 
and satisfies
\begin{align*}
\langle f, \mathscr{H} f \rangle_{L^2} = \int_0^L f^\prime(x)^2 dx - 2 \int_0^L f^\prime(x) f(x) B(x) dx \,. 
\end{align*}
Before going into the optimization procedure, we first need to establish upper and lower bounds 
for the quadratic form. We denote by $M$ the supremum value of the Brownian path on the interval $[0,L]$,
\begin{align*}
M:=\sup_{x \in [0,L]} B(x)\,.
\end{align*}
It is easy to check that, almost surely, 
\begin{align*}
\langle f, \mathscr{H} f \rangle \le ||f'||_{L^2}^2 + \frac{M}{2} (  ||f'||_{L^2}^2 + ||f||_{L^2}^2 )  \,. 
\end{align*}
 The lower bound is slightly more involved: Using Cauchy-Schwarz inequality and minimizing a quadratic 
 form in $(||f||_{L^2}, ||f^\prime||_{L^2})$, we obtain 
 \begin{align}\label{lower-bound}
\langle f, \mathscr{H} f \rangle + (M^2+1) \, ||f||_{L^2}^2 \ge \frac{1}{M^2+3} (||f'||_{L^2}^2 + ||f||_{L^2}^2 ) \ge 0\,. 
\end{align}
We now set 
\begin{align*}
\Lambda_1:= \inf_{f\in H^1;\atop ||f||_{L^2}=1 }  \langle f, \mathscr{H} f \rangle > -\infty, 
\end{align*}
and consider a sequence $(f_n)_{n\in \N}$ of functions in $H^1$ such that $||f_n||_{L^2}=1$and 
 $ \langle f_n, \mathscr{H} f_n \rangle \to \Lambda_1$. 
 It is plain to check from the lower bound \eqref{lower-bound} that almost surely
\begin{align}\label{control-fprime_n}
\sup_{n \in \N} ||f_n^\prime||_{L^2} < +\infty\,.
\end{align}
This uniform control on the $L^2$ norms of the derivatives $f_n^\prime$ is crucial as it permits us 
to prove the existence of a limit point $f^\star_1\in H^1$ such that along a subsequence 
\begin{itemize}
\item $f_n\to f^\star_1$ uniformly in $\mathcal{C}^0$ (using the Ascoli-Arzela Theorem),
\item $f_n\to f^\star_1$ in $L^2$,
\item $f_n\to f^\star_1$ weakly in $H^1$ (using the Banach-Alaoglu Theorem). 
\end{itemize} 
Therefore, $||f_1^\star||_{L^2} = 1$ and we can prove passing to the limit along the subsequence that 
\begin{align*}
\langle f^\star_1, \mathscr{H} f^\star_1 \rangle = \Lambda_1\,. 
\end{align*}
To check that $\mathcal{H}f^\star_1 = \Lambda_1 f^\star_1$ in the sense of distributions, it suffices to write 
that the derivative of the quadratic form in the direction of any smooth function $\varphi$ is zero (as $f_1^\star$ is a minimizer): 
\begin{align*}
\left.
\frac{d}{d\varepsilon} \frac{\langle (f^\star_1+\varepsilon \varphi ), \mathscr{H}(f^\star_1+\varepsilon \varphi )\rangle }{||  f^\star_1+\varepsilon \varphi||_2^2 }\right|_{\varepsilon = 0} =0 \,.
\end{align*}

For the next eigenvalues/eigenvectors, we restrict to the orthogonal complementary 
of Vect$(f_1^\star)$ to obtain following the same method the second eigenfunction 
$f_2^\star$ such that 
$\mathcal{H}f_2^\star = \Lambda_2 f_2^\star$, $\langle f_2^\star, f_1^\star\rangle_{L^2}=0$ and $||f_2^\star||_{L^2}=1$. 
We iterate this argument to obtain the existence of an orthonormal family of eigenfunctions $(f_n^\star)$
respectively associated to eigenvalues $(\Lambda_n)$ which satisfy  
\begin{align*}
\Lambda_1\le\Lambda_2\le\Lambda_3\le\cdots 
\end{align*}
We have not shown yet that the eigenvalues have multiplicity one. 
This fact will actually follow from the forthcoming characterization of the law of the eigenvalues 
in term of a family of interacting diffusions valid 
in the special case of Dirichlet boundary conditions. 
We believe the eigenvalues are also simple with Periodic boundary conditions although 
(to the best of our knowledge) no proof of this fact is available at the time of writing this paper
\footnote{Note that the Laplacian operator has eigenvalues with multiplicity strictly greater than one 
when considered under Periodic boundary conditions. We believe that adding the Gaussian white noise 
will separate the multiple eigenvalues into several simple eigenvalues.}.

The main idea (which was used in many papers ~\cite{lloyd,halperin,virag-1,valko,laure,laure-sine}) is to 
use the Riccati transform 
to rewrite the eigenfunction differential equation 
$\mathscr{H} f_\lambda = \lambda f_\lambda$ (see also \eqref{eigenvalue-pb-1}) 
as a first order stochastic differential equation. 
More precisely, if one sets $X_\lambda(x):= f_\lambda'(x)/f_\lambda(x)$, 
the second degree equation \eqref{eigenvalue-pb-1} 
is mapped to the stochastic differential equation
\begin{align}\label{riccati}
dX_\lambda(x) = -(\lambda+ X_\lambda(x)^2 )\, dt + dB(x) \,. 
\end{align}
The initial condition imposed 
on $f_\lambda$ translates as an initial condition on $X_\lambda$. With Dirichlet boundary conditions for $f_\lambda$, 
one has \begin{equation}\label{initial-cond}
X_\lambda(0)=+\infty.\end{equation} 
The trajectory of $X_\lambda$ determines the trajectory of $f_\lambda$ 
up to a normalization factor
(the equation on $f_\lambda$ is linear).  
The function $X_\lambda$ is associated to an eigenfunction (so that at the same time $\lambda$ 
is an eigenvalue) if and only if it 
the diffusion blows up to $-\infty$ precisely at the end point $t=L$, i.e. $X_\lambda(L)=-\infty$.
under Dirichlet boundary 
condition.

%

Using this observation, we can easily deduce a characterization 
of the joint law of the eigenvalues in term 
of the family of coupled \footnote{They are all driven by the same Brownian motion 
$(B(x))_{x\ge 0}$. } diffusions $\{X_\lambda, \lambda\in \R\}$ 
such that
\eqref{riccati} and \eqref{initial-cond} hold for all $\lambda\in \R$. 
The idea is that the zeros of the eigenfunction $f_\lambda$ for $\lambda= \Lambda_k, k \in \N$
correspond to the zeros of the associated diffusion process $X_\lambda$  
and that the trajectories 
of $X_\lambda$ is a monotonic function of $\lambda$. Tuning the value of $\lambda$ permits one 
to find the eigenvalues, which correspond to the values of $\lambda$ for which the diffusion $X_\lambda$ explodes 
precisely at the end point $x=L$. 
We do not need to explicit the characterization of the joint law for the purpose of this section.  
Let us just state two simple properties regarding the law of the eigenvalues: 
\begin{itemize}
\item The distribution of the minimal eigenvalue $\Lambda_0=\Lambda_0(L)$ is characterized in terms of the family of diffusions $X_\lambda$ as 
\begin{align*}
\PP[\Lambda_0  \le \lambda] = \PP[X_\lambda \mbox{ blows up before time } L]\,.  
\end{align*}
\item The number of $\mathscr{H}$-eigenvalues below $\lambda$ is equal to the number of explosions of the diffusion 
$X_\lambda$ before time $L$. 
\end{itemize} 

Using this characterization, McKean \cite{mckean} proves the following convergence in law after a careful analysis 
of the explosion time of the diffusion $X_\lambda$ for a fixed value of $\lambda$.  
\begin{Proposition}[McKean (1994)]
In the limit $L\to \infty$, the fluctuations of 
the minimal eigenvalue $\Lambda_1$ of $\mathscr{H}$ are governed by the Gumbel 
distribution. More precisely, we have the following convergence in law as $L\to \infty$,
\begin{align*}
- 2 \cdot 3^{1/3} (\ln L)^{1/3} \left[\Lambda_1 + \left(\frac{3}{8}\ln \frac{L}{\pi}\right)^{2/3}  \right] 
\Rightarrow e^{-x} \exp(-e^{-x}) \, dx\,. 
\end{align*}
\end{Proposition}
\begin{remark}
Let us finish this section by pointing out the main difference between the one dimensional construction of the operator and the two dimensional one. The construction presented in this section is crucially related to the fact that we can multiply the white noise $\xi:=dB$ by any function in $H^1$. This allows one 
to define the quadratic form $\langle \mathscr Hf,f\rangle$  by duality on the space $H^1$ 
and to prove that it is lower semi-bounded. In dimension $2$, this picture is blurred, indeed in that case, 
the white noise can only be multiplied by function in $\cap_{\eps}H^{1+\eps}$ and one can think to 
take $H^{1+\eps}$ as a domain of $\mathscr H$. Unfortunately, this approach is not consistent in the sense 
that the quadratic form $\langle \mathscr H f,f\rangle$ is not lower semi-bounded any more. 
As we will see in the next section, the general idea to overcome this problem is to define the operator on a Hilbert space of irregular functions $f$ for which the most irregular part of the product $f\xi$ is compensated by the term $-\Delta f$  so that $\mathscr H f\in L^2$.
\end{remark}

\section{Besov spaces and Bony paraproducts}
\label{sec:bony}
We recall the definitions of the Bony paraproducts, the Besov and Sobolev spaces 
and collect the two main results that we will be useful throughout the paper regarding 
the products between two Schwartz distributions and the effect of differentiation on the regularity 
of the distributions.
We work on the two dimensional torus $\T_L^2:= \R^2/(L^{-1}\Z)^2$ with diameter $L$ (see~\cite{BCD-bk} for a review on this subject).   
For any $f$ in the Schwartz space $\mathscr {S}' ( \mathbb T_L^2,\mathbb R)$ of tempered distributions on $\T_L^2$, the Fourier transform of $f$ will be 
denoted  
$\hat{f} : \mathbb T_L^2\to \C$ (or sometimes $\mathscr{F}f$) and is defined for $k\in \mathbb Z_L^2$ by
\begin{align*}
\hat{f}(k) := \langle f,L^{-1}\exp(i 2\pi \langle k, \cdot\rangle) =\frac{1}{L} \int_{\T_L^2} f(x) \exp(-i 2\pi \langle k,x\rangle) dx. 
\end{align*}
Recall that for any $f\in L^2(\T_{L}^2,\R)$ and $x\in \T_L^2$, 
we have 
\begin{align}\label{decomposition-fourier}
f(x)=\frac{1}{L}\sum_{k \in \Z_L^2} \hat{f}(k) \exp(i 2\pi \langle k,x\rangle). 
\end{align}
 
The Sobolev space $H^\alpha(\mathbb T_L^2,\R)$ with index $\alpha\in \R$ is defined as  
\begin{align*}
H^\alpha(\mathbb T_L^2,\R) := \{f \in \mathscr {S}' ( \mathbb T_L^2,\mathbb R): 
\sum_{k\in (L^{-1}\Z)^2} (1+|k|^2)^\alpha \, |\hat{f}(k)|^2 < +\infty\}\,. 
\end{align*}   

Before recalling the definition of the Besov spaces, we first need to introduce the 
Littlewood-Paley blocks which permit us to decompose a distribution $f$ into an infinite series of smooth 
functions \footnote{This decomposition is more convenient than the Fourier decomposition \eqref{decomposition-fourier}}.
 We denote by $\chi$ and $\rho$ two nonnegative smooth and compactly supported 
radial functions $\R^2\to \C$ such that \footnote{The existence of two such functions in insured by 
\cite[Proposition 2.10]{BCD-bk}}
\begin{enumerate}
	\item The support of $\chi$ is contained in a ball $\{x\in \R^2: |x| \le R\}$ 
	and the support of $\rho$ is contained in an annulus $\{x\in \R^2: a\le |x| \le b\}$;
	\item For all $\xi \in \R^2$, $\chi(\xi)+\sum_{j\ge0}\rho(2^{-j}\xi)=1$;
	\item For $j\ge 1$, $\chi  \rho(2^{-j}\cdot)\equiv 0$ and $\rho(2^{-i}\cdot) \rho(2^{-j}\cdot)\equiv 0$
	for $|i-j|\ge 1$. 
\end{enumerate}
The Littlewood-Paley blocks $(\Delta_j)_{j\geq-1}$ associated to a tempered distribution $f\in\mathscr S'(\mathbb T_L^2,\R)$ 
are defined by 
\begin{align*}
\mathscr F(\Delta_{-1} f) = \chi \cF f\ \mbox{ and\ for }\ j \ge 0, \quad \mathscr F(\Delta_j f) = \rho(2^{-j}.)\cF f.
\end{align*}
Note that, for $f\in\mathscr S'(\mathbb T_L^2,\R)$, the Littlewood-Paley blocks $(\Delta_j f)_{j\ge -1}$ define 
smooth functions (their Fourier transform have compact supports). We also set, for $f\in  \mathscr S'$ and 
$j\ge -1$, 
\begin{align*}
S_j f := \sum_{i=-1}^{j-1} \Delta_i f
\end{align*} 
and note that $S_j f$ converges weakly 
to $f$ when $j\to \infty$.

The Besov space with parameters $p,q \in \R_+,\alpha \in \R$ can now be defined as  
\begin{equation}\label{eq:Besov}
\mathscr B_{p,q}^{\alpha}(\T_L^2, \R):=\left\{u\in \mathscr S'(\T_L^2,\R); \quad ||u||_{\mathscr B_{p,q}^{\alpha}}=
\left(\sum_{j\geq-1}2^{jq\alpha}||\Delta_ju||^q_{L^p} \right)^{1/q} <+\infty\right\}.
\end{equation}
We also define the Besov $\alpha$-H\"older space
\begin{align*}
\mathscr C^{\alpha}:=\mathscr B_{\infty,\infty}^{\alpha}
\end{align*}
which is naturally equipped with the norm 
$||f||_{\mathscr{C}^\alpha}:=||f||_{\mathscr B_{\infty,\infty}^{\alpha}}=\sup_{j\ge -1} 2^{j \alpha} 
||\Delta_j f||_{L^\infty}$.
Note also that the Sobolev space $H^\alpha$ coincides with $\mathscr B_{2,2}^\alpha$.

We now consider the product between two distributions 
 $f\in\mathscr C^{\alpha}$ and $g\in \mathscr C^{\beta}$.  
 At least formally, we can decompose the product $fg$ as  
\begin{align*}
fg=f\prec g+f\circ g+f\succ g
\end{align*}
where
\begin{align*}
f\prec g :=\sum_{j\ge -1} \sum_{i=-1} ^{j-2} \Delta_i f \Delta_j g, 
\quad f\succ g:= \sum_{j \ge -1} \sum_{i=-1}^{j-2} \Delta_i g \Delta_j f
\end{align*}
are usually referred as the \textit{paraproduct terms} whereas 
\begin{align*}
f\circ g:=\sum_{j\geq-1}\sum_{|i-j|\leq 1}\Delta_i f \Delta_j g 
\end{align*} 
is called the \textit{resonating term}. 
The paraproduct terms are always well defined whatever the values of $\alpha$ and $\beta$.
The resonating term is well defined if and only if $\alpha+\beta >0$. This is reminiscent to the well known 
fact that one can not generically form the product of two distributions: the two regularities 
must compensate one another in such a way that the sum is strictly positive.  
These (deterministic) facts can be summarized as in the following proposition where we give estimates 
on the regularities of the paraproducts and resonating terms.   
\begin{Proposition}[Bony estimates]
\label{prop:bony-estim}
Let $\alpha,\beta\in\mathbb R$. We have the following upper bounds: 
\begin{enumerate}
\item If $f\in L^2$ and $g\in \mathscr C^\beta$, then  
\begin{align*}
||f\prec g||_{H^{\beta-\delta}}\leq C_{\delta,\beta}||f||_{L^2}||g||_{\mathscr{C}^\beta}\,. 
\end{align*} 
for all $\delta>0$
\item if $f\in H^{\alpha}$ and $g\in L^{\infty}$ then 
\begin{align*}
||f\succ g||_{H^{\alpha}}\leq  C_{\alpha,\beta}||f||_{H^\alpha}||g||_{\mathscr{C}^\beta}\,. 
\end{align*}
\item If $\alpha<0$, $f\in H^\alpha$ and $g\in\mathscr C^\beta$, then 
\begin{align*}
||f\prec g||_{H^{\alpha+\beta}}\leq C_{\alpha,\beta}||f||_{H^\alpha}||g||_{\mathscr{C}^\beta}\,. 
\end{align*}
\item If $g\in\mathscr C^{\beta}$ and $f\in H^{\alpha}$ for $\beta<0$ then 
\begin{align*}
\|f\succ g\|_{H^{\alpha+\beta}}\leq C_{\alpha,\beta}||f||_{H^\alpha}||g||_{\mathscr{C}^\beta}
\end{align*}
\item If $\alpha+\beta>0$ and $f\in H^\alpha$ and $g\in\mathscr C^\beta$, then 
\begin{align}\label{resonating-bony}
||f\circ g||_{H^{\alpha+\beta}}\leq C_{\alpha,\beta} ||f||_{H^\alpha}||g||_{\mathscr{C}^\beta}\,. 
\end{align} 
\end{enumerate}
where $C_{\alpha,\beta}$ is a finite positive constant which does not depend on the size of the Torus.
\end{Proposition} 
\begin{remark}
We will use extensively the first, the fourth and the last estimates of this Proposition in the Section~\ref{construct-order-1} to construct our operator while the second and the third will be used in the proof of the commutation Lemma~\ref{prop:commu}  . 
\end{remark}

We end this section by describing the action of the Fourier multipliers on the Besov spaces.
Those "multiplications'' in the Fourier space correspond to differentiations (respectively ``integrations'') 
of the distributions 
in the Besov spaces and the following proposition quantifies the loss (resp. gain) of regularity obtained 
by differentiating (resp. ``integrating'') distributions. 
\begin{Proposition}[Schauder estimate]
\label{prop:multiply}
Let $\alpha,n \in\mathbb R$ and $\sigma:\mathbb R^2\setminus\{0\}\to \R $ 
be an infinitely differentiable function 
such that $|D^k\sigma(x)|\leq C(1+|x|)^{-n-k}$ for all $x\in \R^2$. 
For $f\in H^\alpha$ (respectively $\mathscr C^{\alpha}$), we define the distribution $\sigma(D)f$ obtained from applying the 
differentiation operator $\sigma(D)$ to $f$ as 
\begin{align*}
\sigma(D)f:=\mathscr F^{-1}(\sigma \mathscr{F}f)\,. 
\end{align*}
Then, $ \sigma(D)f \in H^{\alpha+n}$ (respectively $\mathscr C^{\alpha+n}$) and
\begin{align*}
||\sigma(D)f||_{H^{\alpha+n}}\leq C_{n,\alpha} C||f||_{H^{\alpha}}\,. 
\end{align*}
and  the same bound hold for the H\"older space $\mathscr C^\alpha$. Moreover the constant $C_{n,\alpha}$ does not depend on $L$ 
\end{Proposition}
Further important technical results related to Besov spaces and Bony paraproducts 
(which will be used in the following) are gathered in the appendix. 

\section{Schr\"odinger operator with singular potential}

\subsection{Definition on a space of paracontrolled distributions}
\label{construct-order-1}
Our goal in this section is to define the linear operator $\mathscr{H}$ defined in \eqref{def-H-xi}. 
More precisely, starting from a test function $f\in L^2(\T_L^2,\R)$, we would like to define the function $g$ such that 
\begin{align*}
g = \mathscr{H}f= - \Delta f + \xi f\,. 
\end{align*}
This operation is ill defined for a generic element $f \in  L^2(\T^2,\R)$ because of the product $\xi f$ which makes sense only if 
$f$ is sufficiently smooth. 
Indeed, we know from Young's integration theory that the product $fg$ between two distributions $f$ and $g$ 
with respective H\"older-Besov regularities $\alpha$ and $\beta$ (see Appendix \ref{sec:appendix} for a reminder on Besov-H\"older spaces)
 is well defined if and only if the sum $\alpha+\beta > 0$. 

Here, we need to define the linear operator $\mathscr{H}$ on a functional space large enough to contain the eigenfunctions
of $\mathscr{H}$.   
The eigenvalue problem associated to the operator $\mathscr{H}$ simply writes 
\begin{align}\label{eigenvalue-pb-2}
\mathscr{H}f_\lambda = \lambda f_\lambda
\end{align}
or in the more explicit form
$$
(1-\Delta)f_\lambda=-f_\lambda\xi+(\lambda+1)f_\lambda
$$
where we seek for an eigenfunction $f_\lambda$ associated to the eigenvalue $\lambda$ such that 
\eqref{eigenvalue-pb-2} holds together with the boundary conditions under consideration. 

Let us fix $\alpha < -1$ such that the white noise $\xi$, which may be seen as a random distribution, belongs to the 
Besov-H\"older space $\mathscr{C}^\alpha$ (we know that $\xi \in \mathscr{C}^{-1-\varepsilon}$ almost surely for any $\varepsilon>0$). 

Then, if $f$ satisfies equation \eqref{eigenvalue-pb-2} , we expect from standard (heuristic) 
arguments the regularity of $f$ to be $\alpha+2$ (thanks to the additional regularity induced by the Laplacian), 
which barely prevents us from defining the product $\xi f$ using standard Young integration (this is the classical problem which motivated 
It\^o's theory of stochastic integrals). 
Powerful tools to make sense of such products in the present stochastic context 
have been recently developed by Gubinelli, 
Imkeller and Perkowsky in ~\cite{gip}. 
 Another alternative more general theory has been independently developed by Hairer in ~\cite{hairer}, the so called {\it theory of regularity structures}, 
 which permits one to tackle the same problems, with applications to singular stochastic partial differential equations, in a broader context. 
 This later theory has received much attention recently.
 
For our study, we shall use the Fourier approach introduced in ~\cite{gip} as it is somehow more basic and does not require a large background of harmonica analysis. 
Working in Sobolev spaces will turn out to be crucial for us because they are Hilbert spaces.

We come back to the eigenvalue problem \eqref{eigenvalue-pb-2} which may be rewritten, using Fourier's multipliers 
and thanks to the {\it Bony paraproduct 
decomposition} (see again Appendix \ref{sec:appendix} for a reminder on this), as  
\begin{align}\label{eigenvalue-pb-FM}
f&=f\prec \sigma(D)\xi-f^\sharp, 
\\ f^\sharp&=\sigma(D)( (\lambda+1) f + f\circ\xi+f\succ\xi)+\sigma(D)(f\prec\xi)-f\prec\sigma(D)\xi\notag 
\end{align}
where $D$ is the differentiation operator (see Appendix \ref{sec:appendix}) and $\sigma(k)=-(1+|k|^{2})^{-1}$ for $k \in \Z_L^2$.

Thanks to a formal analysis of \eqref{eigenvalue-pb-FM}, we expect $f\in H^{\alpha+2}$ 
if $\xi\in  \mathscr{C}^{\alpha}$. Indeed, auto consistently, if $f\in H^{\alpha+2}$ 
and $\xi \in \mathscr{C}^\alpha\subseteq H^{\alpha}$, then we know using the Schauder's estimate Proposition \ref{prop:multiply} that the 
regularity of $(\lambda+1) f + f\circ\xi+f\succ\xi\in H^{2\alpha+2}$ 
increases by $2$ when multiplied by $\sigma(D)$
and also that 
$\sigma(D)(f\prec\xi)-f\prec\sigma(D)\xi \in H^{2 \alpha+ 4}$, so that finally $f^\sharp  \in H^{2\alpha+4}$.  
Those heuristic remarks motivate, following ~\cite{gip}, the next definition which shall permit us 
to make sense of the resonating term $f\circ\xi$ (yet ill defined) and define the operator 
$\mathscr{H}$ on the space of paracontrolled distribution $\mathscr D_\xi^\gamma$ introduced.  

\begin{definition}\label{def-D-xi}
Let $\alpha<-1$ and $\xi\in \mathscr C^{\alpha}(\mathbb T_L^2)$. For $\gamma \le \alpha+2$, 
we define the space of distributions  which are paracontrolled by $\sigma(D)\xi$, i.e.
\begin{align}\label{def-D-xi}
\mathscr D_\xi^\gamma
=\left\{f\in H^{\gamma}(\mathbb T_L^2), \quad f^\sharp:=f-f\prec \sigma(D)\xi \in H^{2\gamma}\right\}\,. 
\end{align}
The space $\mathscr D_\xi^\gamma$ equipped with the scalar product $\langle \cdot,\cdot\rangle_{\mathscr D_\xi^\gamma}$, 
defined for $f,g\in\mathscr D_\xi^\gamma$, by
\begin{align*}
\langle f,g\rangle_{\mathscr D_\xi^\gamma}=\langle f,g\rangle_{H^\gamma}+\langle f^\sharp,g^\sharp\rangle_{H^{2\gamma}}.
\end{align*}
is a Hilbert space.
\end{definition}

\begin{remark}
The Hilbert space $(\mathscr D_\xi^\gamma,\langle\cdot,\cdot\rangle_{\mathscr D_\xi^\gamma})$ is continuously embedded in $L^2(\mathbb T_L^2)$.
\end{remark}

Now we claim that we can give a meaning to the resonating term $f\circ \xi$ 
for any $f\in\mathscr D_\xi^\gamma$ with $\gamma \le \alpha+2$ in a ``robust way'', provided we enhance (consistently) the information 
contained in the white noise $\xi$ (see below).

Using \eqref{eigenvalue-pb-FM},  we can decompose $f\circ \xi$ as 
\begin{align}
f\circ\xi&=(f\prec \sigma(D) \xi)\circ\xi+f^\sharp\circ\xi\notag
\\&=f(\xi\circ \sigma(D) \xi)+\mathscr{R}(f,\sigma(D) \xi,\xi)+f^\sharp\circ\xi \label{decomp-reson}
\end{align}
where 
\begin{align*}
\mathscr{ R}(f,\sigma(D) \xi,\xi)=(f\prec \sigma(D) \xi)\circ\xi-f(\xi\circ \sigma(D) \xi)\,. 
\end{align*}
If $f\in \mathscr D_\xi^\gamma$ with $\gamma \le \alpha+2$, then $f^\sharp\in  H^{2 \gamma}$ and the Bony estimate
(see Proposition \ref{prop:bony-estim} Eq. \eqref{resonating-bony}) insures that $f^\sharp\circ \xi $ 
is well defined with regularity $2\gamma+\alpha>0$. 

The key result of ~\cite{gip} yields that the trilinear operator $(f,g,h) \to \mathscr{R}(f,g,h)$, 
which is well defined for smooth test functions
 $f,g,h$, can be continuously extended to any product space
 $\mathscr C^{\alpha}\times\mathscr C^{\beta}\times\mathscr C^{\gamma}$ provided that $\alpha+\beta+\gamma>0$. 
We need to modify slightly this so called {\it commutation} Lemma  ~\cite[Proposition 4.7]{gip} because we work with 
$f\in H^{\alpha+2}$ instead of $\mathscr{C}^{\alpha+2}$.
This commutation Lemma is a crucial tool as it permits us to 
define the resonating term $f\circ \xi$ for a general $f\in \mathscr{D}_\xi^\gamma$ and $\xi \in \mathscr{C}^\alpha$ 
thanks to the knowledge of $\xi\circ\sigma(D) \xi (=\Xi_2) \in \mathscr{C}^{2\alpha+2}$.

\begin{Proposition}\label{prop:commu}
Given $\alpha\in(0,1)$, $\beta,\gamma\in\mathbb R$ such that $\beta+\gamma<0$ and $\alpha+\beta+\gamma>0$, the following trilinear operator 
$\mathscr{R}$ defined for any smooth functions $f,g,h$ by
\begin{align*}
\mathscr R(f,g,h):=(f\prec g)\circ h-f(g\circ h)
\end{align*}
can be extended continuously to the product space $H^\alpha\times\mathscr C^\beta\times\mathscr C^\gamma$.  
Moreover, we have the following bound 
\begin{align*}
||\mathscr R(f,g,h)||_{H^{\alpha+\beta+\gamma-\delta}}\lesssim||f||_{H^\alpha}||g||_{\mathscr C^\beta}||h||_{\mathscr C^\gamma}
\end{align*}
for all $f\in H^\alpha$, $g\in\mathscr C^\beta$ and $h\in\mathscr C^\gamma$, and every $\delta>0$ where the last bound is uniform in $L$.
\end{Proposition}
\begin{remark}
Let us point out that a general version of this commutation Lemma was proved recently in~\cite{dp} for the Besov space $\mathscr B_{2,\infty}^{\alpha}$. However the result presented in~\cite{dp} does not cover our special case but with a slight modification of the proof we are able to proof the needed result  ( see in the Appendix~\ref{sec:appendix} for the proof).
\end{remark}
Applying Proposition \ref{prop:commu} in our case with $f\in H^{\alpha+2}$, $g= \sigma(D) \xi \in \mathscr{C}^{\alpha+2}$ and
$h=\xi \in \mathscr{C}^{\alpha}$ for $\alpha <-1$, we deduce that $\mathscr{ R}(f,\sigma(D) \xi,\xi) \in  H^{3\alpha+4}$
is well defined almost surely.   

Now we still have to define the resonating term $\xi\circ \sigma(D) \xi$ which appears in \eqref{decomp-reson}. 
This is where the enhancement of $\xi\in \mathscr{C}^\alpha$ is needed. 
Actually from the equation~\eqref{decomp-reson} we can see the resonating term as a continuous functional of $(\xi,\xi\circ\sigma(D)\xi)$ which push us to introduce the following definition 
  \begin{definition}
Let  $\alpha<-1$ and $\mathscr E^\alpha=\mathscr C^{\alpha}\times\mathscr C^{2\alpha+2}$. 
Then, the space of rough distributions $\mathscr X^\alpha$ is defined as the closure of the set 
\begin{align}\label{closure-set}
\left\{(\xi,\xi\circ \sigma(D) \xi+c),\quad \xi\in C^{\infty}(\mathbb T_L^2),c\in\mathbb R\right\}
\end{align}
for the topology of the Banach space $\mathscr E^\alpha$. A generic element of $\mathscr X^\alpha$ will be denoted by $\Xi=(\Xi_1,\Xi_2)$. 
If $\xi\in\mathscr C^\alpha$ is such that $\xi=\Xi$, we say that $\Xi$ is an enhancement (or lift) of $\xi$. 
\end{definition}

\begin{remark}
One should notice that if $\xi_\varepsilon$ is a mollification of a two dimensional white noise $\xi$ then 
the expectation of the regularized resonating term  $\xi_\eps\circ \sigma(D)\xi_\eps$ blows up when $\varepsilon \to 0$ as   
\begin{align*}
\mathbb E[\xi_\eps\circ \sigma(D)\xi_\eps]\sim - \log(\eps).
\end{align*} 
Therefore, there is no hope to define $\xi\circ \sigma(D)\xi$ as the limit of $\xi_\eps\circ\sigma(D)\xi_\eps$ when $\varepsilon\to 0$.
One should subtract the diverging expectation so that $\xi_\eps\circ \sigma(D)\xi_\eps-\mathbb E[\xi_\eps\circ \sigma(D)\xi_\eps]$ 
converges, as we shall see in section~\ref{sec:Stoch}. This is precisely the reason why we need to introduce the constants $c\in \R$ 
to form the closure of the set of smooth functions introduced in \eqref{closure-set}.
\end{remark}
\begin{remark}
Let us point that in general the space of rough distribution have a complicated algebraic structure however in our special case $\mathscr X^\alpha$ turn out to be the closure of the couple of smooth function in the space $\mathscr E^{\alpha}$ which of course a linear space. See Lemma~\ref{lemma:rd} for the exact statement.  
\end{remark}

Endowed with this definition, we can now define the resonating term $f\circ \xi$ simply by postulating 
the value of $\xi\circ \sigma(D)\xi= \Xi_2$ from the enhancement $\Xi$ of $\xi$. We have the following Proposition. 

\begin{Proposition}\label{prop:reson}
Let  $-4/3<\alpha<-1$ and $-\frac{\alpha}{2}<\gamma \le \alpha+2$. 
Denote by $\Xi=(\xi,\Xi_2)\in\mathscr X^\alpha$ an enhancement of $\xi\in \mathscr{C}^\alpha$ and let $f\in\mathscr D_\xi^{\gamma}$. 
We can now define  
$f\circ\xi$ as 
\begin{equation}
\label{eq:def-reson}
f\circ\xi=f\Xi_2+\mathscr R(f,\sigma(D)\xi,\xi)+f^\sharp\circ\xi\,. 
\end{equation}

We have the following bound 
\begin{equation}
\label{eq:bound-reson}
||f\circ\xi||_{H^{2\alpha+2}}\lesssim||f||_{\mathscr D_\xi^\gamma}||\Xi||_{\mathscr E^\alpha}(1+||\Xi||_{\mathscr E^\alpha})\,. 
\end{equation}
\end{Proposition}

\begin{proof}
If $f\in\mathscr D_\xi^\gamma$ and $\Xi_2\in \mathscr C^{2\alpha+2}$, we know from Proposition \ref{prop:bony-estim} by Bony 
that,  if $\gamma+ 2\alpha+ 2 >0$, then the product $f\Xi_2$ is well defined with regularity $\min(\gamma,2\alpha+2,\gamma+2\alpha+2) = 2\alpha +2$ (under our assumptions on $\alpha$ and $\gamma$) 
i.e.  $f\Xi_2 \in H^{2\alpha+2}$.   Moreover, we have 
\begin{align}\label{bound-Xi_2}
||f\Xi_2||_{H^{2\alpha+2}}\lesssim||f||_{H^\gamma}||\Xi_2||_{\mathscr C^{2\alpha+2}}\,. 
\end{align}
  
 From Proposition \ref{prop:commu}, the second term  $\mathscr R(f,\sigma(D)\xi,\xi)$ is also well defined if 
 $\gamma+ 2\alpha+ 2 >0$ and 
 $\mathscr R(f,\sigma(D)\xi,\xi) \in H^{\gamma+ 2\alpha+ 2}$. Using in addition the Schauder's estimate 
 Proposition \ref{prop:multiply},  we obtain the following bound 
 \begin{align}\label{bound-R}
||\mathscr{R}(f,\sigma(D)\xi,\xi)||_{H^{\gamma+2\alpha+2}}\lesssim||f||_{H^{\gamma}}||\sigma(D) \xi ||_{\mathscr C^{\alpha+2}}
||\xi||_{\mathscr C^{\alpha}}\lesssim||f||_{H^\gamma}||\xi||^2_{\mathscr C^{\alpha}}\,. 
\end{align}

A sufficient condition for 
the existence of a positive number $\gamma >0$ such that $\gamma \le \alpha+2$ and $\gamma+ 2\alpha+2 > 0$ is 
 $-4/3 < \alpha <-1$. 
Under this condition, it is sufficient for us to pick any $\gamma$ such that $2/3< \gamma \le \alpha+ 2$. 

The resonating term $f^\sharp \circ \xi$ is also well defined 
because $f^\sharp\in H^{2\gamma},\xi \in \mathscr C^{\alpha}$
with $\alpha+2\gamma>0$ and we have thanks to Bony's Proposition \ref{prop:bony-estim}
the upper bound 
\begin{align}\label{bound-res-sharp}
||f^\sharp\circ\xi||_{H^{2\gamma+\alpha}}\lesssim||\xi||_{\mathscr C^\alpha}||f^\sharp||_{H^{2\gamma}}\,. 
\end{align}

Gathering the three upper bounds \eqref{bound-Xi_2}, \eqref{bound-R} and \eqref{bound-res-sharp}, 
we obtain \eqref{eq:bound-reson}. 
 
\end{proof}

The bound \eqref{eq:bound-reson} on the norm of the resonating term implies the following crucial 
continuity approximation of $f\circ\xi$ with smooth functions. 

\begin{corollary}
Let $-4/3<\alpha<-1, 2/3<\gamma < \alpha+2$ 
and $ \Xi:=(\xi, \Xi_2) \in \mathscr X^\alpha$ and  
$\Xi^\eps:=(\xi_\eps,\xi_\eps\circ \sigma(D)\xi_\eps-c_\eps), \eps >0, c_\eps \in \R$ 
 a family of smooth functions such that 
\begin{align*}
\Xi^\eps\underset{\varepsilon\to 0}{\longrightarrow} \Xi \quad \mbox{ in } \quad \mathscr E^\alpha\,.
\end{align*}Let also $f_\varepsilon$ be a smooth approximation of $f\in \mathscr D_\xi^\gamma$ such that 
\begin{align*}
||f-f_\eps||_{H ^\gamma}+||f_\eps^\sharp-f^\sharp||_{H^{2\gamma}}\underset{\varepsilon\to 0}{\longrightarrow} 0
\end{align*}
where $f_\varepsilon^\sharp$ is the smooth function defined from $f_\eps$ and $\xi_\eps$ as  
$f_\varepsilon^\sharp:= f_\varepsilon- f_\varepsilon \prec \sigma(D) \xi_\eps$.

Then, we have the following continuity approximation of the resonating term $f\circ\xi$  by smooth functions
\begin{align*}
||f_\eps\circ \xi_\eps+c_\eps f-f\circ\xi||_{H^{2\alpha+2}}\underset{\varepsilon\to 0}{\longrightarrow} 0\,. 
\end{align*} 
\end{corollary}

\begin{proof}
Using the bilinearity of $(f,\xi) \to f\circ\xi$ and the trilinearity of $\mathscr{R}$, we easily check that 
\begin{align*}
||f_\eps & \circ\xi_\eps+c_\eps f- f\circ\xi||_{H^{2\alpha+2}}\\ &\lesssim (||f-f_\eps||_{H^\gamma}
+||f_\eps^\sharp-f^\sharp||_{H^{2\gamma}}) 
(1+||\Xi||_{\mathscr E^\alpha})^2+||\Xi_\eps-\Xi||_{\mathscr E^\alpha}(1+||\Xi||_{\mathscr E^\alpha})||f||_{\mathscr D_\xi^{\gamma}}\,,
\end{align*}
which yields the result. 
\end{proof}

We are finally ready to define the linear operator $\mathscr{H}$ on the space $\mathscr{D}_\xi$ of paracontrolled distribution.
\begin{definition}
\label{def:Op}
Let $\alpha\in(-4/3,-1)$, $\frac{-\alpha}{2}<\gamma\le \alpha+2$  and $\Xi=(\xi,\Xi_2)\in\mathscr X^\alpha$. 
We introduce the linear operator $\mathscr H:\mathscr D_\xi^{\gamma}\to H^{\gamma-2}$ such that 
for $f\in \mathscr D_\xi^{\gamma}$, 
\begin{align*}
\mathscr H f :=-\Delta f+f\xi
\end{align*}
where $f\xi$ is defined through the Bony decomposition 
\begin{align*}
f\xi:=f\prec\xi+f\circ\xi+f\succ\xi
\end{align*}
where the resonating term 
$f\circ\xi\in H^{2\alpha+2}$ is defined thanks to Proposition~\ref{prop:reson}. Then $\mathscr H$ can be seen 
as an unbounded operator on $H^{\gamma-2}(\mathbb T_L^2)$ with domain $\mathscr D^{\gamma}_\xi$. 
\end{definition}
\begin{remark}
Note that the operator $\mathscr H$ as introduced in Definition \ref{def:Op} 
depends on the enhancement $\Xi=(\xi,\Xi_2)\in\mathscr X^\alpha(\mathbb T_L)$.
\end{remark}

As explained in the heuristic discussion which motivated Definition \ref{def-D-xi}, we expect 
the eigenfunctions of the linear  operator $\mathscr{H}$  
(which will later be shown to be self-adjoint) 
to belong to the spaces $\mathscr{D}_\xi^\gamma$ 
for $\gamma< \alpha+2$. 
Those eigenfunctions will form an orthonormal basis of $L^2(\T_L^2)$ and it is therefore natural 
to expect the spaces  $\mathscr{D}_\xi^\gamma$ for $\gamma < \alpha+2$ to be dense in $L^2(\T_L^2)$. This is the content 
of the following Lemma.

\begin{lemma}\label{Lemma: density}
Let $\alpha <-1, \xi \in \mathscr C^{\alpha}$ and  $\frac{2}{3}<\gamma<\alpha+2$. Then, 
the space of paracontrolled distributions $\mathscr D_\xi^{\gamma}$ is dense in $L^2(\mathbb T_L^2,\R)$.
\end{lemma}
\begin{proof}

It is sufficient to prove that $\mathscr D_\xi^{\gamma}$ is dense in $C^{\infty}(\mathbb T_L^2)$. 
Let $g\in C^{\infty}(\mathbb T_L^2)$, define the Fourier multiplier  $\sigma_a(k):=-\frac{1}{1+a+|k|^2}$ for $a>0$ and 
consider the map $\Gamma: H^{\gamma}\to H^{\gamma}$ defined as:
\begin{align*}
\Gamma(f)=\sigma_a(D)(f\prec \xi) + g\,. 
\end{align*}
The idea is to first prove that the map $\Gamma$ admits a fixed point $f_a \in \mathscr{D}_\xi^\gamma$. 
It is then straightforward to deduce that $f_a\to g$ in $L^2(\T_L^2,\R)$ from the fact that $ \sigma_a(D) \xi \to 0$ in $\mathscr{C}^{\alpha-\delta}$ 
when $a \to \infty$ for all $\delta>0$.    

We can bound the Fourier multiplier $\sigma_a$: for any $k\in \Z_L^2$, $r\in\mathbb N^2$ and $\theta\in [0,1]$, 
\begin{align}\label{bound-concave}
|\partial^r\sigma_a(k)| \lesssim \frac{1}{a^{1-\theta} (|k|+1)^{2\theta+|r|}}\,. 
\end{align}
Using \eqref{bound-concave}, the Bony estimate \eqref{prop:bony-estim} and the Schauder's inequality 
Proposition \ref{prop:multiply}, we obtain   
\begin{align}
||\Gamma(f_1)-\Gamma(f_2)||_{H^\gamma}&\lesssim a^{(\gamma - (\alpha+2))/2 }
||\xi||_{\mathscr{C}^\alpha}||f_1-f_2||_{H^\gamma}\notag \\ 
 \label{bound-Gamma}
\end{align}

We now fix $a$ large enough such that $a^{(\gamma - (\alpha+2))/2 } ||\xi||_{\mathscr{C}^\alpha} < 1/2$. 
For such an $a$, the  map 
$\Gamma$ is a contraction and therefore admits a unique fixed point $f_a$. 
With the same arguments as above, 
we easily check that 
$\sup_{a\ge 0}||f_a||_{H^\gamma}\lesssim ||g||_{H^\gamma}$ and
we eventually obtain  
\begin{align*}
||f_a-g||_{H^{\gamma}} \lesssim  a^{\frac{1}{2}(\gamma - (\alpha+2))} ||g||_{H^\gamma}  ||\xi||_{\mathscr C^{\alpha}}
\end{align*}
which permits us to conclude that $f_a \to g$ in $H^\gamma$ when $a$ goes to infinity and thus in $L^2(\mathbb T_{L}^2,\R)$.
  
We still have to prove that $f_a\in\mathscr D^{\gamma}_{\xi}$. By definition, $f_n\in H^\gamma$ and we just have to check that 
$f_a - f_a \prec \sigma(D) \xi \in H^{2\gamma}$. We can decompose this function as
\begin{align*}
f_a-f_a\prec \sigma(D) \xi &= \sigma_a(D) (f_a \prec \xi) 
- f_a \prec \sigma_a(D) \xi \\ &+  f_a\prec( \sigma_a(D) - \sigma(D)) \xi +g\,. 
\end{align*} 
From Proposition~\ref{prop:multiply}, we know that the second term $(\sigma_a(D)(f_n\prec\xi)-f_n\prec\sigma_a(D)\xi) \in H^{2\gamma}$ 
with the following upper-bound  
\begin{align*}
||\sigma_a(D)(f_a\prec\xi)-f_a\prec\sigma_a(D)\xi||_{H^{2\gamma}}\lesssim ||f_n||_{H^\gamma}||\xi||_{\mathscr{C}^\alpha}
\end{align*}
For the last term, we need to bound the derivatives of the Fourier multiplier $\sigma_a(k)-\sigma(k)=-\frac{a}{(|k|^2+1)(a+1+|k|^2)}$: 
 \begin{align}
\label{deriv-sigma0-sigma-a}
\left| \partial^m (\sigma_a-\sigma)(k)\right| \lesssim_{a} (1+|k|)^{-4-|m|} 
\end{align}
for $k \in \Z_L^2$ and $m \in \N^2$, where we have used the standard multi-index notation: if $m=(m_1,m_2), \partial^m= \partial^{|m|}/\partial^{m_1} \partial^{m_2}$.  
Then, thanks to the Schauder's estimate Proposition \ref{prop:multiply}, we know that 
the regularity increases by four when applying the operator $\sigma_a(D) - \sigma(D)$ so that  
$( \sigma_a(D) - \sigma(D)) \xi \in \mathscr{C}^{\alpha+ 4}$ with the following upper-bound
\begin{align*}
||f_a\prec (\sigma_a(D)-\sigma(D))\xi||_{H^{\alpha+4}}\lesssim_{a}||f_a||_{H^\gamma}||\xi||_{\mathscr{C}^\alpha}\,. 
\end{align*}
Gathering the above arguments, we can conclude that $f_a \in \mathscr{D}_\xi^\gamma$
 and the Lemma is proved.
\end{proof}

We now construct the resolvent operator 
$\mathcal{G}_a: L^2\to \mathscr{D}_\xi^\gamma$
and establish that it is bounded operator for $a>0$ sufficiently large. 
The following Proposition basically proves that the (punctual) spectrum of $\mathscr{H}$ is almost surely 
bounded from below (in the sense that it has an almost surely finite lowest element). 
The resolvent operator $\mathcal{G}_a$ shall play a crucial role in the next section to define the spectrum 
of $\mathscr{H}$.  

\begin{Proposition}
\label{prop:resolvent}
Let $-4/3<\alpha<-1$, $2/3<\gamma<\alpha+2$, $\rho\in(\gamma-\frac{\alpha+2}{2},1+\frac{\alpha}{2})$ 
and $\Xi=(\xi,\Xi_2)\in\mathscr X^\alpha$.  
Then, there exists $A:=A(||\Xi||_{\mathscr X^\alpha})$  such that 
for all $a\geq A$ and $g\in H^{2\gamma-2}(\mathbb T_L^2)$, the equation
\begin{align*}
(\mathscr H
+a)f=g
\end{align*}
admits a unique solution $f_a\in \mathscr D^\gamma_\xi$. In addition, the maps 
$\mathcal{G}_a:g\in L^2(\mathbb T_{L}^2)\mapsto \mathcal{G}_ag=f_a\in \mathscr D_\xi^\gamma, a\ge A$ is uniformly 
bounded: 
For any $g\in H^{-\delta}$ with $\delta\in[0,2-2\gamma]$ and $a\ge A$,  
\begin{equation}
\label{boundedness-G-a}
||\mathcal{G}_a g ||_{\mathscr D_\xi^\gamma}\lesssim  a^{-1+\gamma+\frac{\delta}{2}}\|g\|_{H^{-\delta}}
\end{equation}
\end{Proposition}

\begin{proof}
Our proof is based on a fixed point argument. For $A>0$, we introduce 
the following auxiliary Banach space \footnote{$C([A,+\infty),H^\gamma(\mathbb T^2))$ denotes the
space of continuous functions $[A,+\infty) \to H^\gamma(\mathbb T_L^2)$. } 
\begin{align*}
\tilde{\mathscr D}_\xi^{\gamma,\rho,A}=\left\{(f_a,f'_a)_{a\ge A} \in C\left([A,+\infty),H^\gamma \right) ^2;\quad  ||(f,f')||_{\tilde{\mathscr D}_\xi^{\gamma,\rho,A}}<+\infty\right\}
\end{align*}
with : 
\begin{align*}
||(f,f')||_{\tilde{\mathscr D}_\xi^{\gamma,\rho,A}}=\sup_{a\geq A}||f_a'||_{H^\gamma}+\sup_{a\geq A}\frac{||f^\sharp_a||_{H^{2\gamma}}}{a^{\rho}}+\sup_{a\geq A}||f_a||_{H^\gamma}.
\end{align*}
For $f\in \tilde{\mathscr D}_\xi^{\gamma,\rho,A}$ and $(\xi,\Xi_2)\in \mathscr{X}^\alpha$, we can define 
the product $f_a\cdot\xi$ for any $a\ge A$ as
\begin{align}
f_a\cdot \xi:=f_a\prec\xi + f_a\circ\xi + f_a\succ\xi\,,\quad f_a\circ\xi:=f_a' \, \Xi_2+\mathscr{R}(f_a',\sigma(D)\xi,\xi)+
f_a^\sharp \circ \xi\,. \label{new-product}
\end{align}

Let us introduce the map 
$\mathcal{M}$ defined for any $(f,f')\in \tilde{\mathscr D}_\xi^{\gamma,\rho,A}$  as
\begin{align*}
\mathcal{M}(f,f'):=(M(f,f'),f)\,,\quad M(f,f')_a:=\sigma_a(D)(f_a\cdot\xi-g)
\end{align*}
where $\sigma_a(k):=-1/(a+|k|^2)$ for $a>2$.
It is sufficient to prove that the map $\mathcal{M}$ admits a unique 
fixed point in the space $\tilde{\mathscr D}_\xi^{\gamma,\rho,A}$.

We first prove that $\mathcal{M}(\tilde{\mathscr D}_\xi^{\gamma,\rho,A}) \subseteq \tilde{\mathscr D}_\xi^{\gamma,\rho,A}$ 
by checking that $M(f,f')_a\in H^{\gamma}$ and that if $(f,f')\in \tilde{\mathscr D}_\xi^{\gamma,\rho,A}$, then we have 
$M(f,f')^{\sharp} := M(f,f') - f \prec\sigma(D)\xi \in H^{2\gamma}$ (using the notations introduced, we have 
$M(f,f')'=f$). 
 
We have the following decomposition 
\begin{align*}
M(f,f')_a= \sigma_a(D)(f_a\prec\xi ) + \sigma_a(D) (f_a\circ \xi + \xi \prec f_a) -\sigma_a(D) g \,. 
\end{align*}
Using the fact that $|\partial^m \sigma_a(k)| \le a^{\theta-1} (|k|+1)^{-2\theta-|m|}$ for any $m\in \Z^2$, 
we use again the Schauder and Bony estimates which give the following upper-bounds 
\begin{align*}
||\sigma_a(D) (f_a\prec \xi)||_{H^\gamma} &\lesssim a^{(\gamma-(\alpha+2))/2} ||f_a||_{H^\gamma} ||\xi||_{\mathscr{C}^\alpha}\,,\\
||\sigma_a(D) (f_a \circ \xi+ f_a\succ \xi)||_{H^\gamma} &\lesssim a^{-1-\alpha/2}  (  ||f_a ||_{H^\gamma} || \xi||_{\mathscr{C}^\alpha}+ || f_a \circ \xi||_{H^{2\alpha+2}} ) \,, \\
|| \sigma_a(D) g ||_{H^\gamma} &\lesssim a^{\gamma/2-1} ||g||_{L^2}\,. 
\end{align*}
To bound $|| f_a \circ \xi||_{H^{2\alpha+2}}$, we write
$f_a \circ \xi = f_a\circ \xi -f_a^\sharp\circ \xi + f_a^\sharp \circ \xi$ and, using also 
\eqref{new-product},
\begin{align}\label{control-diff}
||f_a\circ\xi-f^\sharp_a\circ\xi||_{H^{2\alpha+2}}
&\lesssim ||f'_a||_{H^\gamma}\left(||\Xi_2||_{\mathscr C^{2\alpha+2}}+||\xi||^2_{\mathscr C^{\alpha}}\right) \,.
\end{align}
and
\begin{align}\label{control-f-sharp}
||f^\sharp_a\circ\xi||_{H^{2\gamma+\alpha}}\lesssim ||f^\sharp_a||_{H^{2\gamma}}||\xi||_{\mathscr C^\alpha}
\lesssim a^\rho\, ||(f,f')||_{\tilde{\mathscr D}_\xi^{\gamma,\rho,A}} ||\xi||_{\mathscr{C}^\alpha} \,. 
\end{align}
Gathering those upper-bounds, we obtain
\begin{align}
\label{eq:boundmap}
||M(f,f')_a||_{H^\gamma}
\lesssim a^{\max(\frac{\gamma-(\alpha+2)}{2}, \rho- 1- \frac{\alpha}{2})}||(f,f')||_{\tilde{\mathscr D}^\gamma_{\xi}}(||\xi||_{\mathscr C^\alpha}+||\xi||^2_{\mathscr C^\alpha}+||\Xi_2||_{\mathscr C^{2\alpha+2}})+a^{\gamma/2-1}||g||_{L^2}\,. 
\end{align}
We now prove that 
$M(f,f')^{\sharp} := M(f,f') - f \prec\sigma(D)\xi \in H^{2\gamma}$. 
 For this, we decompose $M(f,f')^{\sharp}$ as follows for $a\ge A$,
\begin{align*}
M(f,f')_a^{\sharp}&=\sigma_a(D)(f_a\prec\xi)-f_a\prec\sigma_a(D)\xi + f_a\prec(\sigma_a-\sigma)(D)\xi \\
&+\sigma_a(D)(f_a\circ\xi -f^\sharp_a \circ\xi)  +\sigma_a(D) (f^\sharp_a \circ\xi) 
+\sigma_a(D) (f\prec\xi)- \sigma_a(D)g\,. 
\end{align*}
Using the Schauder estimate recalled in Proposition \ref{prop:multiply}, we can easily see that 
\begin{align*}
||\sigma_a(D)(f_a\prec\xi)-f_a\prec\sigma_a(D)\xi||_{H^{2\gamma}} &\lesssim a^{(\gamma- (\alpha+2))/2}
||f_a||_{H^\gamma}||\xi||_{\mathscr C^\alpha}\,.
\end{align*}
For the last term, we use the fact that 
$|\partial^m(\sigma_a-\sigma)(k)| \lesssim a^\theta \, |k|^{-2-2\theta-|m|}$ for any $m\in \N^2$ and the Schauder estimate to 
obtain
\begin{align*}
||f_a\prec(\sigma_a-\sigma)(D)\xi||_{H^{2\gamma}}\lesssim 
a^{\gamma-\frac{\alpha+2}{2}}||f_a||_{H^\gamma}||\xi||_{\mathscr C^\alpha}\,. 
\end{align*}
The Schauder estimate and the bound $|\partial^m \sigma_a(k)| \le a^{\theta-1} |k|^{-2\theta-|m|}$ valid for any $m\in \N^2$
permits us to obtain the following upper bounds (using \eqref{control-diff} and \eqref{control-f-sharp})
\begin{align*}
||\sigma_a(D)(f_a\prec \xi ) ||_{H^{2\gamma}} &\lesssim  a^{\gamma- \frac{\alpha+2}{2} }\, ||f_a||_{H^\gamma} 
||\xi||_{\mathscr C^{\alpha}} \,,\\
||\sigma_a(D)(f_a\circ\xi-f^\sharp_a\circ\xi) ||_{H^{2\gamma}} & \lesssim a^{\gamma-(\alpha+2)} \, 
||f'_a||_{H^\gamma} \, \left(||\Xi_2||_{\mathscr C^{2\alpha+2}}+||\xi||^2_{\mathscr C^{\alpha}}\right) \,, \\ 
||\sigma_a(D)(f^\sharp_a\circ\xi) ||_{H^{2\gamma}}& \lesssim a^{ -1-\alpha/2}  ||f_a^\sharp||_{H^{2\gamma}}
||\xi||_{\mathscr C^{\alpha}} \,.  \, 
\end{align*}
We deduce (using also the inequality $||f_a^\sharp||_{H^{2\gamma}} \le a^\rho \, ||(f,f')||_{\tilde{\mathscr D}^{\gamma,\rho,A}_{\xi}}$) that  for all $a>1$, 
\begin{align*}
\label{eq:boundrem}
&a^{-\rho}||M(f,f')_a^\sharp||_{H^{2\gamma}}
\lesssim \\ &||(f,f')||_{\tilde{\mathscr D}^{\gamma,\rho,A}_{\xi}}
\left(a^{\gamma -(\alpha+2)}(||\Xi_2||_{\mathscr C^{2\alpha+2}}+||\xi||^2_{\mathscr C^\alpha})+
a^{\max(\gamma-\frac{\alpha+2}{2}-\rho,- 1- \frac{\alpha}{2})}||\xi||_{\mathscr C^\alpha}\right)+ 
a^{\gamma-1-\rho}||g||_{L^2} \,. 
\end{align*}
Gathering the above inequalities, we can finally establish the following upper-bound for any $a\ge A$, 
 \begin{equation} \label{eq:fix-bound}
 ||\mathcal{M}(f,f')||_{\tilde{\mathscr D}_\xi^{\gamma,\rho,A}}\lesssim A^{-\varrho}||(f,f')||_{\tilde{\mathscr D}^{\gamma,\rho,A}_{\xi}}(||\xi||_{\mathscr C^\alpha}+||\xi||^2_{\mathscr C^\alpha}+||\Xi_2||_{\mathscr C^{2\alpha+2}})+
\sup_{a\ge A} ||f_a||_{H^\gamma} + A^{\gamma/2-1}||g||_{L^2}
 \end{equation}
where $\varrho:=\min ( \frac{\alpha+2-\gamma}{2},1+\frac{\alpha}{2}-\rho , \rho - \gamma+ \frac{\alpha+2}{2})>0$ (under the constraints imposed on the parameters) 
and conclude that 
$\mathcal{M}(f,f')\in \tilde{\mathscr D}_\xi^{\gamma,\rho,A}$.

Using the linearity of $(f,f')\mapsto \mathcal{M}(f,f')$, we get,
for any $(f,f'),(h,h')\in \tilde{\mathscr D}_\xi^{\gamma,\rho,A}$, the inequality
\begin{align}\label{affine-bound-M}
||\mathcal{M}(f,f')-\mathcal{M}(h,h')||_{\tilde{\mathscr D}_\xi^{\gamma,\rho,A}}&\lesssim A^{-\varrho}||(f,f')-(h,h')||_{\tilde{\mathscr D}^{\gamma,\rho,A}_{\xi}}(||\xi||_{\mathscr C^\alpha}+||\xi||^2_{\mathscr C^\alpha}+||\Xi_2||_{\mathscr C^{2\alpha+2}})
\\&+\sup_{a\ge A}  ||f_a-h_a||_{H^\gamma}\,. \notag
\end{align}
The function $\mathcal{M}$ fails to be contracting. This issue can be circumvented 
by iterating $\mathcal{M}$ one more time: Using \eqref{eq:boundmap} and \eqref{affine-bound-M},   
we see that $\mathcal{M}^2(f,f'):=(M(M(f,f'),f),M(f,f'))$ satisfies  
\begin{align*}
||\mathcal{M}^2(f,f')-\mathcal{M}^2(h,h')||_{\tilde{\mathscr D}_\xi^{\gamma,\rho,A}}
&\lesssim A^{-2\varrho} ||(f,f')-(h,h')||_{\tilde{\mathscr D}^{\gamma,\rho,A}_{\xi}}(||\xi||_{\mathscr C^\alpha}+||\xi||^2_{\mathscr C^\alpha}+||\Xi_2||_{\mathscr C^{2\alpha+2}})^2
\\&+\sup_{a\ge A} ||M(f,f')_a-M(h,h')_a||_{H^\gamma}
\\&\lesssim A^{-\varrho}||(f,f')-(h,h')||_{\tilde{\mathscr D}^{\gamma,\rho,A}_{\xi}} (1+||\xi||_{\mathscr C^\alpha}+||\xi||^2_{\mathscr C^\alpha}+||\Xi_2||_{\mathscr C^{2\alpha+2}})^2\,. 
\end{align*}
If $A$ is large enough so that 
$A^{-\varrho}(1+||\xi||_{\mathscr C^\alpha}+||\xi||^2_{\mathscr C^\alpha}+||\Xi_2||_{\mathscr C^{2\alpha+2}})^2\ll1$, 
the map $\mathcal{M}^2:\tilde{\mathscr D}_\xi^{\gamma,\rho,A} \to \tilde{\mathscr D}_\xi^{\gamma,\rho,A}$ is a contraction. The fixed point Theorem insures that 
the map $\mathcal{M}$ admits a unique fixed point $(f,f')\in \tilde{\mathscr D}_\xi^{\gamma,\rho,A}$. 
Observing that $f'=f$,  we deduce that for all $a\geq A$, $f_a\in\mathscr D_{\xi}^{\gamma}$ and 
Eq. \eqref{eq:fix-bound} shows that for any $a\ge A$,  
\begin{equation}
\label{eq:resolvent bound}
||\mathcal{G}_a g||_{\mathscr D^\gamma_\xi} = ||f_a||_{\mathscr D^\gamma_\xi}\lesssim a^{\gamma-1}||g||_{L^2}\,,
\end{equation}
so that the map $\mathcal{G}_a$ is bounded.
\end{proof}

\begin{remark}
Proposition \ref{prop:resolvent} implies that the spectrum of $\mathscr{H}$ 
is contained in the interval $(-A(\Xi),+\infty)$. Moreover using the resolvent identity $\mathcal G_a-\mathcal G_b=(a-b)\mathcal G_a\mathcal G_b$ gives us by induction that 
$$
\|\mathcal G_af\|_{\mathscr D^{\gamma}_{\xi}}\lesssim \exp(C(A(\Xi)-a))\|f\|_{L^2}
$$
for all point $a<A(\Xi)$ which is not contained in the spectrum and where $C$ is a positive constant.
\end{remark}

The maps $\mathcal{G}_a, a\ge A(\Xi)$ are constructed through a fixed point procedure. We can 
deduce continuity bounds with respect to the rough distribution $(\xi,\Xi_2)\in \mathscr{X}^\alpha$ 
(noise) for those maps. 
 
\begin{lemma}
\label{lemma:cont}
Let $\alpha\in(-4/3,-1)$, $\gamma\in(\frac{2}{3},\alpha+2)$.
Then there exist two constants 
$C,\varrho>0$ such that for all $\Xi=(\xi,\Xi_2),\tilde{\Xi}=(\tilde{\xi},\tilde{\Xi}_2)\in\mathscr X^\alpha$  
and $a\geq C(1+||\Xi||+||\tilde{\Xi}||)^{\frac{2}{\varrho}}$, we have the following continuity bound  
\begin{align*}
||\mathcal{G}_a(\Xi)g-\mathcal{G}_a(\tilde{\Xi})g||_{H^\gamma}\lesssim ||g||_{L^2}||\Xi-\tilde{\Xi}||_{\mathscr E^\alpha}(1+||\Xi||_{\mathscr E^\alpha}+||\tilde{\Xi}||_{\mathscr E^\alpha})^2 
\end{align*}
where $\mathcal{G}_a(\Xi):L^2 \to \mathscr{D}_{\xi}^\gamma$ (respectively $\mathcal{G}_a(\tilde\Xi):L^2 
\to \mathscr{D}_{\xi}^\gamma$) is the resolvent operator
associated to the rough distribution $\Xi\in \mathscr{X}^\alpha$ (resp. $\tilde{\Xi}$)
as constructed in Proposition~\ref{prop:resolvent}.
\end{lemma}

\begin{proof}
From Proposition \ref{prop:resolvent}, the operator $\mathcal{G}_a^{\Xi}$ is well defined for 
$a\geq A(||\Xi||_{\mathscr E^\alpha})$ and satisfies  
\begin{align*}
||\mathcal{G}_a(\Xi)g||_{H^\gamma}\lesssim  a^{\gamma/2-1}||g||_{L^2}\,.
\end{align*}
Let $a\geq A(||\Xi||_{\mathscr E^\alpha})+A(||\tilde{\Xi}||_{\mathscr E^\alpha})$ and, to simplify notations, 
set $f_a=\mathcal{G}_a(\Xi)g$ (resp. $\tilde{f}_a:= \mathcal{G}_a(\tilde\Xi) g$).  Using the relations 
$f_a=\sigma_a(D)(f_a\xi-g)$ satisfied by $f_a$ and $\tilde{f}_a$, we deduce that 
\begin{align*}
f_a-\tilde{f}_a= & \sigma_a(D)(f_a\prec\xi-\tilde{f}_a\prec\tilde{\xi}+
f_a\succ\xi -\tilde{f}_a\succ\tilde{\xi}
+f_a\Xi_2-\tilde{f}_a\tilde{\Xi}_2
+f_a^\sharp\circ\xi-\tilde{f}_a^\sharp\circ\tilde{\xi})\\
&+\sigma_a(D)(\mathscr{R}(f_a,\xi,\sigma(D)\xi)-\mathscr{R}(\tilde{f}_a,\tilde{\xi},\sigma(D)\tilde{\xi}))\,. 
\end{align*}
and 
\begin{equation*}
\begin{split}
f_a^{\sharp}-\tilde{f}_a^{\sharp}=&\sigma_a(D)(f_a\succ\xi - \tilde{f}_a\succ\tilde{\xi}
+f_a\Xi_2-\tilde{f}_a\tilde{\Xi}_2+f_a^\sharp\circ\xi-\tilde{f}_a^\sharp\circ\tilde{\xi})\\
&+\sigma_a(D)(\mathscr{R}(f_a,\xi,\sigma(D)\xi)-\mathscr{R}(\tilde{f}_a,\tilde{\xi},\sigma(D)\tilde{\xi}))
+\mathcal{C}_a(f_a,\xi)-\mathcal{C}_a(\tilde{f}_a,\tilde{\xi})
\\&+ f_a\prec(\sigma_a-\sigma)(D)\xi-\tilde{f}_a\prec(\sigma_a-\sigma)(D)\tilde{\xi}
\end{split}
\end{equation*}
where $\mathcal{C}_a(f_a,\xi)=\sigma_a(D)(f_a\prec\xi)-f_a\prec\sigma_a(D)\xi$. 
Therefore using the same argument as in the proof of Proposition~\ref{prop:resolvent} 
(based on the Schauder, Bony and commutator estimates) 
and the bilinearity of $\mathscr C_a$, we obtain
\begin{align*}
&a^{-\rho}||f_a^{\sharp}-\tilde{f}_a^\sharp||_{H^{2\gamma}}
\\&\lesssim  a^{-(1+\alpha/2)}\left(a^{-\rho}||f_a^{\sharp}-\tilde{f}_a^{\sharp}||_{H^{2\gamma}}||\xi||_{\mathscr C^{\alpha}}+a^{-\rho}||\tilde{f}_a^\sharp||_{H^{2\gamma}}||\tilde{\xi}-\xi||_{\mathscr C^{\alpha}}\right)
\\&+a^{-\min(\frac{2+\alpha-\gamma}{2}+\rho,\rho-\gamma+\frac{\alpha+2}{2})}||f_a-\tilde{f}_a||_{H^\gamma}(||\xi||_{\mathscr C^\alpha}+||\xi||^2_{\mathscr C^\alpha}+||\Xi_2||_{\mathscr C^{2\alpha+2}})
\\&+a^{-\min(\frac{2+\alpha-\gamma}{2}+\rho,\rho-\gamma+\frac{\alpha+2}{2})}||\tilde{f}_a||_{H^\gamma}(||\xi-\tilde{\xi}||_{\mathscr C^\alpha}(1+||\xi||_{\mathscr C^\alpha})+||\tilde{\xi}||_{\mathscr C^\alpha})+||\Xi_2-\tilde{\Xi}_2||_{\mathscr C^{2\alpha+2}})\,. 
\end{align*}
Then choosing $a$ large enough such that $a^{-(1+\alpha/2)}||\xi||\ll1$ and using the fact 
that $||\tilde{f}_a||_{\mathscr D_{\tilde{\xi}}^{\gamma}}\lesssim a^{\gamma/2-1} ||g||_{L^2}$ we obtain that 
\begin{equation*}
\begin{split}
&a^{-\rho}||f_a^{\sharp}-\tilde{f}_a^\sharp||_{H^{2\gamma}}\lesssim  a^{\gamma/2-1}||g||_{L^2}||\xi-\tilde{\xi}||_{\mathscr C^{\alpha}}
\\&+a^{-\min(\frac{2+\alpha-\gamma}{2}+\rho,\rho-\gamma+\frac{\alpha+2}{2})}||f_a-\tilde{f}_a||_{H^\gamma}||\Xi||_{\mathscr E^\alpha}(1+||\xi||_{\mathscr C^\alpha})
\\&+a^{-\min(\frac{2+\alpha-\gamma}{2}+\rho,\rho-\gamma+\frac{\alpha+2}{2})+\gamma/2-1}||g||_{H^\gamma}(||\xi-\tilde{\xi}||_{\mathscr C^\alpha}(1+||\xi||_{\mathscr C^\alpha}+||\tilde{\xi}||_{\mathscr C^\alpha})+||\Xi_2-\tilde{\Xi}_2||_{\mathscr C^{2\alpha+2}}).
\\&\lesssim a^{\gamma/2-1}||g||_{L^2}||\Xi-\tilde{\Xi}||_{\mathscr E^\alpha}(1+||\xi||_{\mathscr C^\alpha}+||\tilde{\xi}||_{\mathscr C^\alpha})
\\&+a^{-\min(\frac{2+\alpha-\gamma}{2}+\rho,\rho-\gamma+\frac{\alpha+2}{2})}||f_a-\tilde{f}_a||_{H^\gamma}||\Xi||_{\mathscr E^\alpha}(1+||\xi||_{\mathscr C^\alpha})\,.
\end{split}
\end{equation*}
Thus it suffices to bound the term $||f_a-f_a||_{H^\gamma}$ which can be treated easily with the same argument and we 
finally obtain 
\begin{equation*}
\begin{split}
&||f_a-\tilde{f}_a||_{H^\gamma}\lesssim a^{-\min(1+\alpha/2,\frac{\alpha+2-\gamma}{2})}||f_a-\tilde{f}_a||_{H^\gamma}(||\xi||_{\mathscr C^\alpha}+||\xi||^2_{\mathscr C^\alpha}+||\Xi_2||_{\mathscr C^{2\alpha+2}})
\\&+a^{-\min(1+\alpha/2,\frac{\alpha+2-\gamma}{2})-1+\gamma}||\tilde{f}_a||_{H^\gamma},. 
\end{split}
\end{equation*}
It is now plain to deduce that for $a$ large enough
$$
||f_a-\tilde{f}_a||_{H^\gamma}\lesssim a^{-\min(1+\alpha/2,\frac{\alpha+2-\gamma}{2})-1+\gamma}||\tilde{f}_a||_{H^\gamma}
$$
and then since $\|f_a^{\sharp}-\tilde{f}_a^\sharp\|_{H^{2\gamma}}$ is controlled by $||f_a-\tilde{f}_a||_{H^\gamma}$, 
the needed result follows.
\end{proof}
\begin{remark}
Another important remark we shall use later 
is that the resolvent $\mathcal G_a$ allows us to describe the space $\mathscr D^\gamma_{\xi}$, indeed if   
$\alpha,\gamma$, $\Xi$ and $A(\Xi)$ satisfy the assumptions of Proposition~\ref{prop:resolvent}, then 
for all $a\geq A$ we have that $\mathcal G_aH^{2\gamma-2}=\mathscr D_{\xi}^{\gamma}$. 
\end{remark}

\subsection{Restriction on a space of strongly paracontrolled distribution}
So far, we have constructed the operator $\mathscr{H}$ on the space $\mathscr{D}_\xi^\gamma$ with values 
in the Sobolev space $H^{\gamma-2}$. The space $\mathscr{D}_\xi^\gamma$ depends on
the realization of the noise $\xi$ and our construction also uses an enhancement $(\xi,\Xi_2)\in \mathscr X^\alpha$ 
($\Xi_2$ is an additional input) 
of the noise $\xi\in \mathscr{C}^\alpha$.  At this point, the function $g:= \mathscr{H} f \in H^{\gamma-2}$
is a distribution with regularity $\gamma-2 \le \alpha$.

In this subsection, we define a smaller space $\mathscr{D}_\Xi^\gamma$ (continuously embedded in $\mathscr{D}_\xi^\gamma$)
such that, when restricted to the subspace $\mathscr{D}_\Xi^\gamma$, the operator $\mathscr{H}$, as constructed in Definition 
\ref{def:Op}, takes values in $L^2(\mathbb T_L^2)$. 
We shall also establish that the resolvent $\mathcal{G}_a:L^2 \to \mathscr{D}_\Xi^\gamma$ 
is a bounded (continuous) operator. 

Before giving the definition of the subspace $\mathscr{D}_\Xi^\gamma$, let us  analyze the regularity of the eigenvectors of $\mathscr{H}$ 
such that $\mathscr{H}f=\lambda f$ for some $\lambda\in \R$. 
 
First, if $f\in \mathscr{D}_\xi^\gamma$ with $2/3<\gamma < \alpha+2$, it is easy to see that 
\begin{align*}
-\Delta f=-\Delta f^\sharp-2\, \nabla f\prec\nabla (\sigma(D) \xi) +(1-\Delta)f\prec \sigma(D) \xi -f\prec\xi+
\end{align*}
where $\nabla$ is the gradient. It follows that  
\begin{align*}
\mathscr H f =-\Delta f^\sharp-2\nabla f\prec\nabla(\sigma(D) \xi)+ (1-\Delta) f\prec \sigma(D) \xi+f\Xi_2+\mathscr{R}(f,\xi,\sigma(D) \xi)
+f^\sharp\circ\xi+f\succ\xi\,. 
\end{align*}
Checking the regularity of each term in this later expression, we see that $\mathscr H f\in H^{2\gamma-2}$. We note that 
$2\gamma-2 \in (-2/3,0)$ if $2/3<\gamma<\alpha+2$. 

Now, coming back to the eigenvalue problem, we see that, if $f\in \mathscr{D}_\xi^\gamma$ (for $\gamma<\alpha+2$) 
is an eigenvector of 
$\mathscr{H}$, then the associated remainder $f^\sharp$ satisfies 
\begin{align*}
(1-\Delta) f^\sharp=\lambda f+ 2\nabla f\prec\nabla (\sigma(D) \xi) -(1-\Delta)f\prec \sigma(D) \xi -f\succ\xi-f\Xi_2-\mathcal{ R}(f,\xi,\sigma(D) \xi)-f^\sharp\circ\xi+f^{\sharp}\,. 
\end{align*}
Rewriting this equality with Fourier multipliers, we get 
\begin{align*}
 f^\sharp= \sigma(D) \left(   2\nabla f\prec\nabla (\sigma(D) \xi)-(1-\Delta) f\prec \sigma(D) \xi -f\succ\xi-f\Xi_2 \right) + f^\flat
\end{align*}
where 
\begin{align*}
f^\flat= \sigma(D)(\lambda f- \mathscr{R}(f,\xi,\sigma(D)\xi ) - f^\sharp\circ\xi+f^{\sharp})\,. 
\end{align*}
It is easy to check from this last expression that $f^\flat \in H^{3\gamma} \subseteq H^2$ (where we have noticed that $3\gamma >2$). 
As a conclusion of this analysis, we have proved that, if $f\in \mathscr{D}_\xi^\gamma$ is an eigenvector 
of $\mathscr{H}$, then it admits the following paracontrolled expansion  
\begin{align*}
f = f\prec \sigma(D) \xi+  B(f,\Xi) + f^\flat 
\end{align*}
where $B$ is the a bilinear form defined as
\begin{align*}
B: (f,\Xi) \in H^{\gamma} \times \mathscr X^\alpha \longrightarrow 
\sigma(D)(2\nabla f\prec\nabla (\sigma(D) \xi)-(1-\Delta)f\prec \sigma(D) \xi -f\succ\xi - f\Xi_2) \in H^{2\gamma} 
\end{align*}
and where the remainder term $f^\flat\in H^{3\gamma}$. 

This discussion motivates the following definition for the subspace 
$\mathscr{D}_\Xi^\gamma \subseteq \mathscr{D}_\xi^\gamma$ 
which is such that $\mathscr{D}_\Xi^\gamma$ contains the eigenvectors of $\mathscr{H}$. 

\begin{definition}\label{def-paracontrolled-2}
Let $-4/3<\alpha<-1$, $-\frac{\alpha}{2}<\gamma \le \alpha+2$  and $\Xi=(\xi,\Xi_2)\in\mathscr X^\alpha$, 
we define the space of strong paracontrolled distribution
\begin{align*}
\mathscr D_{\Xi}^{\gamma}:=\left\{f\in H^\gamma;\quad f^\flat:=f- f\prec \sigma(D)\xi - B(f,\Xi) \in H^2\right\}
\end{align*}
 equipped with the following scalar product 
\begin{align*}
\langle f,g\rangle_{\mathscr D_{\Xi}^{\gamma}}=\langle f,g\rangle_{H^\gamma}+\langle f^\flat,g^\flat\rangle_{H^2}\,. 
\end{align*}
\end{definition}
\begin{remark}
Of course, the Hilbert space $\mathscr D_{\Xi}^{\gamma}$ is continuously 
embedded in $\mathscr D_\xi^\alpha$ as $f^\sharp:= f-f\prec\sigma(D)\xi= 
B(f,\Xi) + f^\flat\in H^{2\gamma}$ for $f\in \mathscr D_{\Xi}^{\gamma} $. 
It is also clear that actually $\mathscr D_{\Xi}^{\gamma} \subset H^{\alpha+2}$ 
so that the spaces $\mathscr D_{\Xi}^{\gamma}, \gamma \in [2/3 ;\alpha+2] $ are all equal. 
We drop the super script $\gamma$ for the space 
\begin{align*}
\mathscr D_{\Xi}=\left\{f\in H^{\alpha+2} \;, \quad f^\flat:=f- f\prec \sigma(D)\xi - B(f,\Xi) \in H^2\right\}\,. 
\end{align*}
If $\xi$ is a smooth function, then $\mathscr D_{\Xi}=H^2$. 
\end{remark}

The following technical estimates on the bilinear form $B$ will be useful in the sequel. 
\begin{lemma}
\label{lemma: bbound}
Let $-4/3<\alpha<-1$, $2/3<\gamma\le \alpha+2$.
For $a\ge 2$,  we define the bilinear form $B_a$ such that for $(f,\Xi)\in H^\gamma\times\mathscr X^\alpha$, 
\begin{align*}
B_a(f,\Xi):=\sigma_a(D)(2\nabla f\prec\nabla (\sigma(D) \xi)-(1-\Delta)\Delta f\prec \sigma(D) \xi -f\succ\xi - f\Xi_2)
\end{align*}
where for $a>2$,  $\sigma_a(k):=\frac{1}{a+|k|^2}$. Then we have the following estimates 
\begin{align*}
||B_a(f,\Xi)||_{H^{2\gamma}}\lesssim a^{-\frac{2-\gamma+\alpha}{2}}||f||_{H^\gamma}||\Xi||_{\mathscr E^\alpha},\quad||(B-B_a)(f,\Xi)||_{H^{2\gamma+2}}\lesssim a||f||_{H^\gamma}||\Xi||_{\mathscr E^\alpha}
\end{align*}
\end{lemma}

\begin{proof}
Towards the first estimate, it suffices to notice that, for any $\theta\in (0,1)$ and $m\in \mathbb N^2$, $|\partial^m\sigma_a(k)|\lesssim a^{-(1-\theta)}(|k|+1)^{-2\theta-|m|}$ and the result follows from the Schauder estimate Proposition~\ref{prop:multiply} applied for 
the particular value $\theta=(\gamma-\alpha)/2$.   

The second inequality is proved with the same method recalling Eq. \eqref{deriv-sigma0-sigma-a}
\begin{align*}
|\partial^m(\sigma-\sigma_a)(k)|\lesssim a\, (|k|+1)^{-4-|m|}
\end{align*}
for all $m\in\mathbb N^2$ and using the Schauder estimate again.
\end{proof}

We can finally prove that, when restricted to the space  $\mathscr{D}_\Xi$
 of strongly paracontrolled distributions, the linear operator $\mathscr{H}$ as defined in Definition \ref{def:Op}, 
 takes values in $L^2$. 
 \begin{Proposition}
If $f\in \mathscr D_{\Xi}$, then $\mathscr{H} f \in L^2(\T_L^2)$.  
\end{Proposition}
\begin{proof}
If $f\in \mathscr D_{\Xi}$, we have 
\begin{align}
\mathscr{H} f  &=(1 -\Delta)(f\prec \sigma(D)\xi) +(1- \Delta) B(f,\Xi)- \Delta f^\flat +-f\prec\sigma(D)\xi-B(f,\Xi)+\xi f\\
&=  - \Delta f^\flat-f\prec\sigma(D)\xi-B(f,\Xi)
\label{eq:expansion-2}
\\&+\mathscr{R}(f,\xi,\sigma(D)\xi) + (B(f,\Xi)+f^\flat) \circ \xi \in H^{3\alpha+4} \subset L^2(\T^2)\,. \notag
\end{align}
\end{proof}

In the following Proposition, we prove that any  element $f\in \mathscr{D}_\Xi$ 
can be approximated with smooth functions $(f_\varepsilon)_{\varepsilon>0}$ in $H^2$.  
\begin{Proposition}\label{prop:interpolation}
Let $-4/3<\alpha<-1$, $-\frac{\alpha}{2}<\gamma < \alpha+2$ 
and $\Xi:=(\xi, \Xi_2),\tilde \Xi=(\tilde \xi,\tilde \Xi_2) \in \mathscr X^\alpha$   
Then, for any $f\in \mathscr{D}^\alpha_{\Xi}$, there exists a function 
$g\in\mathscr D_{\tilde\Xi}^\alpha$  such that 
\begin{align*}
||f-g||_{H^{\gamma}} + || f^\flat - g^\flat||_{H^2}\lesssim
||f||_{H^\gamma}(1+||\tilde\Xi||_{\mathscr E^\alpha})||\Xi-\tilde\Xi||_{\mathscr E^\alpha}
\end{align*}
where $ g^\flat:=g -g \prec \sigma(D)\tilde\xi - B(f,\tilde\Xi)$. In particular, if $(\xi_\eps,c_\eps)\in C^{\infty}(\mathbb T^2)\times\mathbb R$ is a sequence  such that $(\xi_\eps,\xi_\eps\circ\sigma(D)\xi_\eps-c_\eps)$ converges 
to $\Xi$ in $\mathscr E^\alpha$ as $\eps\to 0$, then there exists a sequence $(f_\eps)$ in $H^2$ such that 
$$
||f_\eps-f||_{H^\gamma}+||f_\eps^\flat-f^\flat||_{H^2}\underset{\eps\to0}{\to}0
$$   
with $f_\eps^\flat:=f_\eps-f_\eps\prec\sigma(D)\xi_\eps-B(f_\eps,(\xi_\eps,\xi_\eps\circ\sigma(D)\xi_\eps-c_\eps))$.
\end{Proposition}

\begin{proof}
Let $f\in \mathscr{D}_\Xi$ and define the map $\Gamma: H^{\gamma}\to H^{\gamma}$ such that  
\begin{align*}
\Gamma (g) := g\prec \sigma_a(D)\tilde\xi + B_a(g,\tilde\Xi) +f - f\prec \sigma_a(D)\xi - B_a(f,\Xi)
\end{align*}
where for $a>2$,  $\sigma_a(k):=\frac{1}{a+|k|^2}$ and 
\begin{align*}
B_a(g,\tilde \Xi)&= \sigma_a(D)(2\nabla g\prec\nabla (\sigma(D) \tilde\xi) +\Delta g \prec \sigma(D) \tilde\xi -g\succ\tilde\xi - 
g\tilde \Xi_2)\,, \\ 
B_a(f,\Xi) &= \sigma_a(D)(2\nabla f\prec\nabla (\sigma(D) \xi) +\Delta f \prec \sigma(D) \xi -f\succ\xi - 
f\Xi_2)  \,. 
\end{align*}
Using bilinearity and the Schauder's estimate Proposition \ref{prop:multiply}, we have the following bound 
for any $g_1,g_2\in H^\gamma$, 
\begin{align*}
||\Gamma(g_1) - \Gamma(g_2)||_{H^{\gamma}}& \lesssim a^{(\gamma-\alpha-2)/2}  (1+ ||\tilde\Xi||_{\mathscr{E}^\alpha}) 
||g_1-g_2||_{H^\gamma}\\
&\lesssim a^{(\gamma-\alpha-2)/2}  (1+ ||\tilde\Xi||_{\mathscr{E}^\alpha}) 
||g_1-g_2||_{H^\gamma}\,. 
\end{align*}
For $a$ sufficiently large such that $a^{(\gamma-\alpha-2)/2}  (1+ ||\tilde \Xi||_{\mathscr{E}^\alpha})<1$, 
the map $\Gamma$ is a contraction and therefore admits a unique fixed point 
$a$ such that  
\begin{align*}
g & = g \prec \sigma_a(D)\tilde\xi + B_a(g,\tilde\Xi) 
+f - f \prec \sigma_a(D)\xi - B_a(f,\Xi)\\ 
& =g \prec \sigma_a(D)\tilde\xi + B_a(g,\tilde\Xi) 
+f ^\flat +f\prec(\sigma-\sigma_a)(D)\xi+ B(f,\Xi) - B_a(f,\Xi)  \,. 
\end{align*}
Recalling that $\sigma(k)-\sigma_a(k)= \frac{a}{(|k|^2+1)(a+|k|^2)}$ and Eq. \eqref{deriv-sigma0-sigma-a}, 
it is easy to check that 
\begin{align*}
B(f,\Xi)- B_a(f,\Xi) = (\sigma(D) - \sigma_a(D)) (2\nabla f\prec\nabla (\sigma(D) \xi) +\Delta f \prec \sigma(D) \xi -f\succ\xi - 
f\Xi_2) \in H^{2\gamma+2}.
\end{align*} 
Therefore, using also the definition of $f\in \mathscr{D}_\Xi$ which implies that $f^\flat\in H^2$, we deduce that  $g \in \mathscr D_{\tilde \Xi}^\alpha$.  

Now, using the Bony and Schauder's Propositions \ref{prop:bony-estim} and \ref{prop:multiply}, we obtain 
\begin{align*}
||f-g ||_{H^{\gamma}}  &\lesssim ||( g -f) \prec \sigma_a(D) \tilde\xi ||_{H^\gamma} + 
||f \prec \sigma_a(D) (\tilde \xi -\xi) ||_{H^\gamma} \\ & + 
||B_a(g-f,\tilde\Xi) ||_{H^\gamma} + ||B_a(f,\tilde\Xi-\Xi) ||_{H^\gamma}  \\
&\lesssim a^{(\gamma-\alpha-2)/2} (1+||\Xi||_{\mathscr{E}^\alpha})
||g-f||_{H^\gamma} + ||f||_{H^\gamma} ||\tilde\Xi - \Xi ||_{\mathscr{E}^\alpha}.
\end{align*}
Fixing  $a$ sufficiently large, we have the bound 
\begin{align*}
||f-g||_{H^{\gamma}}\lesssim
||f||_{H^\gamma} || \Xi - \tilde\Xi ||_{\mathscr{E}^\alpha}
 \,.  
\end{align*}
Observing that :
\begin{align*}
f^\flat-g^\flat &=g\prec(\sigma_a-\sigma)(D)\tilde\xi - f \prec(\sigma_a-\sigma)(D)\xi
+ B_a(g,\tilde\Xi)- B_a(f,\Xi) + B(f,\Xi)- B(f,\tilde\Xi) 
\end{align*}
Recalling that $(\sigma_a(D)-\sigma(D)) \xi \in \mathscr{C}^{\alpha+4}$ (see end of the proof of Lemma \ref{density}),
we can finally prove that :
$$
||f^\flat-g^\flat||_{H^2}\lesssim ||f-g||_{H^\gamma}||\tilde\Xi||_{\mathscr E^\alpha}+||f||_{H^\gamma}||\Xi-\tilde\Xi||_{\mathscr E^\alpha}\lesssim||f||_{H^\gamma}(1+||\tilde\Xi||_{\mathscr E^\alpha})||\Xi-\tilde\Xi||_{\mathscr E^\alpha}
$$

\end{proof}

\begin{Proposition}
Let $-4/3<\alpha<-1$, $2/3< \gamma\le \alpha+2$, $\Xi=(\xi,\Xi_2)$,  $\tilde{\Xi}=(\tilde{\xi},\tilde{\Xi}_2)\in\mathscr X^\alpha$. 
Following Definition \ref{def:Op}, we introduce the two linear operators $\mathscr{H}^\Xi:= -\Delta+\xi$ and 
$\mathscr{H}^{\tilde{\Xi}}:= -\Delta+\tilde{\xi}$  
with respective domains $\mathscr{D}_{\Xi}$ and $\mathscr{D}_{\tilde\Xi}$. 
We have the following continuity upper-bound for any $f\in \mathscr{D}_{\Xi}$,   $g\in    \mathscr{D}_{\tilde{\Xi}}$,
\begin{align}\label{eq:continuity}
||\mathscr{H}(\Xi) f-\mathscr{H}(\tilde{\Xi}) g||_{L^2}\lesssim &
\, \left( || f-g||_{H^{\gamma}} + ||f^\flat - g^\flat||_{H^2} + ||\Xi-\tilde{\Xi}||_{\mathscr{E}^\alpha}\right) \\
&\times (1+ ||\tilde{\Xi}||_{\mathscr{E}^\alpha} )^2 \times (1+ ||g||_{H^{\alpha+2}} + ||g^\flat|| _{H^2}) \notag\,. 
\end{align}
Moreover, the operator $\mathscr{H}:\mathscr{D}_{\Xi} \to L^2(\mathbb T_L)$ is symmetric in $L^2$ in the sense that, 
for any $f,g\in \mathscr{D}_{\Xi} $, we have
\begin{align*}
\langle \mathscr{H} f,g\rangle_{L^2}=\langle f,\mathscr{H} g\rangle _{L^2} \,.
\end{align*}
\end{Proposition}

\begin{proof}
The estimate \eqref{eq:continuity} follows from the expansion obtained in 
\eqref{eq:expansion-2} for $\mathscr{H}^1 f$ and $\mathscr{H}^2 g$, from
the bilinearity of the maps $\mathscr{R}$ and $B$, from Proposition \ref{prop:commu} (continuity bound for $\mathscr{R}$) 
and Proposition \ref{lemma: bbound} (continuity bound for $B$) 
and from the Schauder's and Bony's estimates Propositions \ref{prop:multiply} 
and \ref{prop:bony-estim} for paraproducts. 

Towards the symmetry of the operator $\mathscr{H}$ 
defined on the domain $\mathscr{D}_\Xi$ for $\Xi\in \mathscr X^\alpha$, 
we introduce a family of smooth functions  
 $\Xi_\varepsilon:= (\xi_\varepsilon,\xi_\varepsilon\circ 
\sigma(D)\xi_\eps-c_\eps), \eps>0$ such that $\Xi_\varepsilon$ converges to $\Xi$ in $\mathscr{E}^\alpha$ (this family exists 
by definition of the space $\mathscr X^\alpha$). Then, we know from Proposition \ref{prop:interpolation} that there exists 
a family of functions $f_\eps,\eps>0$ in $H^2$
 such that 
 \begin{align}\label{conv-f-eps}
||f-f_\eps||_{H^\gamma} + ||f^\flat -f_\eps^\flat||_{H^2} \underset{\eps\to 0}{\longrightarrow} 0
\end{align}
for any $2/3<\gamma< \alpha+ 2$ 
where $f_\eps^\flat:= f_\eps-f_\eps\prec\sigma(D)\xi_\eps-B(f,\Xi_\eps)$. 

Thanks to the smoothness of $\xi_\eps$, it is straightforward to define 
the operator $\mathscr{H}_\eps:=-\Delta+\xi_\eps$ on the domain $H^2$.  
 
From the estimates \eqref{eq:continuity} and \eqref{conv-f-eps}, we deduce that 
$\mathscr{H_\eps}f_\eps$ converges towards $\mathscr{H} f$ in $L^2$ for any $f\in \mathscr{D}_\Xi$. 
If $g\in \mathscr{D}_\Xi$ and $(g_\eps)_{\eps >0}\in H^2$ are such that  
$\mathscr{H_\eps}g_\eps\to\mathscr{H} g$ in $L^2$ , it is plain to see that 
\begin{align*}
\langle \mathscr{H}_\eps f_\eps,g_\eps\rangle_{L^2}=\langle f_\eps,\mathscr{H}_\eps g_\eps \rangle_{L^2}\,.  
\end{align*}
We obtain the symmetry of $\mathscr{H}$ by sending $\eps\to 0$ in the last equality.
\end{proof}

We now come back to the study of the resolvent operator $\mathcal{G}_a$
by establishing that $\mathcal{G}_a$ takes values in the space of strongly paracontrolled distributions 
$\mathscr{D}_\Xi$. We also prove that the resolvent $\mathcal{G}_a$ is a 
self-adjoint and compact operator $L^2\to L^2$. 
Proposition \ref{prop:resolvent-2} will be used later to establish the main properties 
of the spectrum of the operator $\mathscr{H}$.

\begin{Proposition}
\label{prop:resolvent-2}
Let  $-4/3 <\alpha< -1$, $2/3 < \gamma<\alpha+2$ and $A:=A(||\Xi||_{\mathscr{X}^\alpha})$ as
 introduced in Proposition~\ref{prop:resolvent}.  

Then, for all $a\geq A$, the operator $\mathscr{H}+a:\mathscr D_{\Xi} \to L^2$ is invertible 
with inverse $\mathcal{G}_a:L^2 \to \mathscr D_{\Xi}$. 
In addition, the operator $\mathcal{G}_a:L^2\to L^2$ is bounded, self-adjoint and compact.  
\end{Proposition}

\begin{proof}
Let $g\in L^2$ and set $f_a:=\mathcal{G}_a g\in\mathscr D_\xi^\gamma$. By definition,  
$(\mathscr{H}+a)f_a=g$
and it is easy to check that 
\begin{align*}
(1-\Delta) f_a^\sharp=g-a f_a + 2\nabla f_a \prec\nabla (\sigma(D) \xi)  +\Delta f_a\prec \sigma(D) \xi -f_a\succ\xi-f_a\Xi_2-\mathscr{R}(f_a,\xi,\sigma(D) \xi)-f_a^\sharp\circ\xi+f_a^{\sharp} \,. 
\end{align*}
It follows that  
\begin{align}\label{f_a-flat}
f_a^\flat:= f_a^\sharp-B(f_a,\Xi)=
\sigma(D)( g-a f_a - \mathscr{R}(f_a,\xi,\sigma(D) \xi)-f_a^\sharp\circ\xi+f^{\sharp_a}-f_a\prec\sigma(D)\xi)
\in H^2
\end{align}
so that $f_a\in\mathscr D_{\Xi}$. 

We now prove that the map $\mathcal{G}_a:L^2\to\mathscr D_{\Xi}^\gamma$ is bounded. 
We have 
\begin{align*}
||\mathcal{G}_a g||_{\mathscr{D}_\Xi^\gamma} &=  ||\mathcal{G}_a g||_{H^\gamma} 
+ ||(\mathcal{G}_a g)^\flat ||_{H^2}\\
&\le ||\mathcal{G}_a g||_{\mathscr{D}_\xi^\gamma} + ||f_a^\flat ||_{H^2}\\
&\le a^{\gamma/2-1} ||g||_{L^2} +  ||f_a^\flat ||_{H^2}\,
\end{align*}
where we have used the boundedness of $\mathcal{G}_a:L^2 \to\mathscr D_{\xi}^\gamma$ (see \eqref{boundedness-G-a}). It suffices to upper-bound the $H^2$ norm of the remainder $f_a^\flat$. 
Using \eqref{f_a-flat}, we have 
\begin{align*}
||f_a^\flat||_{H^2} &\lesssim 
(a ||f_a||_{H^\gamma}+||f^\sharp_a||_{H^{2\gamma}})(1+||\xi||_{\mathscr C^\alpha})^2+||g||_{L^2} \\
&\lesssim a||f_a||_{\mathscr D_\xi^\gamma}(1+||\xi||_{\mathscr C^\alpha})^2+||g||_{L^2}\\
&\lesssim \left(1+ a^{ \gamma/2}  (1+||\xi||_{\mathscr C^\alpha})^2\right) ||g||_{L^2} 
\end{align*}
and we can finally deduce that for $a>2$,
\begin{align*}
||\mathcal{G}_a g||_{\mathscr{D}_\Xi^\gamma} \lesssim 
\left(a^{\gamma/2-1}+ a^{ \gamma/2}  (1+||\xi||_{\mathscr C^\alpha})^2\right) ||g||_{L^2} \,. 
\end{align*}

The fact that the resolvent operator $\mathcal{G}_a:L^2\to L^2$ is self-adjoint follows from the symmetry 
of $\mathscr{H}$. 
The resolvent operator  $\mathcal{G}_a: L^2 \to L^2$ is the composition 
of the bounded operator $\mathcal{G}_a:L^2\to H^\gamma$ ($\mathcal{G}_a:L^2\to \mathscr{D}_\Xi^\gamma$ is bounded and $||\cdot||_{H^\gamma}\le ||\cdot||_{ \mathscr{D}_\Xi^\gamma}$)
with the compact injection
operator $i: H^\gamma\to L^2$ (this fact follows from Rellich-Kondrachov Theorem, see Appendix~\ref{lemma:rellich}
for a reminder) and therefore is a compact operator.    
\end{proof}
\begin{remark}
Before proceeding with the study of the operator $\mathscr H$,  let us point out that the space of strong paracontrolled distributions $\mathscr D_{\Xi}$ is dense in the space $\mathscr D^{\gamma}_{\xi}$. Indeed, we have $\mathscr D_\xi^\gamma=\mathcal G_aH^{2\gamma-2}$ and $\mathscr D_{\Xi}=\mathcal G_aL^2$ and the result follows from the fact that $H^{2\gamma-2}$ is dense in $L^2$.
\end{remark}
A simple application of the spectral Theorem to the operator $\mathcal{G}_a$ yields the following properties 
for the spectrum of the operator $\mathscr{H}$.  
\begin{corollary}\label{spectrum-general-noise}
Let  $\alpha\in(-4/3,-1)$ and $\Xi=(\xi,\Xi_2)\in\mathscr X^\alpha$. 
Then the operator $\mathscr{H}$ is self-adjoint has a discrete spectrum (only pure points) 
$(\Lambda_n)_{n\geq1}$ which is such that 
\begin{align*}
\Lambda_1\le\Lambda_2\le\cdots\le\Lambda_n \le\cdots 
\end{align*}
with no accumulation points except in $+\infty$. In addition, we have 
\begin{align*}
L^2=\bigoplus_{n\in\mathbb N}\text{Ker }(\Lambda_n-\mathscr{H})
\end{align*} 
and for any $n$, the dimension of the subspace $\text{Ker }(\Lambda_n-\mathscr{H})$ is finite. 
The eigenvalues $\{\Lambda_n,n\in \N\}$ satisfy the following minimax principle
\begin{equation}
\label{eq:ritz-formula}
\Lambda_n=\inf_{F}\sup_{\substack{f\in F;\\||f||_{L^2}=1}}\langle\mathscr{ H}(\Xi) f,f\rangle_{L^2}
\end{equation}
where $F$ ranges over all n-dimensional subspaces of $\mathscr D_{\Xi}$. 
\end{corollary}
The min-max principle~\eqref{eq:ritz-formula} is a very useful variational tool to study the eigenvalues.  
however the fact that the supremum is taken on the linear space $\mathscr D_{\Xi}$ can give a rise to a very complicated computation. As pointed out previously the operator $\mathscr H$ can be defined on the more simpler space $\mathscr D^{\gamma}_{\xi}$, the problem then is that the operator has value in the space $H^{2\gamma-2}(\mathbb T^{2}_L)$ and not in $L^2(\mathbb T^2_L)$. Fortunately this fact turns out to be sufficient to make sense of the quadratic form $\langle \mathscr H(\Xi)f,f\rangle$ by duality. Namely if $f\in\mathscr D^{\gamma}_{\xi}$ and $\gamma>2/3$ then the following bound : 
$$
|\langle\mathscr H(\Xi)f,f\rangle|\leq \|\mathscr H(\Xi)f\|_{H^{2\gamma-2}}\|f\|_{H^{\gamma}}\lesssim\|f\|_{\mathscr D^{\gamma}_{\xi}}\|f\|_{H^{\gamma}}(1+\|\xi\|_{\mathscr C^{\alpha}}+\|\xi\|^2_{\mathscr C^{\alpha}}+\|\Xi_2\|_{\mathscr C^{2\alpha+2}})
$$  
hold.
Now the point is that we can replace the space $\mathscr D_{\Xi}$ by $\mathscr D^{\gamma}_{\xi}$ in the min-max equation~\ref{eq:ritz-formula}. More precisely we have the following statement: 
\begin{lemma}\label{lemma:min-max}[min-max principle]
Given $2/3<\gamma<\alpha+2$ we have the following eigenvalue representation :
$$
\Lambda_n(\Xi)=\inf_{F\subseteq\mathscr D^{\gamma}_{\xi},\text{dim}(F)=n}\sup_{\substack{f\in F;\\||f||_{L^2}=1}}\langle\mathscr{ H}(\Xi) f,f\rangle_{L^2}
$$
\end{lemma}
\begin{proof}
Since $\mathscr D_{\Xi}\subseteq\mathscr D_{\xi}^{\gamma}$ we have trivially the following inequality  
$$
\Lambda_{n}(\Xi)\geq \inf_{F\subseteq\mathscr D^{\gamma}_{\xi},\text{dim}(F)=n}\sup_{\substack{f\in F;\\||f||_{L^2}=1}}\langle\mathscr{ H}(\Xi)f,f\rangle_{L^2}
$$
The other inequality is obtained simply by the fact that the space $\mathscr D_{\Xi}$ is dense in $\mathscr D_{\xi}^{\gamma}$.
\end{proof}
Now we will show that the eigenvalue are continuous as function of $\Xi$. For that we will need the following result : 
\begin{Proposition}\label{prop:growth}
Given  $\alpha\in(-4/3,-1)$, $-\alpha/2<\gamma<\alpha+2$, $n\in \mathbb N$ and $\Xi=(\xi,\Xi)\in\mathscr X^\alpha$ and denoting by $\Lambda_{1}(0)\leq\Lambda_{2}(0)\leq. . .\leq\Lambda_n(0)$ the $n$
 first eigenvalues of $-\Delta$ repeated with their multiplicity. Then for all $C>0$  there exists a free family of vector $f_1,f_2,. . .,f_n\in\mathscr D_{\xi}^{\gamma}$ such that : 
\begin{enumerate}
\item $\frac{1}{2}\leq\|f_i\|_{L^2}\leq 2$
\item $|\langle f_i,f_j\rangle_{L^2}|\leq\frac{C}{4n}$ for $i\ne j$
\item $\|f_i\|_{\mathscr D_{\xi}^{\gamma}}\lesssim_{\alpha,\gamma} (1+\Lambda_i(0))^{2\gamma}+n^{\frac{2\gamma-\alpha}{2+\alpha}}\|\xi\|_{\mathscr C^{\alpha}}$ 
\end{enumerate}
And therefore from the min-max principle gives us the following bound 
\begin{equation}\label{eq:growth}
\Lambda_n(\Xi)\lesssim n((1+\Lambda_n(0))^{2\gamma}+n^{\frac{2\gamma-\alpha}{2+\alpha}}\|\xi\|^{\frac{\alpha+3}{\alpha+2}}_{\mathscr C^{\alpha}})^2(1+\|\Xi\|_{\mathscr X^\alpha})^2
\end{equation}
hold for all $n$.
\end{Proposition}
\begin{proof}
Let $e_1(0),e_2(0),. . .,e_n(0)$ an orthonormal family of eigenvector respectively associated to the $n$ first eigenvalues of $-\Delta$ then as in the proof of the Lemma~\ref{Lemma: density} we prove that for $a^{-1}|\xi|^{\frac{2}{2+\alpha}}_{\mathscr C^{\alpha}}$ small enough the map $\Gamma_i: L^{2}\to L^2$ defined by 
$$
\Gamma_i(f)=f\prec\sigma_a(D)\xi+e_i(0)
$$
admit a unique fix point $f_i$ which satisfy $\|f_i\|_{L^2}\leq 2\|e_i(0)\|_{L^2}=2$. Now we will show that the family  $f_1,f_2,. . .,f_n$ is free for $a$ large enough. Indeed let us assume that 
$$
f_1=\sum_{i=2}^{n}\kappa_if_i
$$
using that $f_i$ is a fixed point of the map $\Gamma_i$ we get easily the following equation 
$$
(f_1-\sum_{i=2}^n\kappa_if_i)\prec\sigma_a(D)\xi=\sum_{i=1}^n\kappa_ie_i(0)-e_1(0)
$$
As usual from Bony and Schauder estimate we get 
$$
\|\sum_{i=2}^n\kappa_ie_i(0)-e_1(0)\|_{L^2}\lesssim a^{-(1+\alpha/2)}\|(f_1-\sum_{i=2}^n\kappa_if_i)\|_{L^2}\|\xi\|_{\mathscr C^{\alpha}} 
$$ 
$$
\lesssim \sqrt{n} a^{-(1+\alpha/2)}(1+\sum_i|\kappa_i|^2)\|\xi\|_{\mathscr C^{\alpha}}
$$
on the other side we have that 
$$
\|\sum_{i=2}^n\kappa_ie_i(0)-e_1(0)\|^2_{L^2}=1+\sum_i|\kappa_i|^2
$$
which is impossible if we choose $a$ such that  $a^{-1} n^{\frac{1}{\alpha+2}}\|\xi\|^{\frac{2}{\alpha+2}}_{\mathscr C^{\alpha}}$ is small enough. Now if we take $i\ne j$ is easy to see that 
$$
|\langle f_i,f_j\rangle_{L^2}|\lesssim \|f_i-e_i(0)\|+\|f_j-e_j(0)\|_{L^2}
$$
where we have used the Cauchy-Schwartz inequality and the orthogonality between $e_i(0)$ and $e_j(0)$. Once again Bony estimate and the fact that $\|f_i\|_{L^2}\leq 2$ gives :
$$
\|f_i-e_i(0)\|_{L^2}\lesssim a^{-(1+\frac{\alpha}{2})}\|\xi\|_{\mathscr C^{\alpha}}
$$
which for $a^{-1}n^{\frac{2}{\alpha+2}}\|\xi\|^{\frac{2}{\alpha+2}}_{\mathscr C^{\alpha}}$ small enough (Depending on $C$) ensure that 
$$
|\langle f_i,f_j\rangle_{L^2}|\leq \frac{C}{4n}
$$
Now we have to control the norm of our vector in the space $\mathscr D_{\xi}^{\gamma}$. To do that let us start by observing that  $f_i^{\sharp}=f\prec(\sigma_a(D)-\sigma(D))\xi+e_i(0)$ then Bony and Schauder estimates allow us to get the two following bounds :
$$
\|f_i^{\sharp}\|_{H^{2\gamma}}\lesssim a^{\frac{2\gamma-\alpha}{2}}\|\xi\|_{\mathscr C^\alpha}+(1+\Lambda_i(0))^{2\gamma}
$$
and  
$$
\|f_i\|_{H^\gamma}\lesssim \|\xi\|_{\mathscr C^{\alpha}}+(1+\Lambda_i(0))^\gamma
$$
Which gives the needed estimate for $\|f\|_{\mathscr D_\xi^{\gamma}}$ if we take $a=K  n^{\frac{2}{\alpha+2}}\|\xi\|^{\frac{2}{\alpha+2}}_{\mathscr C^{\alpha}}$ for $K>0$ a sufficiently large constant. Now we will prove the bound~\eqref{eq:growth}. Using the fact that $\text{span}\{f_1,f_2,. . .,f_n\}$ is a $n$- dimensional sub-space of $\mathscr D_\xi^{\gamma}$ and the min-max principle we get easily that:  
$$
\Lambda_n(\Xi)\leq\sup_{\substack{f\in\text{span}\{f_1,f_2,. . .,f_n\}\\ \|f\|_{L^2}=1}}\langle\mathscr H(\Xi)f,f\rangle_{L^2}\lesssim(1+\|\Xi\|_{\mathscr X^\alpha})^2\sup_{\substack{f\in\text{span}\{f_1,f_2,. . .,f_n\}\\ \|f\|_{L^2}=1}}\|f\|^2_{\mathscr D_\xi^\gamma}
$$  
 Let $f$ such that 
$$
f=\sum_{i=1}^n\kappa_if_i
$$
then
$$
\|f\|^2_{\mathscr D_\xi^\gamma}\lesssim \max_{i}\|f_i\|^2_{\mathscr D_\xi^\gamma}\big(\sum_{i=1}^n|\kappa_i|\big)^2
$$ 
on the other side we have 
$$
\|f\|_{L^2}=\sum_{i=1}^n|\kappa_i|^2\|f_i\|_{L^2}+\sum_{i\ne j}\kappa_i\kappa_j\langle f_i,f_j\rangle_{L^2}
$$ 
which can be lower-bounded using the two firsts properties of the function $f_i$. Namely
$$
\|f\|_{L^2}\geq \frac{1}{2}\sum|\kappa_i|^2-\frac{1}{4n}(\sum_{i}\kappa_i)^2\geq \frac{1}{4n}(\sum|\kappa_i|)^2
$$ 
and finally we can conclude that 
$$
\frac{\|f\|_{\mathscr D_\xi^\gamma}}{\|f\|_{L^2}}\lesssim n\max_{i}\|f_i\|^2_{\mathscr D_\xi^\gamma}
$$
which by the third property of the function $f_i$ complete the proof.
\end{proof}

\begin{Proposition}
\label{prop: continuity-eig}
Given $\alpha\in(-4/3,-1)$, $\Xi\in\mathscr X^\alpha$ and let us denote by $(\Lambda_n(\Xi))_{n\geq1}$ the set of the eigenvalue of the operator $\mathscr{H}=\mathscr{H}(\Xi)$. Then for all $n\geq 1$ the map $\Lambda_n :\mathscr X^\alpha\to\mathbb R$ is locally Lipschitz. More precisely it exist $M>0$ such that for all $n\in\mathbb N$, $\Xi,\tilde \Xi\in\mathscr X^\alpha$ and $-\alpha/2<\gamma<\alpha+2$ 
$$
|\Lambda_n(\Xi)-\Lambda_n(\tilde \Xi)|\lesssim_{\gamma,\alpha}n\|\Xi-\tilde\Xi\|_{\mathscr X^\alpha}(1+\|\Xi\|_{\mathscr X^\alpha}+\|\tilde \Xi\|_{\mathscr X^\alpha})^M\left(1+n^{\frac{2\gamma-\alpha}{\alpha+2}}+(1+\Lambda_n(0))^{2\gamma}\right)^2
$$ 
Where the last bound is uniform on $L$.
\end{Proposition}  
\begin{proof}
The min-max principle for the resolvent $\mathcal G_a$ gives that 
$$
\big|\frac{1}{\Lambda_n(\Xi)+a}-\frac{1}{\Lambda_n(\tilde \Xi)+a}\big|\leq \|\mathcal G_a(\Xi)-\mathcal G_a(\tilde \Xi)\|_{\mathcal L(L^2,L^2)}
$$
where $\mathcal L(L^2,L^2)$ is the space of bounded operator on $L^2$ equipped on his usual norm.Then using the bound given in the Lemma~\ref{lemma:cont} we can easily deduce that 
$$
|\Lambda_n(\Xi)-\Lambda_n(\tilde \Xi)|\lesssim \|\Xi-\tilde\Xi\|_{\mathscr X^\alpha}(1+\|\Xi\|_{\mathscr X^\alpha}+\|\tilde \Xi\|_{\mathscr X^\alpha})^2(\Lambda_n(\Xi)+a)(\Lambda_n(\tilde \Xi)+a)
$$
for all $a\geq (1+\|\Xi\|_{\mathscr X^\alpha}+\|\tilde \Xi\|_{\mathscr X^\alpha})^{\frac{2}{\varrho}}$ with $\varrho$ is as in the Lemma~\ref{lemma:cont} it is sufficient to use the bound~\eqref{eq:growth}. 
\end{proof}

\section{Renormalization for the Anderson hamiltonian} 
\label{sec:Stoch}
\subsection{Renormalization for the white noise potential}
For $\alpha <-1$, we consider a smooth approximation $\xi_\eps$ 
of the white noise $\xi\in \mathscr{C}^\alpha$  and 
establish the convergence in $\mathscr{X}^\alpha$ of the mollified family
$(\xi_\eps,\xi_\eps\circ \sigma(D)\xi_\eps+c_\eps)$ for some diverging constants 
$(c_\eps)_{\eps>0}$ towards some $(\xi,\Xi_2^{wn})\in\mathscr{X}^\alpha$. 
The following Theorem is a cornerstone of our study as it permits us to handle concretely 
the relevant case of the white noise and use the approximation theory developed in the previous sections  
to this case.  

\begin{theorem}\label{conv-white-noise}
Let $\alpha<-1$ and $\xi\in \mathscr{C}^\alpha$ 
a white-noise on the two dimensional torus $\mathbb T_L^2$. We consider $\xi^{\eps}$ a smooth approximation of the white noise $\xi$ defined as
\begin{align*}
\xi_\eps(x):=\sum_{k\in\mathbb Z_L^2}\theta(\eps |k|)\hat{\xi}(k)\frac{1}{L}\exp(i 2\pi \langle k,x\rangle)
\end{align*}
where $\theta$ is a smooth function on $\mathbb R\setminus\{0\}$ with compact support, 
such that $\lim_{x\to0}\theta(x)=1$. Then, there exists a diverging sequence of constants 
$(c_\eps)_{\eps >0} = (c_\eps(\theta))_{\eps>0}$ and 
a limit point $\Xi^{wn}=(\xi,\Xi^{wn}_2)\in\mathscr X^\alpha$ such that the following convergence  
\begin{align*}
\Xi^\eps:=(\xi_\eps,\xi_\eps\circ \sigma(D)\xi_\eps+c_\eps) \underset{\eps\to0}{\longrightarrow} \Xi^{wn}
\end{align*}
holds in $L^p(\Omega,\mathscr E^\alpha)$ for all $p>1$ and almost-surely in $\mathscr E^\alpha$. 
Moreover $\Xi^{wn}$ is independent on the choice of $\theta$. 
\end{theorem}
\begin{remark}
 For simplicity we will denote by $\Xi$ (instead of $\Xi^{wn}$) the rough distribution associate to the white noise. 
\end{remark}

\begin{proof}
We first prove the convergence of $\xi_\eps$ in the space $\mathscr{C}^\alpha$ for any $\alpha <-1$. 
For $i\geq -1$ and $x\in \mathbb{T}_L^2$, we have 
\begin{align*}
\Delta_i(\xi_\eps-\xi)(x) = 
\sum_{k \in \Z_L^2}  \rho(2^{-i} |k|)  \hat{\xi}(k) (\theta(\eps |k|) -1)\frac{1}{L}\exp(i 2\pi \langle k,x\rangle)
\end{align*}
where the $ \hat{\xi}(k), k\in \Z_L^2$ form a family of independent and identically distributed complex 
Gaussian variables such that, for any $k,\ell \in \Z_L^2$,  
\begin{align*}
\hat{\xi}(-k) &= \overline{\hat{\xi}(k)},  \\
\E[\hat{\xi}(k) \hat{\xi}(\ell) ] &= \delta(k+\ell) \,.  
\end{align*}
We deduce the following upper-bound valid up to a constant independent of $i\geq -1$
and for any $\delta >0,x\in \mathbb{T_L}^2$, 
\begin{align*}
\mathbb E\left[|\Delta_i(\xi_\eps-\xi)(x)|^2\right]&\lesssim \sum_{k\in\mathbb Z_L^2}|\theta(\eps |k|)-1|^2 
\rho(2^{-i}|k|)^2\\
& \lesssim 2^{(2+\der)i}\left(L^{-2}\sum_{k\in\mathbb Z_L^2 }\frac{|\theta(\eps |k|)-1|^2}{(1+|k|)^{2+\der}}\right)
\end{align*}
where we have used the fact that the function $\rho$ is supported on a compact annulus for the second line 
\footnote{The sum in the first line contains only the terms with indices $|k|\sim 2^{i}$. }.  
Before proceeding with the computation let us observe that by Riemann-sum approximation we have have the following bound 
$$
L^{-2}\sum_{k\in\mathbb Z_L^2\setminus }\frac{|\theta(\eps |k|)-1|^2}{(1+|k|)^{2+\der}}\lesssim\int_{\mathbb R^2}\frac{|\theta(\eps |y|)-1|^2}{(1+|y|)^{2+\der}}\dd y
$$
which show in particularly that our bound is uniform in  $L$. Integrating over $x\in \mathbb{T}_L^2$, 
we obtain an upper-bound for the $L^2$ norm of $\Delta_i(\xi_\eps-\xi)$ which can be generalized for any $p>1$ thanks 
to the Gaussian hypercontractivity property~\cite{janson} 
 \begin{align*}
\mathbb E\left[||\Delta_i(\xi_\eps-\xi)||^p_{L^p}\right] &\lesssim \int_{\mathbb T_L^2}\mathbb E\left[|\Delta_i(\xi_\eps-\xi)(x)|^2\right]^{\frac{p}{2}} dx \\
& \lesssim_p L^{2} 2^{ip(1+\der/2)}\left(L^{-2}\sum_{k\in\mathbb Z_L^2}\frac{|\theta(\eps |k|)-1|^2}{(1+|k|)^{2+\der}}\right)^{p/2}\,. 
\end{align*}
Multiplying both sides of this inequality by $2^{-ip(1+\frac{\delta}{2})}$ and summing over $i\geq-1$ gives 
\begin{align*}
\mathbb E\left[||\xi_\eps-\xi||^p_{\mathscr B_{p,p}^{-1-\frac{\der}{2}} }\right]\lesssim_p L^2 \left(L^{-2}\sum_{k\in\mathbb Z_L^2}\frac{|\theta(\eps |k|)-1|^2}{1+|k|^{2+\der}}\right)^{p/2}
\end{align*}
From the embedding property \ref{Proposition:Bes-emb} of the Besov spaces, we obtain 
\begin{align*}
\mathbb E\left [||\xi_\eps-\xi||^p_{\mathscr C^{-1-\frac{\der}{2}-\frac{2}{p}} }\right]\lesssim_p L^2\left(L^{-2}\sum_{k\in\mathbb Z_L^2}\frac{|\theta(\eps |k|)-1|^2}{1+|k|^{2+\der}}\right)^{p/2}\,. 
\end{align*}
For any $\delta >0$, the bounded convergence Theorem permits us to show 
that the right hand side of the last equation converges to zero and we can finally conclude that 
$\xi_\eps$ converge to $\xi$ in $L^p(\Omega,\mathscr C^\alpha)$ for any $\alpha<-1$.

Towards the almost sure convergence of $\xi_\eps$, 
 the same arguments apply to $\xi_\eps-\xi_{\eps'}$ instead of $\xi_\eps-\xi$ and we get 
\begin{align*}
\mathbb E\left [||\xi_\eps-\xi_{\eps'}||^p_{\mathscr C^{-1-\frac{\der}{2}-\frac{2}{p}}}\right]
\lesssim L^2\left(L^{-2}\sum_{k\in\mathbb Z_L^2}\frac{|\theta(\eps |k|)-\theta(\eps'|k|)|^2}{1+|k|^{2+\der}}\right)^{p/2}
\end{align*}
Using the fact that $|\theta(\eps |k|)-\theta(\eps' |k|)|\lesssim_p |\eps-\eps'|^\eta|k|^\eta$, we obtain that for any 
$0<\eta<\der$
\begin{align*}
\mathbb E\left [||\xi_\eps-\xi_{\eps'}||^p_{\mathscr C^{-1-\frac{\der}{2}-\frac{2}{p}}}\right]\lesssim |\eps-\eps'|^{p\eta/2}\,. 
\end{align*}
This proves the convergence of $(\xi_\eps)_{\eps>0}$ in $L^p(\Omega,\mathscr{C}^\alpha)$ for any $\alpha<-1$ 
and $p>0$. 

Towards the almost sure convergence, 
the Kolmogorov criterion applied for $p>2/\eta$ permits us to show 
that there exists $\kappa\in (0,\frac{\eta}{2})$ 
such that for $p$ large enough, we have almost surely, 
\begin{align*}
||\xi_\eps-\xi_{\eps'}||_{\mathscr C^{-1-\frac{\der}{2}-\frac{2}{p}}  }\lesssim |\eps-\eps'|^{\kappa}\,. 
\end{align*}
The sequence $(\xi_\eps)_{\eps>0}$ is therefore almost surely a Cauchy sequence 
in the Banach space $\mathscr{C}^{-1-\frac{\der}{2}-\frac{2}{p}}$ where $p$ (resp. $\delta$) is arbitrarily large
(resp. small).  We can conclude that $(\xi_\eps)_{\eps>0}$ converges almost surely in 
$\mathscr{C}^{\alpha}(\T_L^2)$ for any $\alpha <-1$.

Now we turn to the second chaos term which may be expanded as
\begin{equation*}
\begin{split}
 \xi_\eps \circ \sigma(D) \xi_\eps(x)
 =L^{-2}\sum_{\substack{k\in\mathbb Z_L^2 ,\ell \in\mathbb Z_L^2, \\ |i-j|\leq1}}\rho(2^{-i}|k|) \rho(2^{-j}|\ell|)
\theta(\eps |k|)\frac{\theta(\eps |\ell| )}{1+|\ell|^2} \hat{\xi}(k)\hat{\xi} (\ell) e_{k+\ell}(x)\,. 
\end{split}
\end{equation*}  
Taking the expectation in this last formula yields 
\begin{align*}
-\mathbb E \left[\xi_\eps \circ \sigma(D) \xi_\eps(x) \right]
=L^{-2}\sum_{\substack{ k\in \mathbb Z_L ^2  ,\\  |i-j| \leq1} } \rho(2^{-i}|k|) \rho(2^{-j}|k|)
\frac{|\theta(\eps |k|)|^2}{1+|k|^2}=\sum_{\substack{ k\in \mathbb Z_L ^2 } }
\frac{|\theta(\eps |k|)|^2}{1+|k|^2}
\end{align*}
so that it is  sufficient to set
\begin{align*}
c_\eps:=L^{-2} \sum_{k \in\mathbb Z_L^2 } \frac{|\theta(\eps | k |)|^2}{1+|k|^2}\,. 
\end{align*}
One can easily check that for any function $\theta$, which is smooth on $\R\setminus\{0\}$ and such that 
$\theta(x)\to 1$ when $x\to 0$, we have    
\begin{align*}
c_\eps=\frac{1}{2\pi} \log(\frac{1}{\eps})+O(1)\,. 
\end{align*}
For $x \in \T_L^2$, we set 
\begin{align*}
\Xi^\eps_2(x) := \xi_\eps \circ \sigma(D) \xi_\eps(x) +c_\eps\,. 
\end{align*}
We will now prove that the sequence 
$(\Xi^\eps_2)_{\eps>0}$ is a Cauchy convergent 
sequence in the Banach space  $\mathscr C^{2\alpha+2}(\T_L^2)$
for any $\alpha <-1$. 
We have
\begin{align*}
\Xi_2^\eps(x)=L^{-2}\sum_{\substack{k\in\mathbb Z_L^2 ,\ell \in\mathbb Z_L^2, \\ |i-j|\leq1}}\rho(2^{-i}|k|) \rho(2^{-j}|\ell|)
\theta(\eps |k|)\frac{\theta(\eps |\ell| )}{1+|\ell|^2} ( \hat{\xi}(k)\hat{\xi} (\ell) - \delta(k+\ell)) 
e_{k+\ell}(x) 
\end{align*}
We want to bound the second moments of the Littlewood-Paley blocks 
of the difference 
$\Xi_2^\eps - \Xi_2^{\eps'}$ for $\eps, \eps'>0$: For any $q\geq -1$,  
\begin{equation*}
\begin{split}
&\mathbb E[|\Delta_q(\Xi^\eps_2- \Xi_2^{\eps'}) (x)|^2]
\\&= L^{-4}\sum_{\substack{k\in\mathbb Z_L^2 ,\ell \in\mathbb Z_L^2,
\\ |i_1-j_1| \leq 1,|i_2-j_2|\leq1}}
\rho(2^{-q}|k+\ell|)^2  \left| \Pi_{m=1,2}\rho(2^{-i_m}|k|)\rho(2^{-j_m}|\ell|) \right| 
\frac{|\theta(\eps |k|)\theta(\eps |\ell|)-
\theta(\eps' |k|)\theta(\eps'|\ell|)|^2}{(1+|\ell|^2)^2}
\\
&+L^{-4}\sum_{\substack{k\in\mathbb Z_L^2 ,\ell \in\mathbb Z_L^2\setminus\{0\},
\\ |i_1-j_1| \leq 1,|i_2-j_2|\leq1}}
\rho(2^{-q}|k+\ell|)^2  \left| \Pi_{m=1,2}\rho(2^{-i_m}|k|)\rho(2^{-j_m}|\ell|) \right| 
\frac{|\theta(\eps |k|)\theta(\eps |\ell|)-
\theta(\eps' |k|)\theta(\eps'|\ell|)|^2}{(1+|k|^2)(1+ |\ell|^2)}\,. 
\end{split}
\end{equation*}
Due to the fact that $|i_1-j_1| \le 1$ and that 
the support of $\rho$ is a fixed annulus, we can deduce that the 
indices of the non zero terms of the
two sums appearing in the last formula are such that 
\begin{align*}
|k| \lesssim |\ell| \lesssim |k|
\end{align*}
where the undisplayed multiplicative constants in those inequalities do not depend on $i_1,j_1,i_2,j_2$. 
This implies that the two sums are of the same order in the sense that one can bound 
any of the two sums with the other one. 
Using the fact that $\rho$ is bounded, we can bound the sum over $i_1,j_1,i_2,j_2$ by a constant 
and we are eventually reduced to estimate the quantity 
\begin{align}\label{quantity1}
\sum_{\substack{k\in\mathbb Z_L^2, \ell \in \Z_L^2\setminus\{0\}, \\a |k|\le|\ell|\le A |k|}} \rho(2^{-q}|k+\ell|)^2
\frac{|\theta(\eps |k|)\theta(\eps |\ell|)-\theta(\eps'|k|)\theta(\eps'|\ell|)|^2}{(1+|\ell|^2)^2}
\end{align} 
for some given constants $a,A>0$ (independent on L). We set $n=k+\ell$ and note that if 
$a |k|\le|\ell|\le A |k|$, then $n \le |k|+|\ell|\le (a^{-1}+1) |\ell| $. We deduce that, up to a multiplicative constant
and denoting by $\delta>0$ a fixed (small enough) parameter, 
we can bound \eqref{quantity1} by  
\begin{align*}
\left(\sum_{n\in\mathbb Z_L^2}\frac{\rho(2^{-q}|n|)^2}{1+|n|^{2-\der}}\right) &
\sup_{n\in \Z_L^2} \left( \sum_{ \substack{ k \in \Z_L^2, \ell \in \Z_L^2: k+\ell =n, \\
 a |k|\le|\ell|\le A |k|} } \frac{|\theta(\eps |k|)\theta(\eps |\ell|)-\theta(\eps'|k|)\theta(\eps'|\ell|)|^2}{1+|\ell|^{2+\der}} \right)
 \\
 & \lesssim 2^{q \delta} \sup_{n\in \Z_L^2} \left( \sum_{ \substack{ k \in \Z_L^2, \ell \in \Z_L^2: k+\ell =n, \\
 a |k|\le|\ell|\le A |k|} } \frac{|\theta(\eps |k|)\theta(\eps |\ell|)-\theta(\eps'|k|)\theta(\eps'|\ell|)|^2}{|\ell|^{2+\der}} \right)
\end{align*}
where we have used the fact that $n\to \rho(2^{-q}|n|)$ is supported in a ball (in fact an annulus) of $\Z_L^2$
with radius $2^q$. 
Using the inequality $|\theta(\eps |k|)\theta(\eps |\ell|)-\theta(\eps'|k|)\theta(\eps'|\ell|)| \lesssim 
|\eps-\eps'|^\eta ( |k|^\eta + |\ell|^\eta) $ valid for $\eta>0$ small enough, we easily check that 
for $0<\eta < \delta/2$,
\begin{align*}
L^{-4}\sup_{n\in \Z_L^2} \left( \sum_{ \substack{ k \in \Z_L^2, \ell \in \Z_L^2: k+\ell =n, \\
 a |k|\le|\ell|\le A |k|} } \frac{|\theta(\eps |k|)\theta(\eps |\ell|)-\theta(\eps'|k|)\theta(\eps'|\ell|)|^2}{1+|\ell|^{2+\der}} \right)
 \lesssim |\eps-\eps'|^{2\eta}\,.
\end{align*}
Where the last bound is uniform in $L$. Gathering the above inequalities, we finally obtain the second moment estimate, valid with 
$0<\eta < \delta/2$,
\begin{align*}
\mathbb E\left[|\Delta_q(\Xi^\eps_2- \Xi_2^{\eps'}) (x)|^2\right] \lesssim 2^{q\delta}  |\eps-\eps'|^{2\eta}\,. 
\end{align*}
where this bound is uniform in $L$. Using as before the Gaussian hypercontractivity and the Besov embedding arguments, we deduce 
the following upper-bound 
\begin{align*}
\E\left[|| \Xi^\eps_2- \Xi_2^{\eps'} ||^p_{\mathscr C^{-\der-2/p}}\right]\lesssim L^2 |\eps-\eps'|^{p \eta}
\end{align*}
The convergence of the sequence  $(\Xi^\eps_2)_{\eps>0}$ in $L^p(\Omega,\mathscr{C}^{2\alpha+2}(\T_L^2))$ for 
$\alpha<-1$ and $p>0$ is proved.

The Kolmogorov criterion permits us to conclude 
that the sequence $(\Xi^\eps_2)_{\eps>0}$ converges almost surely 
in the space $\mathscr{C}^{2\alpha+2}(\T_L^2)$ 
for any $\alpha <-1$. 
\end{proof}

\subsection{Growth of the eigenvalue}
\label{section:growth}
In this section we are interested to quantify the growth of the eigenvalue when $L$ the size of the Torus become large. Of course to do that the first step is to control the growth of the rough distribution associate to the white noise when $L$ become large. Namely we have the following result 
 \begin{lemma}
\label{lemma:white noise growth}

Given $L\in\mathbb N^{\star}$, $\alpha<-1$ and $\xi$ a white noise on the two dimensional torus $\mathbb T^2_{L}$ of size $L$  then there exists two finite constant $C\geq 2$ and $\lambda$ such that
$$
\sup_{L}L^{-C}\mathbb E\left[\exp({\lambda\|\xi\|^2_{\mathscr C^{\alpha}}}+\|\Xi_2\|_{\mathscr C^{2\alpha+2}})\right]<+\infty
$$
 and therefore  
$$
A_{\alpha}=\sup_{L}\frac{\|\xi\|^2_{\mathscr C^\alpha(\mathbb T^2_L)}}{\log(L)}+\sup_{L}\frac{\|\Xi_2\|_{\mathscr C^{2\alpha+2}(\mathbb T^2_L)}}{\log(L)}<+\infty
$$
almost surely, Moreover $\mathbb E[ \exp(hA_{\alpha})]<+\infty$ for a sufficiently small constant $h>0$. 
\end{lemma}
\begin{proof}
Let us recall that the white noise have the following representation
\begin{equation}
\label{eq:wn-foruier}
\xi=\sum_{k\in\mathbb Z_L^2}g_k e_k^L
\end{equation}
with $e_k^L(x)=\frac{1}{L}\exp(2i\pi\langle k,x\rangle)$ with $g_k$ is an i.i.d sequence of Gaussian random variable which satisfy that $g_{-k}=\overline{g_{k}}$. Then the same computation as in the Theorem~\ref{conv-white-noise} allow us to get that 
$$
\mathbb E[|\Delta_q\xi(x)|^p]\leq  \mathbb E[|\mathcal N(0,1)|^{p}]\big( L^{-2}\sum_{k\in\mathbb Z_L^2}|\theta(2^{-q}|k|)|^2\big)^{p/2}\lesssim  \mathbb E[|\mathcal N(0,1)|^{p}]2^{qp}
$$ 
with $\mathcal N(0,1)$ is centered Gaussian random variable with unit variance. Integrating over $x$ and taking the sum on $q$ gives :
\begin{equation}
\label{eq:wn.growth}
\sup_{L}L^{-2}\mathbb E[\|\xi\|^p_{\mathscr C^{-1-\kappa}}]\lesssim_{\kappa}\mathbb E [|\mathcal N(0,1)|^{p}]
\end{equation}
for $p$ such that  $\frac{2}{p}\leq \frac{\kappa}{2}$ and where we have used the Besov embedding. 
$$
\mathbb E[e^{\lambda \|\xi\|^2_{\mathscr C^{-1-\kappa}}}]=\sum_{r=0}^{\frac{8}{\kappa}}\frac{\lambda^r}{r!}\mathbb E[\|\xi\|^{2r}_{\mathscr C^{-1-\kappa}}]+\sum_{r>\frac{8}{\kappa}}\frac{\lambda^r}{r!}\mathbb E[\|\xi\|^{2r}_{\mathscr C^{-1-\kappa}}].
$$
 Using the inequality~\eqref{eq:wn.growth} the second sum of  this last equation can be bounded in the following way : 
$$
\sum_{r>\frac{8}{\kappa}}\frac{\lambda^r}{r!}\mathbb E[\|\xi\|^{2r}_{\mathscr C^{-1-\kappa}}]\lesssim L^2 \sum_{r>\frac{8}{\kappa}}\frac{\lambda^r}{r!}\mathbb E[|\mathcal N(0,1)|^{2r}]=L^2\mathbb E[\exp(\lambda|\mathcal N(0,1))|^2)]<+\infty
$$
 under the condition that $\lambda$ is small enough. Since the first sum is a finite sum we bound  each term of it using the Jensen inequality:
$$
\sum_{r=0}^{\frac{8}{\kappa}}\frac{\lambda^r}{r!}\mathbb E[\|\xi\|^{2r}_{\mathscr C^{-1-\kappa}}]\lesssim L^{\frac{8}{\kappa}}  \sum_{r=0}^{\frac{8}{\kappa}}\frac{\lambda^r}{r!}\mathbb E[|\mathcal N(0,1)|^{\frac{\kappa r}{4}}]^{\frac{8}{\kappa}}
$$
from which we can conclude that: 
$$
\sup_{L}L^{-(\max(2,\frac{8}{\kappa}))} \mathbb E\big[\exp(\lambda \|\xi\|^2_{\mathscr C^{\alpha}(\mathbb T^2_L)})\big]<\infty
$$
for $\alpha=-1-\kappa$.  Therefore using Fubini theorem gives :
$$
\mathbb E\left[\sum_L\frac{1}{L^{a}}\exp(\lambda  \|\xi\|^2_{\mathscr C^{\alpha}(\mathbb T^2_L)})\right]<\infty
$$
for $a>\max(2,\frac{8}{\kappa})+1$, which in particularly prove that:
$$
 \|\xi\|_{\mathscr C^{\alpha}(\mathbb T^2_L)}\lesssim_{\lambda} a\sqrt{\log (L)+\log A}
$$
with $A=\sum_L\frac{1}{L^{a}}\exp(\lambda  \|\xi\|^2_{\mathscr C^{\alpha}(\mathbb T^2_L)})$ is an integrable random variable.
To complete our proof, we just have to establish the same estimate for $\|\Xi_2\|_{\mathscr C^{2\alpha+2}}$ for which 
the same computation as in Theorem~\ref{conv-white-noise} allows us to get the 
following upper bounds 
$$
\sup_{L}\mathbb E[|\Delta_q\Xi_2(x)|^2]\lesssim 2^{q\kappa}
$$
 for all $\kappa>0$ small enough. Now since $\Delta_q\Xi_2$ is in the second chaos of the white noise $\xi$ the Gaussian hypercontractivity tell us that 
$$
\mathbb E[|\Delta_q\Xi_2(x)|^p]\lesssim A_p\mathbb E[\Delta_q\Xi_2(x)]^{\frac{p}{2}}\lesssim A_p 2^{q\kappa p/2}
$$ 
for all $\delta>0$ and where $A_p=\mathbb E[|\mathcal N(0,1)|^{2p}]$. Then repeating the argument used to control 
the growth of the  white noise allows us to obtain the following integrability result
$$
\mathbb E\left[\sum_{L}\frac{1}{L^2}\exp(\lambda\|\Xi_2\|_{\mathscr C^{-\kappa}})\right]<\infty
$$
which finishes the proof.
\end{proof}
Now a crucial observation is that the eigenvalues satisfy a rescaling property. Indeed let $V$ a smooth $L$-periodic potential, $\tilde\Lambda_n(V)$ is the $n$-lowest eigenvalue of $-\Delta+V$ and $e_n(V)$  an eigenvector associate to $\tilde\Lambda_n(V)$. Then for $r>0$ we can see that the function $e_n^r(x)=e_n(V)(rx)$ is an eigenvector of the operator $-\Delta+r^2V(r\cdot)$ with eigenvalue $r^2\tilde\Lambda_n(V)$ and therefore 
\begin{equation*}\label{eq:smooth-scaling}
\tilde\Lambda_n(V)=\frac{1}{r^2}\tilde\Lambda_n(r^2V(r\cdot))
\end{equation*}
where $\tilde\Lambda_n(V(r\cdot))$ is the $n$-lowest eigenvalue of  $-\Delta+V(r\cdot)$ seen as an operator of $\mathbb T^2_{r^{-1}L}$. As the reader can guess we want to extend the identity to the irregular stetting for that we observe that the identity~\eqref{eq:smooth-scaling} can be reformulated in the following manner  
$$
\Lambda_n(V,V\circ\sigma(D)V+c)=\frac{1}{r^2}\Lambda_n(r^2V(r\cdot),r^4V(r\cdot)\circ\sigma(D)(V(r\cdot))+r^2c)
$$ 
for every $c\in\mathbb R$. Since this eigenvalue identity holds for any smooth function $V$ it can be extend to the case of the white noise easily. Indeed let $\xi$ the white noise on $\mathbb T^2_L$ and $\xi_\eps=\theta(\eps|D|)\xi$. From the fact that $\xi_\eps$ is smooth and $\xi_\eps(r\cdot)=\theta(\eps |D|)\xi(r\cdot)=(\theta(\frac{\eps|D|}{r})\xi(r\cdot))(\cdot)=(\xi(r\cdot))^{\eps}$ we get immediately the following relation  
\begin{equation*}\label{eq:scaling mollification}
\begin{split}
\Lambda_n(\xi_\eps,\xi_\eps\circ\sigma(D)\xi^{\eps}+c_\eps)&=\frac{1}{r^2}\Lambda_n(r^2\xi_\eps(r\cdot),r^4\xi_\eps(r\cdot)\circ\sigma(D)(\xi_\eps(r\cdot))+r^2c_\eps )
\\&=\frac{1}{r^2}\Lambda_n\big(r^2(\xi(r\cdot))^{\frac{\eps}{r}},r^4(\xi(r\cdot))^{\frac{\eps}{r}}\circ\sigma(D)((\xi(r\cdot))^{\frac{\eps}{r}})+r^2c_\eps\big)
\end{split}
\end{equation*}
 where we recall that $c_\eps=-\mathbb E[\xi_\eps\circ\sigma(D)\xi_\eps]$ is the diverging constant given in th Theorem~\ref{conv-white-noise}. Now the point is to take the limit when $\eps$ goes to zero in this equation, indeed the continuity of $\Lambda_n$ and the convergence result of the Theorem~\ref{conv-white-noise} tell us that the left hand side of this equality converge to $\Lambda_n(\Xi)$. On the other hand is easy to see that $\tilde\xi_r:= r\xi(r\cdot)$ is a white noise on $\mathbb T^2_{\frac{L}{r}}$ and theretofore we have that 
$$
(\xi(r\cdot))^{\frac{\eps}{r}}\circ\sigma(D)((\xi(r\cdot))^{\frac{\eps}{r}})+\frac{1}{r^2}\tilde c_{\eps}=\frac{1}{r^2}(\tilde\xi_r^{\frac{\eps}{r}}\circ\sigma(D)\tilde\xi_r^{\frac{\eps}{r}}+\tilde c_{\frac{\eps}{r}})
$$
converge almost surly in $\mathscr X^{\alpha}(\mathbb T^2_{\frac{L}{r}})$ to $\frac{1}{r^2}\tilde\Xi_2^r$ where $\tilde\Xi^r_2$  is the rough distribution associate to $\tilde\xi_r$ and $\tilde c_\eps=-\mathbb E[\tilde\xi_r^\eps\circ\sigma(D)\tilde\xi_r^\eps]$. Of course this imply in particularly that 
$$
r^4(\xi(r\cdot))^{\frac{\eps}{r}}\circ\sigma(D)((\xi(r\cdot))^{\frac{\eps}{r}})+r^2\tilde c_\frac{\eps}{r}
$$
converge to $r^2\tilde \Xi^r_2$. To handle the right side of the equation~\eqref{eq:scaling mollification} we start by observing that 
$$
\Lambda_n\big(r^2(\xi(r\cdot))^{\frac{\eps}{r}},r^4(\xi(r\cdot))^{\frac{\eps}{r}}\circ\sigma(D)((\xi(r\cdot))^{\frac{\eps}{r}})+r^2c_\eps\big)=\Lambda_n\big(r\tilde\xi_r^\frac{\eps}{r},r^2\tilde\xi_r^{\frac{\eps}{r}}\circ\sigma(D)\tilde\xi_r^{\frac{\eps}{r}}+r^2\tilde c_{\frac{\eps}{r}}\big)+r^2(c_\eps-\tilde c_\frac{\eps}{r}).
$$
At this point the continuity of the map $\Lambda_n$ imply that the first term appearing in the right hand side of this equation converge when $\eps$ goes  to $0$ toward $\Lambda_n(r\tilde\xi_r,r^2\tilde\Xi_2^{r})$. Now it remain to control the difference between the diverging constant  $c_\eps-\tilde c_\frac{\eps}{r}$ which can be written explicitly :
$$
c_\eps-\tilde c_\frac{\eps}{r}=\frac{1}{L^2}\sum_{\mathbb Z^2_{L}}\frac{|\theta(\eps|k|)|^2}{1+|k|^2}-\frac{r^2}{L^2}\sum_{\mathbb Z^2_{\frac{L}{r}}}\frac{|\theta(r^{-1}\eps|k|)|^2}{1+|k|^2}=\frac{1-r^2}{L^2}\sum_{\mathbb Z_L^2}\frac{|\theta(\eps|k|)|^2}{(1+|k|^2)(1+r^2|k|^2)}
$$
which by dominate convergence goes to  
$$
m_{r,L}=\frac{1-r^2}{L^2}\sum_{\mathbb Z_L^2}\frac{1}{(1+|k|^2)(1+r^2|k|^2)}<\infty
$$
when $\eps$ goes to zero. Therefore we can conclude that 
\begin{equation}\label{eq:scaling}
\Lambda_n(\Xi)=\frac{1}{r^2}(\Lambda_n(r\tilde\xi_r,\tilde\Xi_2^r)+r^2m_{r,L})
\end{equation}
Before proceeding with our computation let us observe that Riemann-sum approximation gives the following inequality   
\begin{equation*}\label{eq:constant}
\begin{split}
m_{L,r}\lesssim L^{-2}(1-r^2)(1+\sum_{n\geq 1}\frac{n}{(1+\frac{n}{L}^2)(1+r^2\frac{n}{L}^2)})&\lesssim \frac{1-r^2}{L^2}+(1-r^2)\int_0^{+\infty}\frac{\rho d\rho}{(1+\rho^2)(1+r^2\rho^2)}
\\&= \frac{1-r^2}{L^2}+(1-r^2)\log(\frac{1}{r})
\end{split}
\end{equation*}
where we have assumed $r\in(0,1)$. Now let us come back to our original problem which is to bound the eigenvalue $\Lambda_n(\Xi)$ for that we will use the scaling property~\ref{eq:scaling} with $r=\frac{1}{\sqrt{\log L}}$. Namely we have that 
$$
\frac{1}{\log( L)}\Lambda_n(\Xi)=\Lambda_n(\tilde\Xi^r)+\frac{1}{\log L}m_{L,r} 
$$
where 
$$
\Xi^r=(\frac{\tilde\xi_r}{\sqrt{\log L}},\frac{\tilde \Xi_2^r}{\log L})
$$
and we recall that $\tilde\xi_r$ is a white noise on $\mathbb T^2_{L\sqrt{\log L}}$.  
 Using the growth estimate for the eigenvalue given in the Proposition~\ref{prop: continuity-eig} allow us to get that 
\begin{equation}\label{eq:bound}
\begin{split}
&\mathbb E[|\frac{1}{\log( L)}\Lambda_n(\Xi)|^p]\lesssim \big(\frac{1}{\log L} m_{L,\frac{1}{\sqrt{\log L}}}\big)^p+\mathbb E[|\Lambda_n(\Xi^r)|^p]\\&
\lesssim1+\Lambda_n(0)^p+n^p\mathbb E\left[\|\Xi^r\|_{\mathscr X^\alpha(\mathbb T^2_{L\sqrt{\log L}})}(1+\|\Xi^r\|_{\mathscr X^\alpha(\mathbb T^2_{L\sqrt{\log L}})})^{pM}\right]\left(1+n^{\frac{2\gamma-\alpha}{\alpha+2}}+(1+\Lambda_n(0))^{2\gamma}\right)^{2p}
\end{split}
\end{equation}
for all $p>1$ and where we have used that $\frac{1}{\log L} m_{L,\frac{1}{\sqrt{\log L}}}\lesssim 1$. On the other hand  Lemma~\ref{lemma:white noise growth} give us the following bound: 
$$
\mathbb E[(\|\tilde\xi\|^2_{\mathscr C^\alpha(\mathbb T^2_{L\sqrt{\log L}})}+\|\tilde\Xi_2\|_{\mathscr C^{2\alpha+2}(\mathbb T_{L\sqrt{\log L}}^2)})^p]\lesssim (\log(L\sqrt{\log L})^p\lesssim (\log(L))^p. 
$$
which yield that
$$
\sup_{L}\mathbb E\|\Xi^r\|^p_{\mathscr X^\alpha}<+\infty. 
$$
 Therefore the bound~\ref{eq:bound} allow us to conclude that:  
$$
\sup_{L}\mathbb E|\frac{1}{\log( L)}\Lambda_n(\Xi)|^p<+\infty
$$
.

\subsection{Tail estimate for the minimal eigenvalue}
In this section we will assume that $L=1$ and we are interested to study the tail estimate for the minimal eigenvalue $\Lambda_1$, namely we have the following result 
\begin{Proposition}
\label{prop:tail estimates}
Let $\Lambda_1$ the first eigenvalue of the operator $\mathscr H(\Xi)$ where $\Xi$ is the rough distribution associated. Then there exist $C_1,C_2>0$  and  such that :
$$
e^{C_1x}\leq \mathbb P(\Lambda_1\leq x) \leq e^{C_2x}
$$
when $x\to-\infty$
\end{Proposition}
\begin{proof}
\subsubsection*{Upper bound:}
As pointed out previously the following relation 
$$
r^2\Lambda_1(\Xi)\equiv\Lambda_1(\tilde\Xi^r)+r^2m_{1,r}
$$
 hold in law for every $r>0$, where $\tilde \Xi^r=(r\tilde \xi,r^2\tilde \Xi_2)$ with $\tilde\xi$ is a white noise on $\mathbb T^2_{\frac{1}{r}}$ and $\tilde \Xi_2$ the associate rough distribution. Of course this relation obliviously imply the following equality  
$$
\mathbb P(r^2\Lambda_1(\Xi)\leq-1)=\mathbb P(\Lambda_1(\tilde\Xi^r))\leq-1-r^2m_{1,r})
$$  
Since $r^2m_{1,r}\to^{r\to0} 0$ is easy to see that 
$$
\mathbb P(\Lambda_1(\tilde\Xi^r)\leq -\frac{3}{2})\leq\mathbb P(\Lambda_1(\tilde\Xi^r))\leq-1-r^2m_{1,r})\leq\mathbb P(\Lambda_1(\tilde\Xi^r))\leq-\frac{1}{2})
$$
for $r$ small enough. Using the continuity estimate  for the ground state given in the Proposition~\ref{prop: continuity-eig} we can see that the event $\left\{(\Lambda_1(\tilde\Xi^r))\leq-\frac{1}{2}\right\}$ is contained in $\left\{\|\tilde\Xi^r\|_{\mathscr X^\alpha}(1+\|\tilde \Xi^r\|_{\mathscr X^{\alpha}})^M\geq C\right\}$ for a deterministic constant $C>0$ and for $\alpha<-1$. Thus we have 
$$
\mathbb P(r^2\Lambda_1(\Xi)\leq-1)\leq\mathbb P(\|\tilde\Xi^r\|_{\mathscr X^\alpha}(1+\|\tilde \Xi^r\|_{\mathscr X^{\alpha}})^M\geq C)
$$
splitting the right hand side according of this equation to the event $\{\|\Xi\|_{\mathscr X^\alpha}\geq 1\}$ and it complementary we can finally bound our probability by 
$$
\mathbb P(r^2\Lambda_1(\Xi)\leq-\frac{1}{2})\leq\mathbb P(\|\tilde\Xi^r\|_{\mathscr X^\alpha}\geq 1)+\mathbb P(\|\tilde\Xi^r\|_{\mathscr X^{\alpha}}\geq C2^{-M})
$$
Now it suffice to observe that 
$$
\mathbb P(\|\Xi\|_{\mathscr X^\alpha}\geq 1)\leq \mathbb P(r\|\tilde \xi^r\|_{\mathscr C^{\alpha}}\geq\frac{1}{2})+\mathbb P(r^2\|\tilde\Xi^r_2\|_{\mathscr C^{2\alpha+2}}\geq \frac{1}{2})\leq r^{-\theta} e^{\frac{-2\lambda}{r^2}}\sup_{r}r^{\theta}(\mathbb E[e^{\lambda \|\tilde\xi\|^2_{\mathscr C^{\alpha}}}]+\mathbb E[e^{\lambda \|\tilde\Xi_2\|_{\mathscr C^{2\alpha+2}}}])
$$
 for $\lambda>0$ small enough, where we have used the Markov inequality, then choosing $\theta$ according to the  Lemma~\ref{lemma:white noise growth} allow us to get the needed upper bound. The term $\mathbb P(\|\Xi\|_{\mathscr X^{\alpha}}\geq C2^{-M})$ can be treated in the same way and of course if take $x=-\frac{1}{r^2}$ this bound can be reformulated in the following way 
$$
\mathbb P(\Lambda_1(\Xi)\leq x)\lesssim x^{\theta}e^{2\lambda x}
$$ 
\subsection*{Lower bound:}
\label{sec:h}
Given $c<0$, a subset $S\subset\mathbb T_{\frac{1}{r}}^2$ which have size 1 (ie: $|S|=1$), $f$ a smooth function on $\mathbb T_L^2$ with support contained in $S$ and such that $\int_{S}f^2=1$, $b=-\|\nabla f\|^2_{L^2}+\frac{c}{2}$ and $h(x)=b\mathbb I_{S}$ where $\mathbb I_{S}$ is the characteristic function of the set $S$. Then is easy to see from the min-max principle that :
 $$
\Lambda_1(h,h\circ\sigma(D)h)\leq \|\nabla f\|^2_{L^2}+b\int_{S}f^2=\frac{c}{2}
$$
and before proceeding with proof let us remark that $b$ does not depend on $r$. The continuity bound of the eigenvalue gives us that:
$$
\Lambda_1(\tilde\Xi^r)\leq \frac{c}{2}+C(\|r\tilde\xi-h\|_{\mathscr C^{\alpha}}+\|r^2\tilde\Xi_2-h\circ\sigma(D)h\|_{\mathscr C^{2\alpha+2}})(1+\|\Xi\|_{\mathscr X^{\alpha}}+\|h\|^2_{L^{\infty}})^M
$$
for a deterministic constant $C$ (which depend only on $c$). Now if $(\|r\tilde\xi-h\|_{\mathscr C^{\alpha}}+\|r^2\tilde\Xi_2-h_N\circ\sigma(D)h_N\|_{\mathscr C^{2\alpha+2}})\leq \delta$ for a fixed $\delta$ which satisfy  $C\delta(\delta+1)^M\leq -\frac{c}{4}$. Then  
$$
\Lambda_1(\tilde\Xi^r)\leq\frac{c}{4}
$$
therefore if $c=-6$  the event $\{\|r\tilde\xi-h\|_{\mathscr C^{\alpha}}+\|r^2\tilde\Xi_2-h\circ\sigma(D)h\|_{\mathscr C^{2\alpha+2}})\leq\delta\}$ is contained in $\{\Lambda_0(\tilde\Xi^r)\leq-\frac{3}{2}\}$ and this immediately implies that 
\begin{equation}
\label{eq:l}
\mathbb P(\Lambda_0(\tilde\Xi^r)\leq-\frac{3}{2})\geq\mathbb P\left(\|r\tilde\xi-h\|_{\mathscr C^{\alpha}}\leq\frac{\delta}{2};\|r^2\tilde\Xi_2-h\circ\sigma(D)h\|_{\mathscr C^{2\alpha+2}}\leq\frac{\delta}{2}\right)
\end{equation}
To get a lower bound for the right hand side of this inequality we will use the Cameron-Martin theorem, indeed :
\begin{equation}
\label{eq:cm}
\begin{split}
\small
&\mathbb P\left(\|r\tilde\xi-h_N\|_{\mathscr C^{\alpha}}\leq\frac{\delta}{2};\|r^2\tilde\Xi_2-h\circ\sigma(D)h\|_{\mathscr C^{2\alpha+2}}\leq\frac{\delta}{2}\right)
=\exp(-r^{-2}\|h\|_{L^2})\mathbb E[\exp(r^{-1}\tilde\xi(h))\mathbb I_{A^r}]
\\&=\exp(-\frac{\|h\|_{L^2}}{2r^2})\mathbb  E[\exp(r^{-1}\tilde\xi(h))|A^r]\mathbb P(A^r)
\\&\geq\exp(-\frac{b^2}{2}r^{-2})\mathbb P(A^r)
\end{split}
\end{equation}

with $A^r=\left\{\omega;\|r\tilde\xi(\omega)\|_{\mathscr C^{\alpha}}\leq\frac{\delta}{2};\|r^2\tilde\Xi_2(\omega+r^{-1}h)-h\circ\sigma(D)h\|_{\mathscr C^{2\alpha+2}}\leq\frac{\delta}{2}\right\}$ and where we have used the Jensen inequality to obtain the last lower bound. Finally to get our lower bound it suffice to control the probability that the event $A^r$ happen for $r$ large enough. On the other hand we observe that : 
\begin{equation*}
\tilde\Xi_2(\omega+r^{-1}h)=\tilde\Xi_2(\omega)+r^{-1}\tilde\xi(\omega)\circ \sigma(D)h+r^{-1}\sigma(D)\tilde\xi(\omega)\circ h
\end{equation*}
which gives the following representation
$$
A^r=\left\{r\|\tilde\xi\|_{\mathscr C^{\alpha}}\leq\frac{\delta}{2};\|r^2\tilde\Xi_2+r\tilde\xi\circ\sigma(D) h+rh\circ\sigma(D)\tilde\xi\|_{\mathscr C^{2\alpha+2}}\leq\frac{\delta}{2}\right\}.
$$ 
At this point the Lemma~\ref{lemma:white noise growth} tell us that 
$$
r\|\tilde\xi\|_{\mathscr C^{\alpha}}+r^2\|\tilde\Xi_2\|_{\mathscr C^{2\alpha+2}}\to^{r\to0}
$$
almost surely. On the other side using the fact that $\|h\|_{L^\infty(\mathbb T^2_\frac{1}{r})}\leq b$ gives us that  
$$
r\|\tilde\xi\circ \sigma(D)h\|_{\mathscr C^{2\alpha+2}}\leq r\|\tilde\xi\circ\sigma(D)h\|_{H^{2\alpha+3}}\lesssim r\|\tilde\xi\|_{\mathscr C^{\alpha}}\|h\|_{L^\infty}\to^{r\to0}0
$$
almost surely. Same type of estimate show that $r\|\sigma(D)\tilde\xi\circ h\|_{\mathscr C^{2\alpha+2}}$ vanish when r goes to $0$. All this convergence imply in particularly that $\mathbb P(A^r)\to^{r\to0}1$ which combined with the two bound~\eqref{eq:cm} and \eqref{eq:l} allow us to get:
$$
\mathbb P(\Lambda_1(\tilde\Xi^r)\leq-\frac{3}{2})\geq \frac{1}{2}\exp(-\frac{b^2}{2}r^{-2})
$$   
for all $r$ small enough, which is the needed lower bound  due to the fact that $\mathbb P(\Lambda_0(\Xi)\leq -\frac{1}{r^2})\geq\mathbb P(\Lambda_0(\tilde\Xi^r)\leq-\frac{3}{2})$ . 
\end{proof}

\appendix

\section{Other useful results on Besov spaces and Bony paraproducts} \label{sec:appendix}

The following Besov embedding property is used in the proof of the convergence of the mollified white noise 
(see Theorem \ref{conv-white-noise}). 

\begin{Proposition}\label{Proposition:Bes-emb} 
Let $1\leq p_1\leq p_2\leq +\infty$ and $1\leq q_1\leq q_2\leq +\infty$. For all $s\in \mathbb R$, 
the space $\mathscr B_{p_1,q_1}^{s}$ is continuously embedded in $\mathscr B_{p_2,q_2}^{s-d(\frac{1}{p_1}-\frac{1}{p_2})}$. In particular, we have \begin{align*}||u||_{\mathscr{C}^{\alpha-\frac{d}{p}}}\lesssim||u||_{B_{p,p}^{\alpha}}.\end{align*}
 \end{Proposition}   
 We will need also the following useful extension of the Schauder estimate. 
 \begin{Proposition}\label{schauder-estimate-2}
Let $f\in H^\alpha$, $g\in\mathscr C^\beta$ with $\alpha\in(0,1),\beta\in \R$ and $\sigma:\R^2 \setminus \{0\} \to \R$ 
an infinitely differentiable function 
such that $|D^k\sigma(x)|\lesssim|x|^{-n-k}$

. Set 
\begin{align*}
\mathscr C(f,g):=\sigma(D)(f\prec g)-f\prec\sigma(D)g.
\end{align*}
Then, 
\begin{align*}
||\mathscr C(f,g)||_{H^{\alpha+\beta+n-\delta}}\lesssim||f||_{H^\alpha}||g||_{\mathscr{C}^\beta}\,. 
\end{align*}
for all $\delta>0$
\end{Proposition}
\begin{remark}
A proof of this Lemma is contained in the first version of~\cite{gip} (or in~\cite{cc13}) where $f$ is in the Besov-H\"older space $\mathscr C^{\alpha}$, however a slight modification show that the same result is true for Sobolev space.
\end{remark}
Before proving Proposition \ref{prop:commu}, 
we give an elementary commutation Lemma. 

\begin{lemma}\label{elementary-commutation-lemma}
Let $\alpha\in(0,1)$, $f\in H^\alpha$ and $g\in L^\infty$.  Then,
\begin{align*}
||\Delta_j(fg)-f\Delta_jg||_{L^2}\lesssim 2^{-j\alpha}||f||_{H^\alpha}||g||_{L^\infty}\,. 
\end{align*}
\end{lemma}

\begin{proof}
For $y\in \T^2$, we introduce the inverse Fourier transform 
\begin{align}\label{eq-theta}
\theta_j(y):= \sum_{k \in \Z^2} \rho(2^{-j} |k|) \exp(i 2 \pi \langle k, y\rangle)
\end{align}
of the function $\rho(2^{-j}\cdot)$ introduced to define the Littlewood-Paley blocks. 
For $x\in \T^2$, we have by definition 
\begin{equation}
\begin{split}
\left|\Delta_{j}(fg)(x)-f(x)\Delta_jg(x)\right|^2&=\left|\int_{\mathbb T^2}\theta_j(x-y)g(y)(f(y)-f(x))d y\right|^2
\\&\lesssim||g||^2_{L^\infty}
\left( \int_{\mathbb T^2}|\theta_j(x-y)|^2|x-y|^{2\alpha+2} dy \right) \int_{\mathbb T^2}\frac{|f(y)-f(x)|^2}{|y-x|^{2\alpha+2}} d y\,. 
\end{split}
\end{equation}
From \eqref{eq-theta}, we see that $\theta_j(y)$ concentrates in a ball of radius 
$2^{-j}$ and is of order $2^{2j}$ for $j$ large. Therefore, we deduce that  
\begin{align*}
\int_{\mathbb T^2}|\theta_j(x-y)|^2|x-y|^{2\alpha+2} dy \lesssim 2^{-2 j  \alpha}\,. 
\end{align*}
We eventually obtain  
\begin{align*}
||\Delta_j(fg)-f\Delta_jg||_{L^2}\lesssim 2^{-j\alpha}||g||_{L^\infty}\left(\int_{\mathbb T^2\times\mathbb T^2}\frac{|f(y)-f(x)|^2}{|y-x|^{2\alpha+2}} dxdy\right)^{\frac{1}{2}}
\end{align*}
which permits us to conclude thanks to the equivalent definition of the Sobolev space $H^\alpha$.
\end{proof}

We now give a simple consequence of the previous Lemma.   
\begin{lemma}
Let $\alpha\in(0,1)$, $f\in H^\alpha$ and $g\in\mathscr C^\beta$ and set 
\begin{align*}
R_j(f,g):=\Delta_j(f\prec g)-f\Delta_j g\,. 
\end{align*}
Then
\begin{align*}
||R_j(f,g)||_{L^2}\lesssim 2^{-j(\alpha+\beta)}||f||_{H^\alpha}||g||_{\mathscr C^\beta}\,. 
\end{align*}
\end{lemma}

\begin{proof}
We have by definition that 
$$
\Delta_i(f\prec g)=\sum_{j;j\sim i}\Delta_i(f\prec\Delta_j g)=\sum_{j;j\sim i}f \Delta_j\Delta_ig+\sum_{j;j\sim i}\Delta_i(f\Delta_jg)-f\Delta_i\Delta_jg-\sum_{j;j\sim i}\Delta_i(f\succ\Delta_jg+f\circ\Delta_j g)
$$
So that the first sum over $g$ can be chosen such that $\sum_{j,j\sim i}\Delta_i\Delta_j g=\Delta_ig$. Then the Lemma~\ref{elementary-commutation-lemma} take care of the second sum of this equation and the paraproduct estimate gives the needed bound for the the last term.
\end{proof}

{\it Proof of Proposition \ref{prop:commu}}.
Let us write that 
\begin{equation}
\begin{split}
(f\prec g)\circ h&=\sum_{|i-j|\leq1,k}1_{k\lesssim i}\Delta_i(\Delta_kf\prec g)\Delta_jh=\sum_{i\sim j,k}1_{k\lesssim i}\Delta_kf\Delta_ig\Delta_jh+\sum_{i\sim j,k\lesssim i}R_{i}(\Delta_kf,g)\Delta_j h
\\&=f(g\circ h)+\sum_{i\sim j,i\leq k-N}\Delta_kf\Delta_ig\Delta_jh+\sum_{i\sim j,k\lesssim i}R_{i}(\Delta_kf,g)\Delta_j h
\end{split}
\end{equation}
Now let us remark that for fixed $k$ the sum $\sum_{i\sim j}1_{i\leq k-N}\Delta_kf\Delta_ig\Delta_jh$ is supported in a ball $2^k\mathcal B$ then is suffice to setimate the $L^2$ norm 
$$
2^{k\alpha}\left|\left|\sum_{i\sim j}1_{i\leq k-N}\Delta_kf\Delta_ig\Delta_jh\right|\right|_{L^2}\lesssim (2^{k\alpha}||\Delta_kf||_{\alpha})||g||_{\beta}||h||_{\gamma}\sum_{i\sim j;i\leq k-N}2^{-i(\beta+\gamma)}
$$
therefore using the fact that $\beta+\gamma<0$ we obtain the needed bound for this term. Now we remark that for fixed $j$ the sum $\sum_{i\sim j,k\lesssim i}R_{i}(\Delta_kf,g)\Delta_j h$ is localized in a ball of size $2^j$ then estimating the sum 
$$
\left|\left|\sum_{i\sim j,k\lesssim i}R_{i}(\Delta_kf,g)\Delta_j h\right|\right|_{L^2}\lesssim||\sum_{k\lesssim i}\Delta_kf||_\alpha||g||_{\beta}||h||_{\gamma}2^{-j\gamma}\sum_{j\sim i}2^{-i(\alpha+\beta)}
$$
which gives the needed estimates. 
\qed

We end the appendix by recalling a version of the Rellich-Kondrachov Theorem
\begin{lemma}
\label{lemma:rellich}
Let $\gamma<\beta$. Then, the injection $i:H^\beta(\mathbb T_L^2)\to H^\gamma(\mathbb T_L^2)$ is compact.
\end{lemma}
\begin{proof}
Let $(f_n)_{n\in\mathbb N}$ be a bounded 
sequence in $H^\beta$ and set $M:=\sup_n||f_n||_{H^\beta}$. It is 
well known that there exists a subsequence of 
$(f_{n})$  (still denoted for simplicity by $f_n$) which converges in $\mathscr S'(\mathbb T_L^2)$ 
to some $f\in \mathscr S'(\mathbb T_L^2)$, or equivalently such that the Fourier coefficients converge, 
i.e.    
$\lim_n\hat f_n(k)=\hat f(k)$ for all $k\in\mathbb Z_L^2$. 
It is easy to prove that $f\in H^{\beta}$ using the Fatou Lemma
\begin{align*}
\sum_{k\in\mathbb Z_L^2}(1+|k|)^{2\beta}
|\hat f(k)|^2\leq\liminf_{n}\sum_{k\in\mathbb Z_L^2}(1+|k|)^{2\beta}|\hat f_n(k)|^2\leq M^2.
\end{align*}
Let us now prove that $(f_n)$ converges to $f$ in the space $H^{\gamma}(\T_L^2)$.  
We have 
\begin{align*}
||f_n-f||^2_{H^\gamma}&\leq
\sum_{|k|\leq N}(1+|k|)^{2\gamma}|\hat f_n(k)-\hat f(k)|^2+
\sum_{|k|> N}(1+|k|)^{2(\gamma-\beta)} (1+|k|)^{2\beta} |\hat f_n(k)-\hat f(k)|^2\\
&\leq 
\sum_{|k|\leq N}(1+|k|)^{2\gamma}|\hat f_n(k)-\hat f(k)|^2+2 N^{2(\gamma-\beta)} M^2\,.
\end{align*}
The convergence in $H^\gamma$ follows from this latter inequality. 
\end{proof} 
Now let us end by giving a more simplest description of the space $\mathscr X^\alpha$
\begin{lemma}\label{lemma:rd}
Given $\alpha<-1$ and let denote by $\mathscr C^{0,\alpha}$ (respectively  $\mathscr C^{0,2\alpha+2}$) the closure of the space of infinitely differentiable function in the space $\mathscr C^{\alpha}$ (respectively $\mathscr C^{2\alpha+2}$) then the following identity set 
$$
\mathscr X^{\alpha}=\mathscr C^{0,\alpha}\times\mathscr C^{0,2\alpha+2}
$$ 
\end{lemma} 
\begin{proof}
To prove our equality set is sufficient to show that $\mathscr X^{\alpha}$ contain the space $\{0\}\times\mathscr C^{\infty}$. Let $X^{N}(x)=2^{N}cos(2^{N}\pi\langle z,x\rangle)$ for $x\in [0,1]^2$. It was proved in~\cite{cf} that 
\begin{enumerate}
\item $\|X^{N}\|_{\mathscr C^{\alpha}}\to^{N\to+\infty}0$
\item $\|X^{N}\circ\sigma(D)X^{N}+1\|_{\mathscr C^{2\alpha+2}}\to^{n\to+\infty}0$
\end{enumerate}
for all $\alpha<-1$. Let $V\in\mathscr C^{\infty}(\mathbb T^2)$ and $X^{N,V}=VX^{N}$, then is easy to see that 
$$
\|X^{N,V}\|_{\mathscr C^{\alpha}}\lesssim\|V\|_{\mathscr C^{\beta}}\|X^N\|_{\mathscr C^{\alpha}}\to^{N\to+\infty}0
$$ 
where $\beta>-\alpha$. Now we claim that $X^{N,V}\circ\sigma(D)X^{N,V}\to-V^2$ in $\mathscr C^{2\alpha+2}$. Indeed from the Bony estimate we can see that $\|V\circ X^{N}\|_{\mathscr C^{\alpha+\beta}}+\|V\succ X^N\|_{\mathscr C^{\alpha+\beta}}\to 0$ for all $\beta>-\alpha$ which in particularly imply that 
$$
(X^N\circ V+X^N\succ V)\circ \sigma(D)X^{N,V}\to^{N\to+\infty} 0
$$
in $\mathscr C^{2\alpha+2}$. On the other side Schauder estimate allow us to see that 
$$
X^{N,V}\circ\sigma(D)(X^N\circ V+X^N\succ V)\to^{N\to+\infty}0
$$  
in $\mathscr C^{2\alpha+2}$. Then we can conclude that 
$$
\|X^{N,V}\circ\sigma(D)X^{N,V}-(V\prec X^{N})\circ\sigma(D)(V\prec X^N)\|_{\mathscr C^{2\alpha+2}}\to^{N\to+\infty}0
$$
Moreover from the Proposition~\ref{schauder-estimate-2} is easy to show that 
$$
\|\sigma(D)(V\prec X^N)-V\prec\sigma(D)X^N\|_{\mathscr C^{2\alpha+2}}\lesssim \|V\|_{\mathscr C^\beta}\|X^N\|_{\mathscr C^\alpha}\to^{N\to+\infty} 0
$$
Therefore the proof of our convergence result is reduced to the study of  
$$
(V\prec X^N)\circ(V\prec\sigma(D)X^N)
$$ 
which can be handled by using the commutation Lemma~\ref{prop:commu}. Indeed we have the following expansion
\begin{equation*}
\begin{split}
(V\prec X^N)\circ(V\prec\sigma(D)X^N)&=V^2(X^{N+1}\circ\sigma(D)X^N)+\mathscr R(V,X^N,V\prec \sigma(D)X^N)\\&+V\mathscr R(V,X^N,\sigma(D)X^N)
\end{split}
\end{equation*}
which converge to $-V^2$ in the space $\mathscr C^{2\alpha+2}$. Then in particularly we have proved that $(0,c-V^2)\in\mathscr X^\alpha$ for every smooth function $V$ and every $c\in\mathbb R$ which finishes the proof.
\end{proof}


%
%

%
%


\begin{bibdiv}
 \begin{biblist}



\bib{laure-romain}{article}{
author={R. Allez},
author={ L. Dumaz}, 
title={ Tracy-Widom at high temperature}, 
journal={J. Stat. Phys, 6, 1146-1183 (2014)}, 
}




\bib{laure-sine}{article}{
author={ R. Allez}, 
author={ L. Dumaz.}, 
title={\it From Sine kernel to Poisson statistics},
Journal={Electronic journal of probability},
}


\bib{BCD-bk}{book}{
author={H. Bahouri},
author={J-Y. Chemin},
author={R. Danchin},
title={Fourier analysis and nonlinear partial differential equations.},
 series={Grundlehren der Mathematischen Wissenschaften [Fundamental Principles of Mathematical Sciences].},
 volume={343},
 publisher={Springer},
 place={Heidelberg},
 date={2011}
 }


\bib{bony}{article}{
author={J.-M. Bony}, 
title= {Calcul symbolique et propagation des singularit�s
pour les �quations aux d�riv�es partielles
non lin�aires}, 
journal={Ann. Sci. �cole Norm. Sup.}, 
number={14}, 
year={1981},
pages={209�246},
}

\bib{ism}{unpublished}{
author={I.Bailleul},
author={ F.BERNICOT},
title={HEAT SEMIGROUP AND SINGULAR PDES},
journal={arXiv:1501.06822.},
year={2015},
}


\bib{cc13}{unpublished}{
author={R.Catellier},
author={K.Chouk},
title={ Paracontrolled Distributions and the 3-dimensional Stochastic Quantization Equation}
journal={arXiv:1310.6869}
}
\bib{cf}{unpublished}{
author={K.Chouk},
author={P.K.Friz},
title={Support theorem for a singular semilinear stochastic partial differential equation},
journal={{arXiv}:1409.4250},
}


\bib{laure}{article}{
author={ L. Dumaz},
author={B. Vir\'ag}, 
title={The right tail exponent of the Tracy-Widom-beta distribution},
Journal={Ann. Inst. H. Poincar\'e Probab. Statist. 4, 915-933, (2013).},
}
\bib{davie}{article}{
author={E.B.Davies},
title={THE RATE OF RESOLVENT AND SEMIGROUP CONVERGENCE},
journal={quart.J.math.Oxford},
year={1984},
}

\bib{fukushima}{article}{
author={M. Fukushima} 
author={S. Nakao}
journal={Z. Wahrscheinlichkeitstheorie verw. Gebiete, 267-274 (1977)}
}

\bib{lloyd}{article}{
author={H. L. Frisch},
author={S. P. Lloyd},
title={Electron Levels in a One-Dimensional Random Lattice}, 
journal={Phys. Rev.1175-1189 (1960)},
}

\bib{kg}{article}{
author={J.G\"artner}
author={W.K\"onig}
author={Stanislav Molchanov}
titlel={Geometric characterization of intermittency in the parabolic Anderson model}
journal={Annal of Probability}
year={2007}
Volume={35}
number={2}
}

\bib{kg1}{article}{
author={J.G\"artner}
author={W.K\"onig}
title={The Parabolic Anderson Model}
journal={Interacting Stochastic Systems}
year={2005}

}

\bib{gip}{article}{
	title = {Paracontrolled distributions and singular PDEs, http://arxiv.org/abs/1210.2684 (v3 from July 2014)},
	url = {http://arxiv.org/abs/1210.2684},
	journal = {Forum of Mathematics, Pi},
	author = {Gubinelli, M.},
	author={Imkeller, P.},
	author={Perkowski, N.},
	year = {2015},
}


	
\bib{hairer}{article}{	
year={2014},
issn={0020-9910},
journal={Inventiones mathematicae},
doi={10.1007/s00222-014-0505-4},
title={A theory of regularity structures},
url={http://dx.doi.org/10.1007/s00222-014-0505-4},
publisher={Springer Berlin Heidelberg},
keywords={60H15; 81S20; 82C28},
author={Hairer, M.},
pages={1-236},
language={English}
}


\bib{hairer_weber_LDP}{article}{
	title = {Large deviations for white-noise driven, nonlinear stochastic PDEs in two and three dimensions},
	url = {http://arxiv.org/abs/1404.5863},
	journal = {Annales de la Faculté des Sciences de Toulouse},
	author = {Hairer, M.},
	author = {Weber, H.},
	year = {2014},
	}

\bib{HL}{unpublished}{
author={M.Hairer},
author={C.Labb\'e},
title={Multiplicative stochastic heat equations on the whole space},
journal={arXiv:1504.07162}
year={2015}
}


\bib{halperin}{article}{
author={B.I.Halperin},
title={Green's Functions for a Particle in a One-Dimensional Random Potential},
journal={Phys. Rev.A 104-A117 (1965)},
}

\bib{Hu}{unpublished}{
author={Y.Hu},
author={D.Nualart},
author={J.Huang},
title={On the intermittency front of stochastic heat equation driven by colored noises},
journal={arXiv:1506.04670}
}


\bib{HTN}{unpublished}{
author={Y.Hu},
author={D.Nualart},
author={J.Huang},
author={S.Tindel},
title={STOCHASTIC HEAT EQUATIONS WITH GENERAL MULTIPLICATIVE
GAUSSIAN NOISES: H\"OLDER CONTINUITY AND INTERMITTENCY},
journal={Preprint~arXiv:1402.2618},
}


\bib{janson}{book}{,
	title = {Gaussian Hilbert Spaces},
	isbn = {9780521561280},
	language = {en},
	publisher = {Cambridge University Press},
	author = {Janson, S.},
	year = {1997},
}

\bib{wk}{unpublished}{
author={W.K\"onig},
title={The Parabolic Anderson Model},
journal={Preprint},
url={https://www.wias-berlin.de/people/koenig/www/PAMsurveyBook.pdf},
}


\bib{mckean}{article}{
author={H. P. McKean},
journal={A Limit Law for the Ground State of Hill's Equation},
journal={J. Stat. Phys 1227 (1994)},
}

\bib{dp}{unpublished}{
author={D.J.Pr\"omel}
author={M.Trabs}
title={Rough differential equations on Besov spaces}
journal={arXiv:1506.03252}
}


\bib{virag-1}{article}{ 
author={J. A. Ram\'irez}, 
author={B. Rider},
author={ B. Vir\'ag},
title={Beta ensembles, stochastic Airy spectrum, and a diffusion},
journal={J. Amer. Math. Soc 919-944 (2011)},
}


\bib{valko}{article}{
author={B. Valk\'o},
author={ B. Vir\'ag},
title={Continuum limits of random matrices and the Brownian carousel},
journal={Invent. math.463-508 (2009)},
}

\end{biblist}
\end{bibdiv}
\end{document}